\documentclass{article}

\usepackage[utf8]{inputenc} 
\usepackage[T1]{fontenc}    
\usepackage{url}            
\usepackage{booktabs}       
\usepackage{amsfonts}       
\usepackage{nicefrac}       
\usepackage{microtype}      
\usepackage{graphicx}
\usepackage{doi}
\usepackage{orcidlink}
\usepackage{tablefootnote}
\usepackage{amsthm}
\usepackage[left=1.5cm, right=2cm]{geometry}

\usepackage{bm}
\usepackage{algorithm}
\usepackage{algpseudocode}
\usepackage{color}
\usepackage{systeme}
\usepackage{comment}
\usepackage{scalerel}
\usepackage{graphicx}
\usepackage{grffile}
\usepackage{multirow}
\usepackage{pdflscape}
\usepackage{amsmath,amssymb}
\usepackage{cases}
\usepackage{bm}
\usepackage{booktabs}
\usepackage{tabularx}   

\usepackage[capitalise,noabbrev]{cleveref}
\usepackage{algorithm}
\usepackage{algorithmicx}
\usepackage{algpseudocode}
\usepackage{mathtools}
\usepackage{comment}
\usepackage{tikz}
\newlength{\figureheight}
\newlength{\figurewidth}
\usepackage{tikzsymbols}
\usepackage{pgfplots}
\pgfplotsset{compat=newest} 
\usetikzlibrary{plotmarks}
\usetikzlibrary{arrows.meta}
\usetikzlibrary{positioning}
\usepgfplotslibrary{patchplots}
\usepackage{grffile}
\usepackage{multirow}
\usepackage{pdflscape}
 \usepackage{fontawesome}
\pgfplotsset{every axis/.append style={
                    label style={font=\scriptsize},
                    tick label style={font=\scriptsize},
                    legend style={font=\scriptsize}
                    }}
\pgfplotsset{compat=newest}
\pgfplotsset{plot coordinates/math parser=false}
\pgfplotsset{grid style={dotted,gray}}
\usepgfplotslibrary{groupplots}
\pgfdeclarelayer{background}
\pgfdeclarelayer{foreground}
\pgfsetlayers{background,main,foreground}
\usepackage{subcaption}
\usepackage{textcomp}
\usepackage{xcolor}
\usepackage{dsfont}
\usepackage{markdown}

\usepackage{hyperref}       

\newcommand\norm[1]{\left\lVert#1\right\rVert}

\newtheorem{thrm}{Theorem}[section]
\newtheorem{prpstn}[thrm]{Proposition}
\newtheorem{rmrk}[thrm]{Remark}
\newtheorem{asmp}{Assumption}

\definecolor{darkorange25512714}{RGB}{255,127,14}
\definecolor{forestgreen4416044}{RGB}{44,160,44}
\definecolor{steelblue31119180}{RGB}{31,119,180}
\definecolor{crimson2143940}{RGB}{214,39,40}

\makeatletter
\patchcmd{\ALG@step}{\addtocounter{ALG@line}{1}}{\refstepcounter{ALG@line}}{}{}
\newcommand{\ALG@lineautorefname}{Line}
\makeatother

\usepackage{todonotes}

\begin{document}

\title{Optimal control for a class of linear evolution equations with transformed proper orthogonal decomposition}

\author{
    Tobias Breiten\,\orcidlink{0000-0002-9815-4897}\thanks{\texttt{\{tobias.breiten@,burela@tnt.,pschulze@math.\}tu-berlin.de}}\,\,\thanks{Institute of Mathematics, Technische  Universit\"at Berlin, Germany},
    Shubhaditya Burela\,\orcidlink{0009-0003-4442-5297}\footnotemark[1]\,\,\footnotemark[2],
    Philipp Schulze\,\orcidlink{0000-0002-7299-4628}\footnotemark[1]\,\,\footnotemark[2]\,\,\thanks{Institute of Mathematics, University of Potsdam, Germany}
}

\providecommand{\keywords}[1]
{
  \small	
  \textbf{Keywords:} #1\\
}
\providecommand{\ams}[1]
{
  \small	
  \textbf{AMS subject classifications:} #1
}

\maketitle

\begin{abstract}
Solving optimal control problems for high-dimensional partial differential equations (PDEs) is computationally expensive. 
As a workaround, reduced-order models (ROMs) are often employed, which replace the full PDE system and alleviate the computational burden while solving optimal control problems. 
However, when the underlying PDE exhibits transport-dominated behavior, standard reduced-order modeling techniques like proper orthogonal decomposition (POD) that rely on linear approximation spaces may require a large number of modes due to the slowly decaying Kolmogorov $n$-width. 
To overcome this challenge, we focus on the use of transformed POD as a reduced-order model, which is particularly effective in capturing low-dimensional representations of high-fidelity transport-dominated phenomena.
In this work, a reduced-order model is constructed first, followed by the optimization of the reduced system.
For linear systems generated by strongly continuous groups, we prove existence and uniqueness of solutions for the ROM under an explicit smallness condition on the control, as well as the existence of an optimal control for sufficiently large regularization. 
The numerical experiments concern the 1D linear advection equation with a subset of the general transformation groups called the shift transformations.
Additionally, we compare the computational performance of the shifted POD (sPOD) method against the standard POD.
\end{abstract}
\\\noindent\ams{35L02, 49M41, 49K20, 35Q35}

\section{Introduction}\label{sec:Introduction}

In this work, we discuss the use of reduced-order methods for a class of optimal control problems constrained by an abstract linear control system of the form
\begin{equation}\label{eq_def:PDE_continuous}
\begin{aligned}
    \dot{y}(t) &= \mathcal{A} y(t) + \mathcal{B}u(t),  &  \quad t \in (0,T],\\
    y(0) &= y_0,
\end{aligned}
\end{equation}
where $\mathcal{A}\colon \mathcal{D}(\mathcal{A}) \subset H$ is a densely defined and closed linear operator generating a strongly continuous group on a Hilbert space $H.$ The control operator $\mathcal{B}\in\mathcal{L}(U,H)$  is assumed to take the form
\begin{equation}
    \label{eq:control_operator}
    \mathcal{B}u := \sum^{m}_{k=1} b_ku_k,
\end{equation}
with given control shape functions $b_k \in \mathcal{D}(\mathcal{A})$ and $U = \mathbb{R}^{m}$. Associated with \eqref{eq_def:PDE_continuous}, we consider a standard quadratic tracking-type cost functional 
\begin{equation}\label{eq_def:FOM_costFunc_continuous}
    \mathcal{J}(y, u) = \frac{1}{2} \int^{T}_0 \norm{y(t) - y_\mathrm{d}(t)}^2_H \mathrm{d}t + \frac{\mu}{2} \int_0^{T}\norm{u(t)}^2_{U}\,\mathrm{d}t\, ,
\end{equation} 
where $\mu > 0$ and the desired state satisfies $y_{\mathrm{d}} \in L^2(0,T;H)$. For our numerical studies later on, we consider the 1D advection equation
\begin{equation}\label{eq:lin_adv_eq}
\begin{aligned}
    \partial_t y(t,x) + v \partial_x y(t,x) &= u(t,x) && \text{in } (0,T]\times (0,l), \\
    y(t,l)&= y(t,0) && \text{in } (0,T), \\
    y(0,x) &= y_0(x) && \text{in } (0,l),
\end{aligned}
\end{equation}
where $v\in\mathbb{R}$ denotes a constant velocity and $u$ is a control function to be optimized. Note that \eqref{eq:lin_adv_eq} is a particular instance of \eqref{eq_def:PDE_continuous} with $\mathcal{A}\colon \mathcal{D}(\mathcal{A}) \subset L^2(0,l) \rightarrow L^2(0,l), \mathcal{A}=-v \tfrac{\mathrm{d}}{\mathrm{d} x}$ with $\mathcal{D}(\mathcal{A})=H^1_\mathrm{per}(0,l):= \{f \in H^1(0, l) \ | \ f(0) = f(l)\}$, see, e.g., \cite[Examples 2.6.12 \& 2.7.12, Thm.~3.8.6, Cor.~2.1.8, Prop.~2.3.1]{tucsnak_observation_2009}.

Linear quadratic optimal control problems as in \eqref{eq_def:PDE_continuous}-\eqref{eq_def:FOM_costFunc_continuous} are well understood, and detailed treatises can be found in, e.g., the monographs \cite{hinze_optimization_2009,troltzsch_optimal_2010}. In particular, for the setup considered here, we may argue as follows. Since $\mathcal{A}$ generates a strongly continuous (semi-)group on $H$ and $\mathcal{B}\in\mathcal L(U,H)$, for every $u\in L^2(0,T;U)$ the state equation \eqref{eq_def:PDE_continuous} admits a unique mild solution $y\in C([0,T];H)$, and the associated control-to-state map is affine and continuous \cite[Proposition 4.2.5]{tucsnak_observation_2009}.
Consequently, the reduced cost functional is weakly lower semicontinuous, coercive and strictly convex due to $\mu>0$, and therefore admits a unique minimizer. The standard adjoint calculus for linear-quadratic optimal control problems, cf.~\cite[Sections 1.5-1.7]{hinze_optimization_2009}, then yields the adjoint equation 
\begin{equation}\label{eq_def:PDE_adjoint_continuous}
\begin{aligned}
    -\dot{\lambda}(t) &= \mathcal{A}^* \lambda(t) +y(t) - y_\mathrm{d}(t), \quad t \in [0,T),\\
    \lambda(T) &= 0,
\end{aligned}
\end{equation}
and the optimality condition $\mu u(t) + \mathcal{B}^* \lambda(t) = 0$. Note that in the case of \eqref{eq:lin_adv_eq} we even have that $\mathcal{A}^*= -\mathcal{A}=v\frac{\mathrm{d}}{\mathrm{d} x}$ with $\mathcal{D}(\mathcal{A}^*) = H^1_\mathrm{per}(0,l)$, see \cite[Section 2.8]{tucsnak_observation_2009}. Altogether, the first-order necessary optimality conditions read
\begin{subnumcases}{\mathrm{OC}_{\scaleto{\mathrm{FOM}}{4pt}} := }
    \dot{y}(t) = \mathcal{A} y(t) + \mathcal{B} u(t),\quad t \in (0,T]\label{eq_def:OC_FOM_1},  \\
    y(0) = y_0\label{eq_def:OC_FOM_2}, \\[1ex]
    -\dot{\lambda}(t) = \mathcal{A}^* \lambda(t) + y(t) - y_\mathrm{d}(t),\quad t \in [0,T)\label{eq_def:OC_FOM_3},  \\
    \lambda(T) = 0\label{eq_def:OC_FOM_4},\\[1ex]
    \mu u(t) + \mathcal{B}^* \lambda(t)  = 0,\quad t \in [0,T]. \label{eq_def:OC_FOM_5}
\end{subnumcases} 
The optimal control problem consisting of \eqref{eq_def:PDE_continuous} and \eqref{eq_def:FOM_costFunc_continuous} is rather standard and there exists a considerable amount of work dealing with more general control problems for hyperbolic PDEs. Let us for instance refer to \cite{bressan_optimality_2007} and \cite{ulbrich_adjoint-based_2003} where  nonlinear hyperbolic systems are studied. In \cite{borzi_optimal_2003}, the authors discuss a bilinear optimal control problem for optical flow applications. Let us also point to \cite{Lasiecka_Triggiani_2000} where boundary control problems for hyperbolic systems are considered from a classical control-theoretic viewpoint. More recent research, such as \cite{chen_optimal_2021}, investigates the optimal control of hyperbolic PDEs within the framework of optimal transport in Wasserstein spaces.

PDE-constrained optimal control problems are often difficult to solve numerically, as they lead to large-scale optimization challenges, particularly in higher spatial dimensions and/or in cases where a fine discretization is required. For this reason, one is interested in mathematical methods that reduce the complexity of the underlying dynamical system and speed up its simulation by using ROMs. Over the past three decades, significant progress has been made in developing efficient model order reduction (MOR) methods for linear quadratic (LQ) optimal control problems, with a particular focus on projection-based methods, see, e.g., \cite{benner_model_2017, quarteroni_reduced_2016}. These methods rely on projecting the dynamical system onto subspaces consisting of basis elements that contain features of the expected solution. The reduced basis (RB) \cite{benner_model_2017} method is one popular approach, constructing a low-dimensional reduced basis space spanned by snapshots of the original solution. The reduced basis approximation is then efficiently obtained by Galerkin projection onto this space. Early research on RB methods for optimal control, such as \cite{ito_reduced_2001, quarteroni_reduced_nodate}, focused on flow problems, while later studies \cite{negri_reduced_2013, karcher_certified_2018} extended RB methods to the optimal control of parametrized problems. Another MOR technique, balanced truncation (BT), transforms the state-space system into a balanced form, where the controllability and observability Gramians are diagonal and equal, allowing states that are hard to observe and hard to reach to be truncated. Applications of balanced truncation to PDE control problems have been explored in works such as \cite{benner_balancing-related_2013, de_los_reyes_balanced_2011}.

Probably the most widely used MOR method for linear or nonlinear optimal control problems is proper orthogonal decomposition (POD). Early applications of POD in optimal control can be traced back to \cite{ravindran_reduced-order_2000, kunisch_control_1999}, where its use in flow control was explored. POD has also been effectively employed to compute reduced-order controllers \cite{ravindran_adaptive_2002} with nonlinear observers \cite{atwell_reduced_nodate} and to design feedback as well as model predictive controllers \cite{alla_asymptotic_2015, ghiglieri_optimal_2014, kunisch_hjb-pod-based_2004}. Additionally, there has been considerable research on error analysis for POD-based ROMs in optimal control. This includes analysis for nonlinear dynamical systems \cite{hinze_proper_2005}, abstract LQ systems \cite{hinze_error_2008}, parabolic problems \cite{gubisch_pod_2014, studinger_numerical_2012}, and elliptic problems \cite{kahlbacher_pod_2012}. A comparison of a posteriori error estimators for RB and POD methods in LQ optimal control problems was conducted in \cite{tonn_comparison_2011}. One challenge with using the POD approach in optimal control problems is that the basis must be precomputed based on a reference control, which may differ significantly from the final optimal control. As a result, the suboptimal control derived from the POD model may not provide a good approximation of the full-order optimal control. To address this, several adaptive strategies have been developed \cite{Afanasiev_wake_2001, ravindran_adaptive_2002, alla_adaptive_2013, grasle_pod_2017} that update the POD basis during the optimization process. Some of these methods have established strong mathematical foundations over the years, such as the optimality system POD (OSPOD) \cite{kunisch_proper_2008} and trust region POD (TRPOD) \cite{arian_trust-region_nodate}. OSPOD addresses the issue of unmodeled dynamics by updating the POD basis in the direction that minimizes the cost functional, ensuring that the POD-reduced system is computed from the trajectory associated with the optimal control. Convergence results and a-posteriori error estimates for OSPOD have been studied in \cite{volkwein_optimality_nodate, kunisch_uniform_2015}. TRPOD, on the other hand, manages model approximation quality within a trust-region framework by comparing the predicted reduction with the actual reduction in the model function value. 
The combination of trust region methods and model order reduction has also seen notable applications in recent years \cite{bergmann_optimal_2008, yue_accelerating_2013}.

While the use of ROMs for LQ control problems is nowadays rather standard for parabolic problems, hyperbolic problems still pose significant challenges, as they exhibit a slower decay of the Kolmogorov $n$-width \cite{greif_decay_2019,Peh22,HesPU26}, rendering many conventional MOR methods ineffective. This issue is especially prevalent when dealing with transport-dominated fluid systems, such as propagating flame fronts or traveling acoustic and shock waves \cite{krah_front_2022}. 
The main goal of this work is to construct a ROM for \eqref{eq_def:PDE_continuous} using the transformed proper orthogonal decomposition (tPOD)\cite{black_modal_2022} and to determine an optimal control based on the resulting ROM.
The term tPOD is used when considering general transformation groups. 
The sPOD method is the special case in which the transformations are spatial shifts \cite{reiss_shifted_2018,KraMZRS25}.
The tPOD alleviates the issue of the slow decay of the Kolmogorov $n$-width by utilizing a nonlinear approximation ansatz to capture the high-dimensional space, followed by constructing the reduced-order approximation via a nonlinear Galerkin projection \cite{black_efficient_2021, black_projection-based_2020}. 
Other MOR techniques for transport-dominated systems include \cite{rim_transport_2018}, which uses transport reversal and template fitting, \cite{nonino_overcoming_2019}, which employs transport maps, and \cite{krah_front_2022}, which approximates the field variable using a front shape function and a level set function for efficient model reduction. For a detailed overview of MOR methods for transport-dominated problems, see \cite{HesPU26}.

Solving the original optimal control problem consisting of \eqref{eq_def:PDE_continuous} and \eqref{eq_def:FOM_costFunc_continuous} by approximating the FOM using tPOD comes with some challenges, especially since the nonlinear projection framework generically yields nonlinear reduced-order systems.
With this context in mind, using tPOD-G for the general transformation Galerkin models and sPOD-G for the linear advection example considered here, the following are the contributions of this paper:
\begin{itemize} 
    \item Under a smallness condition on the control $u$, in Theorem \ref{thrm:unique_solution} we prove existence and uniqueness of solutions to the \textit{tPOD-Galerkin} (tPOD-G) ROM associated with \eqref{eq_def:PDE_continuous} in a general strongly continuous group setting.
    \item For a sufficiently large regularization parameter $\mu$, Theorem \ref{thrm:control_existence} shows the existence of an optimal control for the underlying reduced-order optimal control problem.
    \item Theorem \ref{thm:first_order_cond} provides necessary optimality conditions for the tPOD-G method in this group setting.
    \item Additionally, in Proposition \ref{prop:max_sv_invariance}, we demonstrate analytically (and later numerically) that the dimension of an arbitrary snapshot set generated by $y(t)$ when transformed by the inverse group transformation is bounded from above for a specific choice of control shape functions $b_k$.  
    \item We examine the 1D linear advection test case with two different variations, a single tilt problem and a double tilt problem. We then compare the results in terms of reduced-order dimension, convergence behavior, and computational time for both the sPOD-G and POD-G methods. 
\end{itemize}
The structure of the paper is as follows. In \Cref{sec:MOR_methods}, we recall the basic ideas of POD and tPOD. 
We briefly review the essential optimality conditions for the POD-G method in \Cref{ssec:POD-G} and then, in \Cref{ssec:sPOD-G}, we thoroughly derive well-posedness and optimality results for the tPOD-G method in a general group setting. In \Cref{sec:results}, we present algorithmic details as well as the numerical results for the 1D linear advection equation, offering a detailed comparison and analysis of timings. Lastly, we conclude with a summary of our findings in \Cref{sec:conclusion}.

\subsection*{Notation}\label{ssec:notations}

We use the notation $\mathbb{R}^{m\times n}$ for the space of $m\times n$ matrices with real entries and the transpose of a matrix $A$ is denoted with $A^\top$.
Moreover, we often write vectors and matrices in terms of their components, e.g., $v=[v_i]_{i=1}^n\in\mathbb{R}^n$, $A=[a_{ij}]_{i,j=1}^n\in\mathbb{R}^{n\times n}$, or $A=[a_{ij}]_{i,j=1}^{m,n}\in\mathbb{R}^{m\times n}$.
The expressions $\norm{v}$ and $\norm{A}$ denote the Euclidean norm of the vector $v$ and the spectral norm of the matrix $A$, respectively.
For matrix-valued coefficient functions, $\|F\|_{L^p(0,T)}$ denotes the $L^p$ norm in time of the spectral norm $\|F(t)\|$, and $p=\infty$ denotes its essential supremum.
For the partial derivative of a function $f$ w.r.t.~the variable $x_i$, we write $\partial_{x_i}f$.
Derivative of $f$ w.r.t.~ any variable except time $t$ is denoted by $f'$.
Throughout the paper, we denote the Bochner spaces of $L^p$ functions with $p\ge 1$ in the interval $[a, b]$ with values in a Hilbert space $X$ by $L^p(a, b;X)$ and we set $L^p(a, b)\vcentcolon= L^p(a, b;\mathbb{R})$. 
Similarly, for $k\ge 1$ we use the notation $H^k(a, b;X)$ and $H^k(a, b)$ for the Bochner spaces with weak derivatives up to order $k$ in $L^2$.
Moreover, we use $H^1_{\mathrm{per}}(a,b)$ for the Sobolev space of weakly differentiable functions with periodic boundary conditions.
The spaces of continuous and continuously differentiable functions from an interval $I\subseteq \mathbb{R}$ to a Hilbert space $X$ are denoted with $C(I,X)$ and $C^1(I,X)$, respectively, and we set $C(I)\vcentcolon= C(I,X)$ and $C^1(I)\vcentcolon= C^1(I,X)$.
The space of bounded linear operators between two Hilbert spaces $X$ and $Y$ is denoted by $\mathcal{L}(X, Y)$ and we set $\mathcal{L}(X) \vcentcolon=\mathcal{L}(X, X)$.
We also often drop the explicit time dependency from some terms for simplicity throughout the paper.

\section{Model order reduction methods}\label{sec:MOR_methods}

In this section, we recall the most important steps for constructing ROMs by POD and tPOD. For POD, the results are all well-known and can be found in, e.g., \cite{gubisch_pod_2014}. For the presentation of Galerkin ROMs associated with tPOD, we closely follow the exposition in \cite{black_projection-based_2020}.

\subsection{POD-Galerkin method}

Given trajectories $y \in C([0, T]; \mathcal{D}(\mathcal{A}))$ obtained by solving \eqref{eq_def:PDE_continuous},
consider snapshots $y_j = y(t_j) \in \mathcal{D}(\mathcal{A})$ for $j=1,\dots,n_t$ and define a snapshot set by 
\begin{equation*}
    \mathcal{Y} := \text{span} \Bigl\{\, y(t_j) \ | \ t_j \in [0, T]   \text{ for } \: 1 \leq j \leq n_t \Bigr\} \subset \mathcal{D}(\mathcal{A})\;.
\end{equation*}
Our aim is to identify a suitable $\ell$-dimensional subspace $\mathcal{Y}_{\ell} \subset \mathcal{Y}$ described by the basis $\{\phi_1, \dots , \phi_{\ell}\}$ that minimizes the approximation error in 
\begin{equation}\label{eq_def:Galerkin_POD_ansatz}
    y(t) \approx \sum^{\ell}_{i=1} \alpha_i(t)\phi_i,
\end{equation} 
via the optimization problem
\begin{subnumcases}{} 
    \mathrm{min} \: \frac{1}{2} \int^{T}_0 \norm{y(t) - \sum^{\ell}_{i=1} \Bigl\langle y(t), \phi_i \Bigr \rangle_H \: \phi_i}^2_H \mathrm{d}t ,\\[1em]
    \text{s.t.} \: \{\phi_i\}^{\ell}_{i=1} \subset \mathcal{D}(\mathcal{A}) \: \text{and} \:\langle \phi_i, \phi_j \rangle_H = \delta_{ij}, \: i, j = 1, \ldots , \ell \;.
\end{subnumcases}
The spatial basis functions $\{\phi_i\}^{\ell}_{i=1}$ are often referred to as POD basis or POD modes, and $\alpha_i(t) = \langle y(t), \phi_i\rangle_H $ are the reduced states.
For a detailed analysis and solution procedure of this optimization problem, see, e.g., \cite{GubV17,quarteroni_reduced_2016}.

After computing the POD basis, we are interested in deriving a low-dimensional approximation of \eqref{eq_def:PDE_continuous}. 
For this, we use the Galerkin ansatz \eqref{eq_def:Galerkin_POD_ansatz} in \eqref{eq_def:PDE_continuous} and subsequently use the Galerkin orthogonality condition to obtain
\begin{equation*}
\begin{aligned}
   &\Bigl \langle \phi_j, \sum^{\ell}_{i=1} \dot{\alpha}_i(t) \phi_i \Bigr \rangle_H - \Bigl \langle \phi_j, \sum^{\ell}_{i=1}\alpha_i(t)\mathcal{A}\phi_i \Bigr \rangle_H  - \Bigl\langle \phi_j, \mathcal{B}u(t) \Bigr \rangle_H  = 0 \\
    &\Bigl \langle \phi_j, \sum^{\ell}_{i=1}\alpha(0)\phi_i \Bigr \rangle_H  = \Bigl\langle \phi_j, y_0 \Bigr\rangle_H
\end{aligned}
\end{equation*}
for $j=1,\ldots,\ell$.
As a consequence, the low-dimensional ROM is given as
\begin{equation}\label{eq_def:POD_Galerkin_continuous}
\begin{aligned}
   & \dot{\alpha}(t) = A_{\ell} \alpha(t) + B_{\ell} u(t)\\
    &\alpha(0)  = \alpha_0
\end{aligned}
\end{equation}
where
\begin{equation*} 
\begin{aligned}
     A_\ell &:= \begin{bmatrix} \langle \phi_1, \mathcal{A}\phi_1\rangle_H  & \ldots  & \langle \phi_1, \mathcal{A}\phi_{\ell}\rangle_H \\ \vdots \\
     \langle \phi_{\ell}, \mathcal{A}\phi_1\rangle_H & \ldots  & \langle \phi_{\ell}, \mathcal{A}\phi_{\ell}\rangle_H \end{bmatrix} \in \mathbb{R}^{\ell \times \ell}, \quad\\[1ex]
      B_{\ell} &:= \begin{bmatrix} \langle \phi_1, \mathcal{B}e_1\rangle_H  & \ldots  & \langle \phi_1, \mathcal{B}e_m\rangle_H \\ \vdots \\
     \langle \phi_{\ell}, \mathcal{B}e_1\rangle_H & \ldots  & \langle \phi_{\ell}, \mathcal{B}e_m\rangle_H \end{bmatrix} \in \mathbb{R}^{\ell \times m}, \\[1ex]
     \alpha_0 &= [\langle \phi_j, y_0\rangle_H ]^{\ell}_{j=1} \in \mathbb{R}^{\ell}.
 \end{aligned}
\end{equation*}

\subsection{tPOD-Galerkin method}\label{ssec:sPODG}
The discussion about to follow is inspired by the exposition described in \cite{black_projection-based_2020}. In contrast to standard POD, the tPOD method decomposes $y(t)$ using a nonlinear decomposition ansatz
\begin{equation}\label{eq_def:sPOD_continuous}
    y(t) \approx \sum^{r}_{i=1} \alpha_i(t) \mathcal{T}_i(z_i(t))\phi_i, 
\end{equation}
where $z_i(t) \in Z_i, i=1,\dots,r$ are time-dependent transformation parameters and  $\mathcal{T}_i\colon Z_i \rightarrow \mathcal{L}(H), i=1, \ldots, r$, are appropriately chosen transformation operators. Throughout this manuscript, we will consider the case $Z_i= \mathbb{R}.$ For the transformation operators, we assume that the space $H$ is $\mathcal{T}_i$-invariant in the sense that $\mathcal{T}_i(z) H \subseteq H$ for all $z \in Z_i$ and $i=1,\ldots,r$. While a general choice of the transformation operators is a challenge on its own, see \cite{burela_parametric_2023}, for specific scenarios $\mathcal{T}_i$ can be designed by analyzing the underlying dynamics of the problem at hand. 
For example, in the case of the linear advection equation \eqref{eq:lin_adv_eq}, it makes sense to consider the periodic shift group, which is the group generated by $\mathcal{A}$, see Proposition \ref{prop:max_sv_invariance} for the motivation of this choice. 
Given $\mathcal{T}_i$, we aim to minimize the approximation error in \eqref{eq_def:sPOD_continuous} and consider the following minimization problem
\begin{equation}\label{eq_def:sPOD_optimization}
\left\{
\begin{aligned}
    \min \:\: & \frac{1}{2} \int^{T}_0 \norm{y(t) - \sum^r_{i=1} \alpha_i(t) \mathcal{T}_i(z_i(t))\phi_i }^2_H \mathrm{d}t\;, \\[1em]
    \text{s.t.} \:\: & \{\phi_i\}^{r}_{i=1} \subset \mathcal{D}(\mathcal{A}) \: \text{and} \: \norm{\phi_i}_H = 1, \alpha_i \in L^2(0, T), \\
    & z_i \in L^2(0, T; Z_i) \;\text{ for }i=1,\ldots,r
\end{aligned}
\right.
\end{equation}
A point to note is that unlike the POD minimization problem where the modes are required to form an orthonormal set, in tPOD, the transformed modes $\{\mathcal{T}_i(z_i(t)) \phi_i\}^r_{i=1}$ are chosen to be just normalized and not necessarily orthonormal, see \cite[Example 4.4]{black_projection-based_2020}.

As noted in \cite[Remark 1.1]{black_projection-based_2020}, the tPOD ansatz can be simplified by considering the same transformation operator with the same transformation parameters for different modes $\phi_i$, i.e.,
\begin{equation}
    y(t) \approx \sum^K_{k=1}\sum^{\ell^k}_{i=1} \alpha_{k, i}(t) \mathcal{T}_k(z_k(t))\phi_{k, i}.
\end{equation}

In particular, if we consider only a single time-dependent transformation parameter $z(t)\in \mathbb R$ and a corresponding 
transformation operator $\mathcal{T}\colon \mathbb R\to \mathcal{L}(H)$, this expression further simplifies according to 
\begin{equation}\label{eq_def:sPOD_advection}
    y(t) = \sum^{\tilde{\ell}}_{i=1} \alpha_{i}(t) \mathcal{T}(z(t)) \phi_{i}\;.
\end{equation}
Subsequently, the minimization problem \eqref{eq_def:sPOD_optimization} is simplified to
\begin{equation}\label{eq_def:sPOD_advection_optimization}
\left\{
\begin{aligned}
    \min \:\: & \frac{1}{2} \int^{T}_0 \norm{y(t) - \mathcal{T}(z(t))\sum^{\tilde{\ell}}_{i=1} \alpha_i(t) \phi_i }^2_H \mathrm{d}t\;, \\[1em]
    \text{s.t.} \:\: & \{\phi_i\}^{\tilde{\ell}}_{i=1} \subset \mathcal{D}(\mathcal{A}) \: \text{and} \: \norm{\phi_i}_H = 1, \langle \phi_i, \phi_j \rangle_H = \delta_{ij}, \\
    & \alpha_i \in L^2(0, T) \;\text{ for }i,j=1,\ldots,\tilde\ell,
\end{aligned}
\right.
\end{equation}
where, in contrast to \eqref{eq_def:sPOD_optimization}, we also enforce the modes to be orthonormal.
Thus, under the assumption that the transformation parameter $z\in L^2(0,T)$ is given, problem \eqref{eq_def:sPOD_advection_optimization} is solvable and its solution can be obtained equivalently via a POD minimization of the transformed data $\mathcal{T}^*(z(t))y(t), t\in [0,T],$ see \cite[Theorem 4.8]{black_projection-based_2020}. 
We proceed by constructing a ROM for \eqref{eq_def:PDE_continuous} via a nonlinear Galerkin projection, also known as the Dirac--Frenkel variational principle, cf.~\cite{Dir30,Fre34,Lub08}, and obtain
\begin{multline}\label{eq:sPOD_galerkin_general}
    \begin{bmatrix} M_1(z(t)) & N(z(t))\alpha(t) \\ \alpha(t)^\top N(z(t))^\top & \alpha(t)^\top M_2(z(t)) \alpha(t)\end{bmatrix} \begin{bmatrix} \dot{\alpha}(t) \\ \dot{z}(t) \end{bmatrix} = \begin{bmatrix} A_1(z(t)) & 0 \\ \alpha(t)^\top A_2(z(t)) & 0 \end{bmatrix}
    \begin{bmatrix} \alpha(t) \\ z(t) \end{bmatrix} \\+ \begin{bmatrix}  B_1(z(t)) \\ \alpha(t)^\top B_2(z(t))\end{bmatrix}u(t)
\end{multline}
where $M_1,N,M_2\colon \mathbb{R}\to\mathbb{R}^{\tilde\ell\times\tilde\ell}$ are defined via
\begin{equation*}
\begin{aligned}
&M_1(z) := \big[\langle \mathcal{T}(z)\phi_i,\;\mathcal{T}(z)\phi_j\rangle_H\big]_{i,j=1}^{\tilde{\ell}},\quad
N(z) := \big[\langle \mathcal{T}(z)\phi_i,\;\mathcal{T}'(z)\phi_j\rangle_H\big]_{i,j=1}^{\tilde{\ell}},\\
&M_2(z) := \big[\langle \mathcal{T}'(z)\phi_i,\;\mathcal{T}'(z)\phi_j\rangle_H\big]_{i,j=1}^{\tilde{\ell}}
\end{aligned}
\end{equation*}
On the right-hand side, we have $A_1,A_2\colon \mathbb{R}\to\mathbb{R}^{\tilde\ell\times\tilde\ell}$ and $B_1,B_2\colon \mathbb{R}\to\mathbb{R}^{\tilde\ell\times m}$ which are defined via
\begin{subequations}
\begin{align}
A_1(z) &:= \big[\langle \mathcal{T}(z)\phi_i,\;\mathcal{A}\big(\mathcal{T}(z)\phi_j\big)\rangle_H\big]_{i,j=1}^{\tilde{\ell}},
&A_2(z) &:= \big[\langle \mathcal{T}'(z)\phi_i,\;\mathcal{A}\big(\mathcal{T}(z)\phi_j\big)\rangle_H\big]_{i,j=1}^{\tilde{\ell}},\\
    \label{eq:ROM_B_matrices}
    B_1(z) &:= \big[\langle \mathcal{T}(z)\phi_i,\;\mathcal{B}e_j\rangle_H\big]_{i,j=1}^{\tilde{\ell},m}, &B_2(z) &:= \big[\langle \mathcal{T}'(z)\phi_i,\;\mathcal{B}e_j\rangle_H\big]_{i,j=1}^{\tilde{\ell},m}.
\end{align}
\end{subequations}
As for the initial condition of the reduced system, for given $z(0)=z_0$ we
choose $\alpha(0)$ such that the approximation error in
\begin{equation*}
    y_0 \approx \sum^{\tilde{\ell}}_{i=1} \alpha_i(0)\, \mathcal{T}(z_0) \phi_i
\end{equation*}
is minimized. Performing an orthogonal projection onto the span of the
transformed modes $\{\mathcal{T}(z_0)\phi_j\}^{\tilde{\ell}}_{j=1}$, this leads to
\begin{equation*}
    \langle \mathcal{T}(z_0)\phi_j, y_0\rangle_H
    = \sum^{\tilde{\ell}}_{i=1} \alpha_i(0)\,
      \langle \mathcal{T}(z_0)\phi_j, \mathcal{T}(z_0)\phi_i\rangle_H
    = \big[M_1(z_0)\,\alpha(0)\big]_j
    \quad\text{for } j=1,\ldots,\tilde\ell .
\end{equation*}
Consequently, $\alpha(0)$ solves the linear system
$M_1(z_0)\,\alpha(0)=\big[\langle \mathcal{T}(z_0)\phi_j, y_0\rangle_H\big]_{j=1}^{\tilde\ell}$,
i.e.,
\begin{equation*}
    \alpha(0) = M_1(z_0)^{-1}
    \big[\langle \mathcal{T}(z_0)\phi_j, y_0\rangle_H\big]_{j=1}^{\tilde\ell},
\end{equation*}
provided that $M_1(z_0)$ is invertible. 
This is always the case in the setting of Section~\ref{ssec:sPOD-G}.

The choice of the operator $\mathcal{T}$ significantly influences the performance of the tPOD method which will only be practicable if the singular values associated with the transformed snapshot set decay rapidly. If the shape functions $b_k$ in the control term can be chosen freely, for operators $\mathcal{A}$ generating a strongly continuous group $\mathcal{S}(t)$ the following proposition demonstrates how certain choices of the shape functions may possibly lead to a low rank of the transformed snapshot matrix.

\begin{prpstn}\label{prop:max_sv_invariance}
    Consider an abstract control system of the form \eqref{eq_def:PDE_continuous} with $\mathcal{A}\colon\mathcal{D}(\mathcal{A})\subset H\to H$ being the infinitesimal generator of a strongly continuous group $\mathcal{S}(t)\in\mathcal{L}(H)$, $H$ a Hilbert space, $y_0\in H$, $u\in L^2(0,T; U)$ with $U=\mathbb{R}^m$, and $\mathcal{B}\in\mathcal{L}(U,H)$ defined via \eqref{eq:control_operator} with $b_1,\ldots,b_m\in \mathcal{D}(\mathcal{A})$.
    Moreover, let $\mathrm{span}\{b_1,\dots,b_m\}$ be an $\mathcal{A}$-invariant subspace.
    Then, there exists a family of linear operators $\mathcal{T}\colon\mathbb{R}\to\mathcal{L}(H)$ such that for any given set of discrete time points $t_1,\ldots,t_{n_t}\in[0,T]$ we have
  \begin{align*}
      \mathrm{dim}(\mathrm{span}\{\mathcal{T}(-t_1)y(t_1),\dots,\mathcal{T}(-t_{n_t})y(t_{n_t})\}) \le \min(m+1,n_t),
  \end{align*}
  where $y$ denotes the mild solution of \eqref{eq_def:PDE_continuous}.
\end{prpstn}
\begin{proof}
    The mild solution is given by
    \begin{align*}
        y(t)    &=\mathcal{S}(t)y_0 + \int_0^t\mathcal{S}(t-s)\mathcal{B}u(s)\,\mathrm{d}s =\mathcal{S}(t)y_0 + \int_0^t \mathcal{S}(t-s)\sum_{k=1}^m b_k u_k(s)\,\mathrm{d}s\\ 
                &=\mathcal{S}(t)\left(y_0 + \int_0^t \mathcal{S}(-s)\sum_{k=1}^m b_k u_k(s)\,\mathrm{d}s\right),
    \end{align*}
    see, e.g., \cite[Def.~3.1.4]{CurZ95}.
    Since $\mathrm{span}\{b_1,\dots,b_m\}$ is an $\mathcal{A}$-invariant subspace, it is also an $\mathcal{S}(-s)$-invariant subspace \cite[Lemma 2.5.4]{CurZ95}. Hence, there exist $\beta_{1,k},\dots,\beta_{m,k} \colon [0,T]\to\mathbb R$ for $k=1,\ldots,m$ such that
     \begin{align*}
        y(t)= \mathcal{S}(t)\left(y_0 + \sum_{i=1}^m b_i \sum_{k=1}^m\int_0^t  \beta_{i,k}(s)  u_k(s)\,\mathrm{d}s\right).
    \end{align*}
    Consequently, choosing $\mathcal{T}(t)=\mathcal{S}(t)$, for any $j\in\lbrace 1,\ldots,n_t\rbrace$ we obtain
    \begin{align*}
        \mathcal{T}(-t_j)y(t_j) = y_0 + \sum_{i=1}^m b_i \sum_{k=1}^m\int_0^{t_j}  \beta_{i,k}(s)  u_k(s)\,\mathrm{d}s \in \mathrm{span}\{y_0,b_1,\dots,b_m\}.
    \end{align*}
    \end{proof}

\paragraph{Simplifications when $\mathcal{T}$ is a periodic shift operator}
With regard to the assumptions on $\mathcal{T}$ and in view of the dynamics of the linear advection equation \eqref{eq:lin_adv_eq}, we consider the periodic shift operator $\mathcal{T}\colon \mathbb{R}\to\mathcal{L}(L^2(0,l))$ defined via
\begin{equation}
    \label{eq:periodic_shift_operator}
    \mathcal{T}(z)\phi(x)=
    \begin{cases}
        \phi(x-\eta) & \text{for } \eta \le x \le l, \\
        \phi(x-\eta+l) & \text{for } 0 \le x < \eta 
    \end{cases}
\end{equation}
with $\eta\vcentcolon=z\mod l$, see e.g.~\cite[Def.~1.2.2]{schulze_energy-based}.
Especially, this choice ensures that $\mathcal{T}$ is isometric and that $H^1_\mathrm{per}(0,l)\subset H=L^2(0,l)$ is a $\mathcal{T}(z)$-invariant subspace for any $z\in\mathbb{R}$.

\begin{rmrk}\label{rem:diff_T_vs_phi}
Note that if $\phi$ is sufficiently regular, we have that $\mathcal{T}'(z)\phi:=\tfrac{\mathrm{d}}{\mathrm{d}z}(\mathcal{T}(z)\phi)=-\mathcal{T}(z)\phi'$ and $\mathcal{T}''(z)\phi=\mathcal{T}(z)\phi''$. In particular, it holds that $\mathcal{T}^*(z)=\mathcal{T}(-z)=\mathcal{T}^{-1}(z).$
\end{rmrk}
Following the assumptions mentioned in \cite[Section 6]{black_projection-based_2020}, with our specific choice of $\mathcal{T}$, we can further simplify the ROM in \eqref{eq:sPOD_galerkin_general} by eliminating the explicit dependency on the shift $z$ in $M_1$, $N$, $M_2$, $A_1$, and $A_2$.
Subsequently, these matrices can be written as
\begin{equation}
    \label{eq:ROM_N_M2}
    \begin{aligned}
        M_1 = I_{\tilde\ell},\quad N = -\big[\langle \phi_i,\;\phi_j'\rangle_H\big]_{i,j=1}^{\tilde{\ell}},\quad M_2 = \big[\langle \phi_i',\;\phi_j'\rangle_H\big]_{i,j=1}^{\tilde{\ell}},\\
        A_1 = \big[\langle \phi_i,\;\mathcal{A}\phi_j\rangle_H\big]_{i,j=1}^{\tilde{\ell}},\quad A_2 = -\big[\langle \phi_i',\;\mathcal{A}\phi_j\rangle_H\big]_{i,j=1}^{\tilde{\ell}}
    \end{aligned}
\end{equation}
where we exploited that the modes $\{\phi_i\}_{i=1}^{\tilde{\ell}}$ are orthonormal. 
Moreover, we find that $A_1 = v N,A_2 = v M_2$.
This further simplifies the ROM and is called sPOD-G ROM
\begin{equation}\label{eq_def:sPOD_galerkin_advection}
    \begin{bmatrix} I_{\tilde{\ell}} & N \alpha(t) \\ \alpha(t)^\top N^\top & \alpha(t)^\top M_2 \alpha(t)\end{bmatrix} \begin{bmatrix} \dot{\alpha}(t) \\ \dot{z}(t) \end{bmatrix} = v\begin{bmatrix} N & 0 \\ \alpha(t)^\top M_2 & 0 \end{bmatrix}
    \begin{bmatrix} \alpha(t) \\ z(t) \end{bmatrix} + \begin{bmatrix}  B_1(z(t)) \\ \alpha(t)^\top B_2(z(t))\end{bmatrix}u(t).
\end{equation}

\begin{rmrk}
    The ROM \eqref{eq:sPOD_galerkin_general} is derived from requiring the residual of the tPOD approximated FOM to be orthogonal to the tangent space of the tPOD approximation manifold at the current state which is the Dirac–Frenkel variational principle, see \cite{black_projection-based_2020} for a derivation.
    The aforementioned variational principle is also the foundation of dynamic low-rank approximation (DLRA) \cite{koch_DLRA} which is a nonlinear dynamical reduced-basis technique which has been used for devising efficient ROMs for transport-dominated problems \cite{nobile_petrovgalerkin_2025} and also in an optimal control setting \cite{saluzzi_dynamical_2026}. 
    Unlike the tPOD-G framework, this method requires no FOM runs to build the basis upfront, rather, is fully equation driven and the basis is computed dynamically.
    However, it presupposes that the state carries a multi-index structure (a minimum of $2$ excluding time $t$) to be factorized \cite{einkemmer_review_2025}. 
    In our case, we deal with a non-parametric 1D linear advection equation, where the state is $y(t, x)$
    and the DLRA assumption of a multi-index structure is not satisfied.
\end{rmrk}

\section{Optimal control with reduced-order models}\label{sec:NOC_MOR_methods}

In this section, we briefly recall standard first-order necessary optimality conditions for the POD-G method. 
Subsequently, we establish theoretical results regarding the existence of solutions to the reduced optimal control problem associated with the tPOD-G method. In particular, we show the existence of a solution to the reduced state equation provided $u$ is small enough and, as a consequence, obtain the existence of an optimal control if the regularization parameter $\mu$ is chosen large enough. We also derive the necessary optimality conditions for the underlying optimal control problem.

\subsection{POD-based optimal control}\label{ssec:POD-G}

As the POD reduced state equation \eqref{eq_def:POD_Galerkin_continuous} is linear, deriving the associated optimality system is straightforward and yields
\begin{subnumcases}{\mathrm{OC}_{\scaleto{\mathrm{POD-G}}{3pt}}:=} 
    \dot{\alpha}(t) = A_\ell \alpha(t) + B_{\ell} u(t) \label{eq_def:OC_FRTO_PODG_1},\\
    \alpha(0)  = \alpha_0 \label{eq_def:OC_FRTO_PODG_2},\\[1em]
    - \dot{\lambda}^{\ell}(t) = A^\top_{\ell} \lambda^{\ell}(t) + \alpha(t) - \hat{y}_\mathrm{d}(t) \label{eq_def:OC_FRTO_PODG_3}, \\
    \lambda^{\ell}(T) = 0  \label{eq_def:OC_FRTO_PODG_4},\\[1em]
    \mu u(t) + B^\top_{\ell} \lambda^{\ell}(t) = 0\;. \label{eq_def:OC_FRTO_PODG_5}
\end{subnumcases}
The first two equations are the reduced-order state equations already given in \eqref{eq_def:POD_Galerkin_continuous}.
Equations \eqref{eq_def:OC_FRTO_PODG_3}-\eqref{eq_def:OC_FRTO_PODG_4} are the reduced-order adjoint equations derived from the reduced-order state equation. Here, $\lambda^{\ell}(t) \in \mathbb{R}^{\ell}$ is the reduced adjoint and $\hat{y}_\mathrm{d}(t) = [\langle \phi_j, y_\mathrm{d}(t)\rangle_H ]^{\ell}_{j=1} \in \mathbb{R}^{\ell}$. Equation \eqref{eq_def:OC_FRTO_PODG_5} is the optimality condition for the control. Let us emphasize that \eqref{eq_def:OC_FRTO_PODG_1}-\eqref{eq_def:OC_FRTO_PODG_5} correspond to the (constrained) minimization of the cost functional \eqref{eq_def:FOM_costFunc_continuous} given as
\begin{equation}\label{eq_def:PODG_costFunc_continuous}
    \mathcal{J}_{\scaleto{\mathrm{POD-G}}{3pt}}(\alpha, u) = \frac{1}{2} \int^{T}_0 \norm{\sum^{\ell}_{i=1} \alpha_i(t)\phi_i - y_\mathrm{d}(t)}^2_H \, \mathrm{d}t + \frac{\mu}{2} \int_0^{T}\norm{u(t)}^2_{U}\,\mathrm{d}t.
\end{equation}

\subsection{tPOD-based optimal control}\label{ssec:sPOD-G}

For the tPOD-G method, we consider the optimization problem
\begin{multline}\label{eq_def:sPOD_costFunc_continuous}
    \underset{\begin{subarray}{c}
  u \in L^2(0, T; U)
  \end{subarray}}{\min} \:\:\mathcal{J}_{\scaleto{\mathrm{tPOD-G}}{3pt}}(\alpha, z, u):= \frac{1}{2} \int^{T}_0 \norm{\sum^{\tilde{\ell}}_{i=1} \alpha_i(t)\mathcal{T}(z(t))\phi_i - y_\mathrm{d}(t)}^2_H \, \mathrm{d}t \\ + \frac{\mu}{2} \int_0^{T}\norm{u(t)}^2_{U}\,\mathrm{d}t\; \quad \text{s.t. } \eqref{eq:sPOD_galerkin_general}, \alpha(0)=\alpha_0, z(0)=z_0.
\end{multline}
Note that, in contrast to the previous subsection, we are faced with an optimal control problem that is constrained by the nonlinear reduced state equation \eqref{eq:sPOD_galerkin_general} for which we understand solutions in the sense of Carath\'eodory, see, e.g., \cite[Chapter I.5.1]{Hal80}.
We now formulate the analysis for a strongly continuous group. The periodic translations considered in \eqref{eq:periodic_shift_operator} are a special case. We choose the transformation group $\mathcal{T}$ to be the evolution generated by $\mathcal A$.

\begin{asmp}\label{asmp:general_assumption}
Let $H$ be a real Hilbert space, $U=\mathbb R^m$, $T>0$, and $\mathcal B\in\mathcal L(U,H)$ with $b:=\|\mathcal B\|_{\mathcal L(U,H)}>0$. Let $\mathcal A$ generate a strongly continuous group $\mathcal T(s)$ on $H$, and fix $C_{\mathcal T}\ge1$, $\omega_{\mathcal T}\ge0$ such that
\begin{equation}\label{eq:group_bound}
 \|\mathcal T(s)\|_{\mathcal L(H)}\le C_{\mathcal T}e^{\omega_{\mathcal T}|s|}\qquad(s\in\mathbb R).
\end{equation}
Given orthonormal modes $\phi_1,\ldots,\phi_{\tilde\ell}\in\mathcal D(\mathcal A^2)$ and initial values $\alpha_0\in\mathbb R^{\tilde\ell}$, $z_0\in\mathbb R$, define
\[
 V\colon \mathbb{R}^{\tilde{\ell}} \to \mathcal{D}(\mathcal{A}^2) \subseteq H,\quad  V\beta:=\sum_{i=1}^{\tilde\ell}\beta_i\phi_i,\qquad W\colon \mathbb{R}^{\tilde{\ell}} \to \mathcal{D}(\mathcal{A}) \subseteq H, \quad W:=\mathcal A V.
\]
The vectors $\phi_1,\ldots,\phi_{\tilde\ell},W\alpha_0$ are assumed to be linearly independent.
\end{asmp}
The domain assumption gives $\mathcal T'(z)\phi_i=\mathcal T(z)\mathcal A\phi_i$ and $\mathcal T''(z)\phi_i=\mathcal T(z)\mathcal A^2\phi_i$, with continuous derivatives in $H$ \cite[Proposition 2.1.5]{tucsnak_observation_2009}.
For periodic advection with nonzero velocity, we have $\mathcal D(\mathcal A^2)=H^2_{\mathrm{per}}(0,l)$. The group parameter now measures time along the uncontrolled flow. For $\mathcal A=-v\partial_x$, we have $(\mathcal T(z)\phi)(x)=\phi(x-vz)$, with periodic interpretation. The spatial shift used in Sections~\ref{ssec:sPODG} and~\ref{sec:results} is therefore $\zeta=vz$. For $v\ne0$, the two reduced formulations are related by this change of variable. All results below use the group parameter $z$.

We use the coefficient matrices $M_1(z),N(z),M_2(z),B_1(z),B_2(z)$ introduced in Section~\ref{ssec:sPODG}. 
With the setting of \Cref{asmp:general_assumption}, the ROM \eqref{eq:sPOD_galerkin_general} takes the form
\begin{equation}\label{eq:group_ROM}
 \mathsf M(w)\dot w=\mathsf M(w)e_z+\mathsf B(w)u,\qquad
 \mathsf M(w):=\begin{bmatrix}M_1(z)&N(z)\alpha\\\alpha^\top N(z)^\top&\alpha^\top M_2(z)\alpha\end{bmatrix},
\end{equation}
where $w=(\alpha,z)$, $e_z=(0,1)\in\mathbb R^{\tilde\ell+1}$ and
$\mathsf B(w):=[B_1(z)^\top,\ B_2(z)^\top\alpha]^\top$.
Indeed, we have $A_1(z)=N(z)$ and $A_2(z)=M_2(z)$. 
For periodic advection, the change to the spatial-shift coordinate described in the previous paragraph recovers \eqref{eq_def:sPOD_galerkin_advection}.
Writing $Y(w):=\mathcal T(z)V\alpha$, we obtain
\begin{equation}\label{eq:tangent_map}
\begin{aligned}
 Y'(w)(\beta,\eta)&=\mathcal T(z)(V\beta+\eta W\alpha),\\
 \mathsf M(w)&=Y'(w)^*Y'(w),\qquad \mathsf B(w)=Y'(w)^*\mathcal B.
\end{aligned}
\end{equation}

Thus the mass matrix is a Gramian matrix. Since $\mathcal T(z)$ is invertible, its positive definiteness is equivalent to the linear independence of $\phi_1,\ldots,\phi_{\tilde\ell},W\alpha$. 
\Cref{asmp:general_assumption} imposes this condition only at the initial amplitude, which can be checked before solving the reduced equation.

We obtain the following well-posedness result for the nonlinear reduced-order equation.
\begin{thrm}\label{thrm:unique_solution}
Let Assumption~\ref{asmp:general_assumption} hold. Set
\begin{equation}\label{eq:explicit_constants}
\begin{gathered}
 \mathsf M_0 := \begin{bmatrix}
   I_{\tilde\ell}  & V^*W\alpha_0 \\[2pt]
   (W\alpha_0)^*V  & \|W\alpha_0\|_H^2
 \end{bmatrix}, \\[1.5ex]
 \sigma := \sqrt{\lambda_{\min}(\mathsf M_0)}, \qquad
 d := \|W\|_{\mathcal L(\mathbb R^{\tilde\ell},H)}, \qquad
 \rho := \frac{\sigma}{2d}, \\[1.5ex]
 Z := |z_0| + T + \rho, \qquad
 K := C_{\mathcal T}\, e^{\omega_{\mathcal T} Z}, \qquad
 L := \frac{2bK}{\sigma}.
\end{gathered}
\end{equation}
If $u$ satisfies
\begin{equation}\label{eq:l2_bnd_u}
 \|u\|_{L^2(0,T;U)}<\varepsilon_*:=\frac{\rho}{L\sqrt T}
 =\frac{\sigma^2}{4dbC_{\mathcal T}\sqrt T}
       e^{-\omega_{\mathcal T}(|z_0|+T+\sigma/(2d))},
\end{equation}
then the initial value problem consisting of \eqref{eq:group_ROM} and $\alpha(0)=\alpha_0$, $z(0)=z_0$ has a unique solution $(\alpha,z)\in H^1(0,T;\mathbb R^{\tilde\ell})\times H^1(0,T)$. In addition, for every fixed $R<\varepsilon_*$ there exists a constant $C_{R}>0$ independent of $u$, such that for all $u$ with $\|u\|_{L^2(0,T;U)}\le R$ we have
\begin{equation}\label{eq:a_priori_bnds}
 \max\big(\|\alpha\|_{H^1(0,T;\mathbb R^{\tilde\ell})},\|z\|_{H^1(0,T)}\big)\le C_R.
\end{equation}
If the group is unitary and, in addition, all $2\tilde\ell$ vectors $\phi_i,\mathcal A\phi_i$ are linearly independent, the same conclusions hold with $\varepsilon_*$ replaced by the simpler threshold
\begin{equation}\label{eq:unitary_threshold}
\varepsilon_{\mathrm{un}}:=\frac{\|\alpha_0\|}{b\sqrt T}
\end{equation}
\end{thrm}
\begin{proof}
We start with a general group setting. 
The reduced equation can be written as an explicit ODE whenever its mass matrix is invertible. Our first aim is therefore to bound this matrix from below in a neighborhood of the uncontrolled solution. 

By Assumption~\ref{asmp:general_assumption}, $\mathsf M_0$ is positive definite, and hence $\sigma>0$. 
Moreover, we have $W\alpha_0\ne0$ and thus $d>0$. All constants in \eqref{eq:explicit_constants} are therefore well defined and depend only on the prescribed data. 
For arbitrary $(\beta,\eta)\in\mathbb R^{\tilde\ell}\times\mathbb R$, the definition of $\mathsf M_0$ gives
\[
 \|V\beta+\eta W\alpha_0\|_H^2
   =(\beta,\eta)^\top \mathsf M_0(\beta,\eta)
   \ge\sigma^2\|(\beta,\eta)\|^2.
\]

We next transfer this lower bound from $\alpha_0$ to amplitudes satisfying $\|\alpha-\alpha_0\|\le\rho$. Using the reverse triangle inequality, we consequently obtain
\begin{align*}
 \|V\beta+\eta W\alpha\|_H
 &\ge\|V\beta+\eta W\alpha_0\|_H
      -|\eta|\,\|W(\alpha-\alpha_0)\|_H\\
 &\ge\sigma\|(\beta,\eta)\|-d\rho|\eta|
 \ge(\sigma-d\rho)\|(\beta,\eta)\|
 =\frac\sigma2\|(\beta,\eta)\|.
\end{align*}

Here the second inequality uses the operator norm of $W$, the third uses $|\eta|\le\|(\beta,\eta)\|$, and the last equality follows from the choice $\rho=\sigma/(2d)$.
 To obtain a lower bound for the mass matrix in \eqref{eq:tangent_map}, we must also control the effect of $\mathcal T(z)$. For $|z|\le Z$, the group property and \eqref{eq:group_bound} also imply
\[
 \|x\|_H=\|\mathcal T(-z)\mathcal T(z)x\|_H
 \le K\|\mathcal T(z)x\|_H\quad\text{for all }x\in H.
\]
Combining the last two estimates and then squaring yields
\begin{equation}\label{eq:mass_lower_bound}
 (\beta,\eta)^\top \mathsf M(w)(\beta,\eta)
 =\|Y'(w)(\beta,\eta)\|_H^2
 \ge\frac{\sigma^2}{4K^2}\|(\beta,\eta)\|^2
\end{equation}
whenever $\|\alpha-\alpha_0\|\le\rho$ and $|z|\le Z$. 
On the open set where $\mathsf M(w)$ is positive definite, we can express \eqref{eq:group_ROM} as an explicit ODE,
\begin{equation}\label{eq:ROM_as_explicit_ODE}
 \dot w=e_z+G(w)u,\qquad G(w):=\mathsf M(w)^{-1}Y'(w)^*\mathcal B.
\end{equation}
The coefficient $G$ is continuously differentiable since the transformed modes are twice continuously differentiable w.r.t.~$z$, so $Y'$, $\mathsf M$ and $\mathsf B$ are continuously differentiable, as is matrix inversion on the set of invertible matrices.
On each compact subset of this open set, the right-hand side is locally Lipschitz in $w$, with an integrable Lipschitz bound proportional to $\|u(t)\|_U$. Its growth bound there is proportional to $1+\|u(t)\|_U$.
The right-hand side satisfies the assumptions of Theorems~5.2 and 5.3 in \cite{Hal80}, which yields
that there exists a maximal interval of existence within the prescribed time interval for which we have a unique solution $(\alpha,z)$ in the sense of Carath\'{e}odory.
We now estimate how far the local solution can move from the uncontrolled solution. For this purpose, we need a uniform bound on $G(w)$ in \eqref{eq:ROM_as_explicit_ODE}. Put $Q(w):=\mathsf M(w)^{-1}Y'(w)^*$. Since $\mathsf M=Y'(w)^*Y'(w)$ is symmetric,
\[
 Q(w)Q(w)^*=\mathsf M(w)^{-1}Y'(w)^*Y'(w)\mathsf M(w)^{-1}=\mathsf M(w)^{-1}.
\]
Consequently, \eqref{eq:mass_lower_bound} gives
$\|G(w)\|\le b\|Q(w)\|=b\sqrt{\|\mathsf M(w)^{-1}\|}\le L$ on the set used in that estimate. This avoids estimating $\mathsf M^{-1}$ and $Y'(w)^*$ separately.

Let $w^0(t):=(\alpha_0,z_0+t)$ denote the uncontrolled solution. 
As long as $\|w(t)-w^0(t)\|\le\rho$, we have $\|\alpha(t)-\alpha_0\|\le\rho$ and $|z(t)|\le Z$.
Integration of \eqref{eq:ROM_as_explicit_ODE} therefore shows that
\begin{equation}\label{eq:distance_bound}
 \|w(t)-w^0(t)\|
 =\left\|\int_0^tG(w(s))u(s)\,\mathrm ds\right\|
 \le L\int_0^t\|u(s)\|_U\,\mathrm ds
 \le L\sqrt T\|u\|_{L^2(0,T;U)}<\rho.
\end{equation}
If there were a first exit time from the closed $\rho$-neighborhood of $w^0(t)$, continuity would give equality with $\rho$ at that time, contradicting \eqref{eq:distance_bound}. 
Thus no such exit occurs on the interval on which the local solution exists. Moreover, on the interval of existence,
\begin{equation}\label{eq:derivative_bounds}
 \|\dot\alpha(t)\|\le L\|u(t)\|_U,\qquad
 |\dot z(t)|\le 1+L\|u(t)\|_U.
\end{equation}
If the maximal right endpoint were some $\tilde t<T$, \eqref{eq:derivative_bounds} would imply, for $0\le s<t<\tilde t$,
\[
 \|w(t)-w(s)\|\le (t-s)+L\int_s^t\|u(\tau)\|_U\,\mathrm d\tau.
\]
Since $u\in L^2(0,T;U)\subset L^1(0,T;U)$, the right-hand side tends to zero as $s,t$ tend to $\tilde t$. Hence $w(t)$ has a finite limit. The strict margin $\rho-L\sqrt T\|u\|_{L^2(0,T;U)}>0$ in \eqref{eq:distance_bound} places this limit inside the neighborhood, where \eqref{eq:mass_lower_bound} ensures that the mass matrix is positive definite. The local existence theorem at this limit continues the solution past $\tilde t$, a contradiction. Thus the solution exists on $[0,T)$. The same estimate gives its continuous extension to $T$, and \eqref{eq:derivative_bounds} gives square-integrable derivatives. This proves existence on $[0,T]$. Uniqueness follows from local uniqueness of the explicit ODE \eqref{eq:ROM_as_explicit_ODE}.

For $\|u\|_{L^2(0,T;U)}\le R<\varepsilon_*$, the estimates \eqref{eq:distance_bound} and \eqref{eq:derivative_bounds} give
\begin{align}\label{eq:general_state_bounds}
 \|\alpha\|_{L^\infty(0,T;\mathbb R^{\tilde\ell})}&\le\|\alpha_0\|+L\sqrt T R, &
 \|z\|_{L^\infty(0,T;\mathbb R)}&\le |z_0|+T+L\sqrt T R,\nonumber\\
 \|\dot\alpha\|_{L^2(0,T;\mathbb R^{\tilde\ell})}&\le LR, &
 \|\dot z\|_{L^2(0,T;\mathbb R)}&\le \sqrt T+LR.
\end{align}
Together with $\|h\|_{L^2(0,T;X)}\le\sqrt T\|h\|_{L^\infty(0,T;X)}$ for $X=\mathbb R^{\tilde\ell}$ or $X=\mathbb R$, \eqref{eq:general_state_bounds} proves \eqref{eq:a_priori_bnds}.

We now prove the simpler result for a unitary group. In this case the additional independence assumption ensures invertibility of the mass matrix for every nonzero amplitude, so it suffices to keep $\|\alpha(t)\|$ away from zero. We first obtain an estimate for this norm directly from the first block equation of \eqref{eq:group_ROM}. The second block equation will then give bounds on the transformation parameter and the time derivatives. For the unitary case, the coefficient matrices satisfy $M_1=I_{\tilde\ell}$, $N=V^*W=-N^\top$ and $M_2=W^*W$, independently of $z$. 

Under the additional independence assumption specified in the theorem, the Gramian and its Schur complement
\begin{equation}\label{eq:mass_mat_split}
 \mathsf M_c:=\begin{bmatrix}I_{\tilde\ell}&N\\N^\top&M_2\end{bmatrix}>0,
 \qquad H_c:=M_2-N^\top N>0
\end{equation}
are positive definite. Set $\kappa:=\lambda_{\min}(H_c)>0$. The factorization
\[
 \mathsf M(w)=\begin{bmatrix}I_{\tilde\ell}&0\\0&\alpha^\top\end{bmatrix}
 \mathsf M_c\begin{bmatrix}I_{\tilde\ell}&0\\0&\alpha\end{bmatrix}
\]
shows that $\mathsf M(w)$ is positive definite whenever $\alpha\ne0$.
We proceed by showing that $(\alpha,z)$ remain bounded from above and that $\alpha$ remains bounded from below on $[0,T]$.

We begin by multiplying the first block equation in \eqref{eq:group_ROM} with $\alpha^\top$ which, since we have $N=-N^\top$ leads to
\begin{align}\label{eq:aux1}
    \frac{1}{2}\frac{\mathrm{d}}{\mathrm{d}s} \|\alpha(s)\|^2=\alpha(s)^\top \dot{\alpha}(s)=\alpha(s)^\top B_1(z(s))u(s).
\end{align}
The transformed modes are orthonormal, so for $z \in \mathbb{R}, u\in U $, Bessel's inequality gives
\begin{equation}\label{eq:bound_B1}
 \|B_1(z)u\|^2=\sum_{i=1}^{\tilde\ell}|\langle\mathcal T(z)\phi_i,\mathcal Bu\rangle_H|^2
 \le\|\mathcal Bu\|_H^2\le b^2\|u\|_U^2.
\end{equation}
On any compact subinterval of existence on which $\alpha\ne0$, continuity gives a positive lower bound for $\|\alpha\|$. 
Together with \eqref{eq:aux1}, this yields, for almost every $t$,
\[
 \frac{\mathrm d}{\mathrm dt}\|\alpha(t)\|
 =\frac{1}{2\|\alpha(t)\|}\frac{\mathrm d}{\mathrm dt}\|\alpha(t)\|^2
 =\frac{\alpha(t)^\top B_1(z(t))u(t)}{\|\alpha(t)\|}.
\]
cf.~\cite[Corollary 8.11]{Bre11}.
The Cauchy-Schwarz inequality and \eqref{eq:bound_B1} therefore imply
\[
 \left|\frac{\mathrm d}{\mathrm dt}\|\alpha(t)\|\right|
 \le \frac{\|\alpha(t)\|\,\|B_1(z(t))u(t)\|}{\|\alpha(t)\|}
 \le b\|u(t)\|_U.
\]
Integrating this estimate over $[0,t]$ and using the Cauchy-Schwarz inequality in time gives
\begin{equation}\label{eq:linf_bnds_alpha}
 \|\alpha_0\|-b\sqrt T\|u\|_{L^2(0,T;U)}
 \le\|\alpha(t)\|\le\|\alpha_0\|+b\sqrt T\|u\|_{L^2(0,T;U)}.
\end{equation}
The lower bound is strictly positive under \eqref{eq:unitary_threshold}. 
Thus, by continuity, $\alpha$ cannot become zero in the considered time interval.

To estimate $z$, we eliminate $\dot\alpha$ from the two block equations of \eqref{eq:group_ROM}. Let $P_V:=VV^*$ be the orthogonal projection onto the mode space and set $h(\alpha):=(I_H-P_V)W\alpha$. Then
\[
 g(\alpha):=\alpha^\top H_c\alpha=\|h(\alpha)\|_H^2\ge\kappa\|\alpha\|^2.
\]
The two block equations in \eqref{eq:group_ROM} read
\begin{equation}\label{eq:unitary_block_equations}
\begin{aligned}
 \dot\alpha&=N\alpha(1-\dot z)+B_1(z)u,\\
 \alpha^\top N^\top\dot\alpha+\alpha^\top M_2\alpha\dot z
 &=\alpha^\top M_2\alpha+\alpha^\top B_2(z)u.
\end{aligned}
\end{equation}
Substituting the first equation into the second and collecting the terms containing $\dot z-1$ yields
\begin{align}\label{eq:deriv_z}
 g(\alpha)(\dot z-1)
 &=\alpha^\top(B_2-N^\top B_1)u
 =\langle\mathcal T(z)h(\alpha),\mathcal Bu\rangle_H.
\end{align}
For the last equality in \eqref{eq:deriv_z}, the definitions in \eqref{eq:ROM_B_matrices} give
\[
 B_1(z)=V^*\mathcal T(z)^*\mathcal B,\qquad
 B_2(z)=W^*\mathcal T(z)^*\mathcal B,
 \qquad N^\top=W^*V.
\]
Consequently,
\begin{align*}
 \alpha^\top(B_2(z)-N^\top B_1(z))u
 &=\alpha^\top W^*(I_H-VV^*)\mathcal T(z)^*\mathcal Bu\\
 &=\langle (I_H-P_V)W\alpha,\mathcal T(z)^*\mathcal Bu\rangle_H\\
 &=\langle\mathcal T(z)h(\alpha),\mathcal Bu\rangle_H.
\end{align*}
Here we used the self-adjointness of $P_V$ and the definition of the adjoint of $\mathcal T(z)$.
Unitarity and the Cauchy-Schwarz inequality imply that the absolute value of the right-hand side is at most $\sqrt{g(\alpha)}\,b\|u\|_U$. Dividing by $g(\alpha)$ cancels one factor $\sqrt{g(\alpha)}$, so
\[
 |\dot z-1|\le\frac{b\|u\|_U}{\sqrt{g(\alpha)}}
 \le\frac{b\|u\|_U}{\sqrt\kappa\|\alpha\|}.
\]
For $\|u\|_{L^2(0,T;U)}\le R<\varepsilon_{\mathrm{un}}$, put $r_-:=\|\alpha_0\|-b\sqrt T R>0$ and $r_+:=\|\alpha_0\|+b\sqrt T R$. The estimates \eqref{eq:linf_bnds_alpha}, \eqref{eq:deriv_z}, and \eqref{eq:unitary_block_equations} give
\begin{align}\label{eq:unitary_state_bounds}
 &|z(t)|=\left|z_0+t+\int_0^t(\dot z(s)-1)\,\mathrm ds\right|
 \le |z_0|+T+\frac{b\sqrt T R}{\sqrt\kappa r_-},\\
 &\|\dot z\|_{L^2(0,T;\mathbb R)}\le \sqrt T+\frac{bR}{\sqrt\kappa r_-},\\
 &\|\dot\alpha(t)\|=\|N\alpha(t)(1-\dot z(t))+B_1(z(t))u(t)\|
 \le b\left(1+\frac{\|N\|}{\sqrt\kappa}\right)\|u(t)\|_U.
\end{align}
The bounds \eqref{eq:unitary_state_bounds} give a finite endpoint limit, while \eqref{eq:linf_bnds_alpha} keeps its amplitude nonzero. The continuation argument used above therefore proves existence on $[0,T]$, and the bounds \eqref{eq:unitary_state_bounds}, together with \eqref{eq:linf_bnds_alpha}, prove \eqref{eq:a_priori_bnds} also in this case.
\end{proof}

\begin{rmrk}\label{rmrk:smallness_cond}
    While the smallness condition might look restrictive, let us emphasize that it is trivially satisfied for $u\equiv0$. Let $\varepsilon$ denote an applicable threshold from Theorem~\ref{thrm:unique_solution}, and fix $0<R<\varepsilon$. The uncontrolled reduced state is $w^0(t)=(\alpha_0,z_0+t)$, and we set
\begin{equation}\label{eq:zero_cost}
 J_0:=\mathcal J_{\scaleto{\mathrm{tPOD-G}}{3pt}}(w^0,0)<\infty.
\end{equation}
For every feasible triplet we have $\frac\mu2\|u\|_{L^2(0,T;U)}^2\le\mathcal J_{\scaleto{\mathrm{tPOD-G}}{3pt}}(\alpha,z,u)$. Hence, if $\mu R^2/2>J_0$, no feasible control with $\|u\|_{L^2(0,T;U)}\ge R$ can improve upon the uncontrolled solution. 
We may therefore minimize on the closed ball of radius $R$, which is weakly closed and lies strictly inside the small-control region \cite[Theorem 2.11]{troltzsch_optimal_2010}.
This also gives a uniform margin from the critical threshold for every minimizing sequence in that ball.
\end{rmrk}

\begin{thrm}\label{thrm:control_existence}
Let Assumption~\ref{asmp:general_assumption} hold. 
Then, for given $y_{\mathrm d}\in L^2(0,T;H)$ and $0<R<\varepsilon$, $\mu>2J_0/R^2$ where $\varepsilon, J_0$ are as given in \Cref{rmrk:smallness_cond}, there exists an optimal control $\bar u$ to \eqref{eq_def:sPOD_costFunc_continuous}. 
It satisfies $\|\bar u\|_{L^2(0,T;U)}<R$. 
\end{thrm}
\begin{proof}
We use the direct method on the closed control ball specified in Remark~\ref{rmrk:smallness_cond}.
From Theorem \ref{thrm:unique_solution} we know that for $u= 0$ there exists a unique solution $(\alpha,z)$ to \eqref{eq:group_ROM}.
From the a priori bounds \eqref{eq:a_priori_bnds} we conclude that $u= 0$ is feasible for \eqref{eq_def:sPOD_costFunc_continuous}.
Consider first the closed ball $\mathcal U_R:=\{u\in L^2(0,T;U)\,|\, \|u\|_{L^2(0,T;U)}\le R\}$. Since $0\in\mathcal U_R$ and the cost is nonnegative, its infimum on this ball is finite.
By abuse of notation, we use $(\alpha(u),z(u))$ in the following to denote the unique solution for a given admissible input $u$.
 Let $\{u_n\}\subset \mathcal U_R$ denote a minimizing sequence for \eqref{eq_def:sPOD_costFunc_continuous}, i.e., 
 \begin{equation*}
     \lim_{n\to\infty}\mathcal{J}_{\scaleto{\mathrm{tPOD-G}}{3pt}}(\alpha_n, z_n, u_n) = \inf_{u \in \mathcal U_R}\underbrace{\mathcal{J}_{\scaleto{\mathrm{tPOD-G}}{3pt}}(\alpha(u), z(u), u)}_{=\vcentcolon J(u)}
 \end{equation*}
with $(\alpha_n,z_n) \vcentcolon=(\alpha(u_n),z(u_n))$ bounded in $H^1(0,T;\mathbb R^{\tilde{\ell}})\times H^1(0,T)$.
    Subsequently, we get
    \begin{equation}\label{eq:weak_conv}
    \begin{aligned}
        u_n &\rightharpoonup \bar{u} \quad \text{in} \; L^2(0, T;U), \quad 
        \alpha_n \rightharpoonup \bar{\alpha} \quad \text{in} \; H^1(0, T;\mathbb{R}^{\tilde{\ell}}),\quad 
        z_n \rightharpoonup \bar{z} \quad \text{in} \; H^1(0, T)
    \end{aligned}
    \end{equation}
    for subsequences denoted in the same manner, see e.g.~Theorem~3 in \cite[App.~D.4]{Eva97}.
    In addition, since the mappings $x\mapsto x$ and $x\mapsto \dot{x}$ from $H^1(0,T)$ to $L^2(0,T)$ are continuous, we also get $\alpha_n \rightharpoonup \bar{\alpha}$, $\dot{\alpha}_n \rightharpoonup \dot{\bar{\alpha}}$ in $L^2(0, T;\mathbb{R}^{\tilde{\ell}})$ as well as $z_n \rightharpoonup \bar{z}$, $\dot{z}_n \rightharpoonup \dot{\bar{z}}$ in $L^2(0, T)$, cf.~\cite[Thm.~3.10]{Bre11}.
    From classical Sobolev embeddings, cf.~\cite[Thm.~8.8]{Bre11}, we obtain $H^1(0, T; \mathbb{R}^d){\hookrightarrow} L^\infty(0,T;\mathbb R^d)$ as well as 
    $H^1(0, T; \mathbb{R}^d)\overset{\mathrm{compact}}{\hookrightarrow} C([0, T];\mathbb{R}^d)$, which implies
    \begin{align}\label{eq:strong_conv}
            \alpha_n &\rightarrow \bar{\alpha} \quad \text{in} \; C([0, T];\mathbb{R}^{\tilde{\ell}}), \quad 
        z_n \rightarrow \bar{z} \quad \text{in} \; C([0, T]),
        \end{align}
        see e.g.~\cite[Thm.~8.1-7]{Kre78}.
Weak lower semicontinuity of the norm gives $\|\bar u\|_{L^2(0,T;U)}\le R<\varepsilon$, so Theorem~\ref{thrm:unique_solution} applies also to the limiting control.
Let us show that $(\bar\alpha,\bar z)=(\alpha(\bar u),z(\bar u))$.
Write $w_n=(\alpha_n,z_n)$ and $\bar w=(\bar\alpha,\bar z)$. 
By \eqref{eq:strong_conv}, all the values of $w_n$ and $\bar w$ are contained in a common compact subset of $\mathbb R^{\tilde\ell+1}$. Continuity of the coefficient matrices implies uniform convergence $\mathsf M(w_n)\to \mathsf M(\bar w)$ and $\mathsf B(w_n)\to\mathsf B(\bar w)$. For arbitrary $q\in L^2(0,T;\mathbb R^{\tilde\ell+1})$, split the term involving the mass matrix as
\begin{align*}
 &\langle \mathsf M(w_n)\dot w_n-\mathsf M(\bar w)\dot{\bar w},q\rangle_{L^2(0,T)}\\
 &\quad=\langle(\mathsf M(w_n)-\mathsf M(\bar w))\dot w_n,q\rangle_{L^2(0,T)}
       +\langle\dot w_n-\dot{\bar w},\mathsf M(\bar w)^\top q\rangle_{L^2(0,T)}.
\end{align*}
The absolute value of the first term is bounded by
\[
 \|\mathsf M(w_n)-\mathsf M(\bar w)\|_{L^\infty(0,T)}\|\dot w_n\|_{L^2(0,T;\mathbb R^{\tilde\ell+1})}\|q\|_{L^2(0,T;\mathbb R^{\tilde\ell+1})}\longrightarrow0.
\]
The second term tends to zero by \eqref{eq:weak_conv}, since the derivatives converge weakly in $L^2(0,T;\mathbb R^{\tilde\ell+1})$ and the fixed test function $\mathsf M(\bar w)^\top q$ belongs to that space. Indeed, its norm is bounded by the uniform bound on $\mathsf M(\bar w)$ times $\|q\|_{L^2(0,T;\mathbb R^{\tilde\ell+1})}$.
For the control term, we use the analogous decomposition
\begin{align*}
 &\langle\mathsf B(w_n)u_n-\mathsf B(\bar w)\bar u,q\rangle_{L^2(0,T)}\\
 &\quad=\langle(\mathsf B(w_n)-\mathsf B(\bar w))u_n,q\rangle_{L^2(0,T)}
       +\langle u_n-\bar u,\mathsf B(\bar w)^\top q\rangle_{L^2(0,T)}\longrightarrow0.
\end{align*}
The first term is bounded by
\[
 \|\mathsf B(w_n)-\mathsf B(\bar w)\|_{L^\infty(0,T)}
 \|u_n\|_{L^2(0,T;U)}\|q\|_{L^2(0,T;\mathbb R^{\tilde\ell+1})}.
\]
This tends to zero because the matrix coefficients converge uniformly and the minimizing sequence is bounded. The second term tends to zero by weak convergence of $u_n$, since $\mathsf B(\bar w)^\top q\in L^2(0,T;U)$. Finally,
\[
 \|(\mathsf M(w_n)-\mathsf M(\bar w))e_z\|_{L^2(0,T;\mathbb R^{\tilde\ell+1})}
 \le\sqrt T\|\mathsf M(w_n)-\mathsf M(\bar w)\|_{L^\infty(0,T)}\|e_z\|\longrightarrow0.
\]
The drift term therefore also converges in the space in which the state equation is tested. 
Thus we can pass to the limit in \eqref{eq:group_ROM}, and also the initial values converge due to \eqref{eq:strong_conv}. 
The limiting control lies in the small-control ball, so Theorem~\ref{thrm:unique_solution} gives uniqueness for this initial value problem. We therefore arrive at $(\bar\alpha,\bar z)=(\alpha(\bar u),z(\bar u))$.
For completeness, the reconstructed states converge strongly as well. On a compact convex set containing the values of $w_n$ and $\bar w$, the map $Y$ in \eqref{eq:tangent_map} has a bounded derivative. Hence there exists $C>0$ such that
\[
 \|Y(w_n)-Y(\bar w)\|_{L^2(0,T;H)}
 \le\sqrt T C\|w_n-\bar w\|_{C([0,T];\mathbb R^{\tilde\ell+1})}\longrightarrow0.
\]
This justifies continuity of the tracking term in the following step.

Finally, due to the continuity of the first term in the cost function, the strong convergence of $\alpha_n$ and $z_n$ in $C([0,T];\mathbb{R}^{\tilde\ell})$ and $C([0,T])$, respectively, and the sequential weak lower semicontinuity of norms, cf.~\cite[Prop.~21.23(c)]{Zei90}, we arrive at
    \begin{equation*}
        \mathcal{J}_{\scaleto{\mathrm{tPOD-G}}{3pt}}(\bar{\alpha}, \bar{z}, \bar{u}) \leq \underset{n \rightarrow \infty}{\mathrm{lim \;inf}}\;\; \mathcal{J}_{\scaleto{\mathrm{tPOD-G}}{3pt}}(\alpha_n, z_n, u_n)
    \end{equation*}
This proves existence on $\mathcal U_R$. Since its minimum is at most $J_0$, we have $\|\bar u\|_{L^2(0,T;U)}^2\le2J_0/\mu<R^2$. Every feasible triplet with $\|u\|_{L^2(0,T;U)}\ge R$ has cost at least $\mu R^2/2>J_0$. The minimizer on $\mathcal U_R$ is therefore also a minimizer of the unrestricted problem.
\end{proof}
\begin{rmrk}
The choice of $\mu$ can be made entirely a priori. Indeed, $\|a-b\|^2\le2\|a\|^2+2\|b\|^2$ and \eqref{eq:group_bound} yield
\[
 J_0\le\overline J_0:=C_{\mathcal T}^2e^{2\omega_{\mathcal T}|z_0|}\|\alpha_0\|^2
       \int_0^T e^{2\omega_{\mathcal T} t}\,\mathrm dt+\|y_{\mathrm d}\|_{L^2(0,T;H)}^2.
\]
Taking $R=\varepsilon/2$, for example, it suffices to choose $\mu>8\overline J_0/\varepsilon^2$.
\end{rmrk}

Combining Theorem~\ref{thrm:unique_solution} and Theorem~\ref{thrm:control_existence}, we can rely on standard optimality results from, e.g., \cite{hinze_optimization_2009} to finally derive the optimality system for the tPOD-G method.

\begin{thrm}\label{thm:first_order_cond}
    Let the assumptions of Theorem~\ref{thrm:control_existence} be satisfied, $(\bar{\alpha}, \bar{z}, \bar{u})$ be a local minimizer of the optimal control problem \eqref{eq_def:sPOD_costFunc_continuous} with $\|\bar u\|_{L^2(0,T;U)}<\varepsilon$, and the modes $\phi_1,\ldots,\phi_{\tilde\ell}$ be in $\mathcal D(\mathcal A^2)$.
    Then, 
    there exists a unique adjoint $(\lambda^{\ell_a}, z^{\ell_a}) \in H^1(0, T; \mathbb{R}^{\tilde{\ell}}) \times H^1(0, T; \mathbb{R})$, such that the following first order optimality conditions hold
\begin{subequations}
\begin{align}
 \mathsf M(\bar w)\dot{\bar w}=\mathsf M(\bar w)e_z+\mathsf B(\bar w)\bar u,\label{eq_def:OC_FRTO_sPODG_1}\\
 \bar\alpha(0)=\alpha_0,\qquad\bar z(0)=z_0,\label{eq_def:OC_FRTO_sPODG_2}\\
 \mathsf M(\bar w)\dot\lambda_{\mathrm r}=\mathsf E(\bar w,\dot{\bar w},\bar u)\lambda_{\mathrm r}-f_{\mathrm r}(t,\bar w),\label{eq_def:OC_FRTO_sPODG_3}\\
 \lambda_{\mathrm r}(T)=0,\label{eq_def:OC_FRTO_sPODG_4}\\
 \mu\bar u+B_1(\bar z)^\top\lambda^{\ell_a}+B_2(\bar z)^\top\bar\alpha z^{\ell_a}=0,\label{eq_def:OC_FRTO_sPODG_5}
\end{align}
\end{subequations}
where $w:=(\alpha,z),\bar w:=(\bar\alpha,\bar z), \lambda_{\mathrm r}:=(\lambda^{\ell_a},z^{\ell_a})$
and
\begin{align*}
 y_{\mathrm d,1}(t,z)&:=[\langle\mathcal T(z)\phi_i,y_{\mathrm d}(t)\rangle_H]_{i=1}^{\tilde\ell}, \quad y_{\mathrm d,2}(t,z):=[\langle\mathcal T'(z)\phi_i,y_{\mathrm d}(t)\rangle_H]_{i=1}^{\tilde\ell},\\
 f_{\mathrm r}(t,w)&:=\begin{bmatrix}M_1(z)\alpha-y_{\mathrm d,1}(t,z)\\
 \alpha^\top N(z)^\top\alpha-\alpha^\top y_{\mathrm d,2}(t,z)\end{bmatrix}
 =Y'(w)^*(Y(w)-y_{\mathrm d}(t)),
\end{align*}
\begin{equation}
\begin{aligned}\label{eq:coefficients_of_adjoint_equation}
E_{11}(w,\dot w,u)&=-M_1'\dot{z}-(1-\dot{z})N^\top,\\
 E_{12}(w,\dot w,u)&=2(\dot z-1)M_2\alpha-B_2u
          +(N^\top-N)\dot\alpha-N'\alpha\dot z,\\
 E_{21}(w,\dot w,u)&=\dot{\alpha}^\top(M_1'-N^\top)
          -\alpha^\top(N')^\top-u^\top B_2^\top,\\
 E_{22}(w,\dot w,u)&=-2\alpha^\top M_2\dot\alpha-\alpha^\top M_2'\alpha
          -\alpha^\top B'_2u+\alpha^\top(N')^\top\dot\alpha.
\end{aligned}
\end{equation}
In the unitary case, we have $M_1=I_{\tilde\ell}$, $M_1'=N'=M_2'=0$ and $N^\top=-N$.
\end{thrm}

\begin{proof}
In view of the results in Theorem \ref{thrm:unique_solution} and Theorem \ref{thrm:control_existence}, the statement follows with standard arguments as given in, e.g., \cite[Section 1.7.2]{hinze_optimization_2009}. Denoting $w=(\alpha,z)$, let us consider \eqref{eq_def:sPOD_costFunc_continuous} in the form
\[
\min_{(w,u)\in\mathcal W\times\mathcal U}\mathcal J_{\scaleto{\mathrm{tPOD-G}}{3pt}}(w,u)
 \quad\text{subject to }e(w,u)=0,
\]
where $e\colon \mathcal W\times\mathcal U\to\mathcal Z$ and $e(w,u)=0$ represents \eqref{eq_def:OC_FRTO_sPODG_1}-\eqref{eq_def:OC_FRTO_sPODG_2}. In particular, let us consider
$\mathcal W=H^1(0,T;\mathbb R^{\tilde\ell+1})$, $\mathcal U=L^2(0,T;U)$, $\mathcal Z=L^2(0,T;\mathbb R^{\tilde\ell+1})\times\mathbb R^{\tilde\ell+1}$ and
\[
 e(w,u):=\big(\mathsf M(w)e_z+\mathsf B(w)u-\mathsf M(w)\dot w,\ w(0)-(\alpha_0,z_0)\big).
\]
To verify the differentiability used below and compute the linearized state equation, we first differentiate the coefficient matrices introduced in Section~\ref{ssec:sPODG}. With the notation of the theorem, this gives
\begin{align}\label{eq:group_coeff_derivatives}
 M_1'&=N+N^\top,\qquad B_1'=B_2,\qquad
 [B'_2]_{ij}=\langle\mathcal T(z)\mathcal A^2\phi_i,\mathcal Be_j\rangle_H,\nonumber\\
 [N']_{ij}&=[M_2]_{ij}+\langle\mathcal T(z)\phi_i,\mathcal T(z)\mathcal A^2\phi_j\rangle_H,\nonumber\\
 [M_2']_{ij}&=\langle\mathcal T(z)\mathcal A^2\phi_i,\mathcal T(z)\mathcal A\phi_j\rangle_H
 +\langle\mathcal T(z)\mathcal A\phi_i,\mathcal T(z)\mathcal A^2\phi_j\rangle_H.
\end{align}
These derivatives are continuous. They are bounded on bounded $z$-intervals.

Let us show the continuous Fréchet differentiability of $e$ and the cost functional $\mathcal J_{\scaleto{\mathrm{tPOD-G}}{3pt}}$ on $\mathcal{W}\times \mathcal{U}$. 
By
\eqref{eq:tangent_map} and \eqref{eq:group_coeff_derivatives} the maps
$w\mapsto Y(w)$, $w\mapsto\mathsf M(w)$ and $w\mapsto\mathsf B(w)$ are
continuously differentiable on $\mathbb R^{\tilde\ell+1}$. Hence the induced
superposition operators are continuously Fr\'echet differentiable from
$C([0,T];\mathbb R^{\tilde\ell+1})$ into $C([0,T];H)$,
$C([0,T];\mathbb R^{(\tilde\ell+1)\times(\tilde\ell+1)})$ and
$C([0,T];\mathbb R^{(\tilde\ell+1)\times m})$, respectively. Moreover,
$\mathcal W\hookrightarrow C([0,T];\mathbb R^{\tilde\ell+1})$, the map
$w\mapsto\dot w$ is bounded and linear from $\mathcal W$ into
$L^2(0,T;\mathbb R^{\tilde\ell+1})$, and the products
$(\mathsf M,\dot w)\mapsto\mathsf M\dot w$ and $(\mathsf B,u)\mapsto\mathsf Bu$
are bounded bilinear maps from
$C([0,T];\mathbb R^{(\tilde\ell+1)\times(\tilde\ell+1)})\times L^2(0,T;\mathbb R^{\tilde\ell+1})$
and $C([0,T];\mathbb R^{(\tilde\ell+1)\times m})\times L^2(0,T;U)$, respectively,
into $L^2(0,T;\mathbb R^{\tilde\ell+1})$. Differentiability of $e$ therefore follows from the
chain rule.
The cost functional is continuously differentiable by the corresponding argument for $Y(w)$ and the quadratic norms.

The partial derivative of $e$ w.r.t.~$w$ evaluated at $(\bar w,\bar u)$ is given by
\begin{equation}\label{eq:linearized_group_residual}
 \partial_w e(\bar w,\bar u)\xi=(\Gamma\xi-\mathsf M(\bar w)\dot\xi,\xi(0)),
\end{equation}
where $\Gamma=\Gamma(t)$ depends on $\bar w(t),\dot{\bar w}(t),\bar u(t)$ and is given by
\begin{align*}
 \Gamma_{11}&=(1-\dot{\bar z})N,\\
 \Gamma_{12}&=(1-\dot{\bar z})N'\bar\alpha+B_2\bar u-M_1'\dot{\bar\alpha},\\
 \Gamma_{21}&=2(1-\dot{\bar z})\bar\alpha^\top M_2
                  +\bar u^\top B_2^\top-\dot{\bar\alpha}^\top N,\\
 \Gamma_{22}&=(1-\dot{\bar z})\bar\alpha^\top M_2'\bar\alpha
                  +\bar\alpha^\top B'_2\bar u-\bar\alpha^\top(N')^\top\dot{\bar\alpha}.
\end{align*}
Since $\bar w$ is continuous, the coefficient matrices evaluated at $\bar z$ and
all factors involving $\bar\alpha$ are bounded, and every term of $\Gamma$
contains at most one of the factors $\dot{\bar\alpha}$, $\dot{\bar z}$,
$\bar u$, which belong to $L^2$. Hence $\Gamma\in L^2(0,T;\mathbb{R}^{(\tilde\ell+1)\times(\tilde\ell+1)})$. Moreover, by the
proof of Theorem~\ref{thrm:unique_solution}, $\mathsf M(\bar w)^{-1}$ is bounded
on $[0,T]$. For given $(g,\eta)\in\mathcal Z$, the equation
$\partial_w e(\bar w,\bar u)\xi=(g,\eta)$ is therefore equivalent to the linear
initial value problem
\[
 \dot\xi=\mathsf M(\bar w)^{-1}\Gamma\,\xi-\mathsf M(\bar w)^{-1}g,
 \qquad \xi(0)=\eta,
\]
whose coefficient and inhomogeneity are square-integrable. It has a unique
solution $\xi\in\mathcal W$, and Gr\"onwall's inequality yields
$\|\xi\|_{\mathcal W}\le C\big(\|g\|_{L^2(0,T;\mathbb R^{\tilde\ell+1})}+\|\eta\|\big)$
with a constant $C$ independent of $(g,\eta)$. Thus
$\partial_w e(\bar w,\bar u)$ has a bounded inverse.

Hence, there exists an adjoint state $\bar p=(\lambda_{\mathrm r},\eta_0)\in\mathcal Z^*=L^2(0,T;\mathbb R^{\tilde\ell+1})\times\mathbb R^{\tilde\ell+1}$ such that
\[
 \partial_w e(\bar w,\bar u)^*\bar p
 =-\partial_w\mathcal J_{\scaleto{\mathrm{tPOD-G}}{3pt}}(\bar w,\bar u).
\]
In the notation of the theorem, this identity means
\begin{equation}\label{eq:adjoint_weak}
 \int_0^T\big(f_{\mathrm r}(t,\bar w)^\top\xi+\lambda_{\mathrm r}^\top\Gamma\xi
             -\lambda_{\mathrm r}^\top \mathsf M(\bar w)\dot\xi\big)\,\mathrm dt
             +\eta_0^\top\xi(0)=0\qquad\text{for }\xi\in\mathcal W.
\end{equation}
So far, we only have $\lambda_{\mathrm r}\in L^2(0,T;\mathbb{R}^{\tilde\ell+1})$. 
To derive the differential equation \eqref{eq_def:OC_FRTO_sPODG_3} and the endpoint condition \eqref{eq_def:OC_FRTO_sPODG_4}, we first establish additional regularity. For this, it is convenient to introduce $p:=\mathsf M(\bar w)\lambda_{\mathrm r}\in L^2(0,T;\mathbb R^{\tilde\ell+1})$. 
Testing with smooth compactly supported $\xi$ gives
\begin{equation}
    \label{eq:p_dot}
    \dot p=-\Gamma^\top\lambda_{\mathrm r}-f_{\mathrm r}(t,\bar w)
\end{equation}
in the sense of distributions. 
The right-hand side belongs to $L^1(0,T;\mathbb{R}^{\tilde\ell+1})$, since $\Gamma$, $\lambda_{\mathrm r}$, and $f_{\mathrm r}$ are square-integrable. 
Hence $p$ is in $W^{1,1}(0,T;\mathbb{R}^{\tilde\ell+1})$ and thus bounded.
As $\mathsf M^{-1}$ is bounded, $\lambda_{\mathrm r}=\mathsf M^{-1}p$ is bounded too. 
Therefore, the right-hand side of \eqref{eq:p_dot} is in $L^2(0,T;\mathbb{R}^{\tilde\ell+1})$, which implies $p\in H^1(0,T;\mathbb{R}^{\tilde\ell+1})$. 
Moreover, $\mathsf M(\bar w)$ is in $H^1(0,T;\mathbb{R}^{(\tilde\ell+1)\times(\tilde\ell+1)})$ by the chain rule, and
$\frac{\mathrm{d}}{\mathrm{d}t}(\mathsf M(\bar{w})^{-1})=-\mathsf M^{-1}\frac{\mathrm{d}}{\mathrm{d}t}(\mathsf M(\bar{w})) \mathsf M^{-1}\in L^2(0,T;\mathbb{R}^{(\tilde\ell+1)\times(\tilde\ell+1)})$. 
The product rule therefore proves $\lambda_{\mathrm r}=\mathsf M^{-1}p\in H^1(0,T;\mathbb{R}^{\tilde\ell+1})$.

We can now integrate \eqref{eq:adjoint_weak} by parts. Its boundary terms are $-p(T)^\top\xi(T)+(p(0)+\eta_0)^\top\xi(0)$. Since both endpoint values of $\xi$ can be chosen freely, $p(T)=0$ and $\eta_0=-p(0)$. The invertibility of $\mathsf M(\bar w(T))$ gives $\lambda_{\mathrm r}(T)=0$. The differential equation for $p$ becomes
\[
 \mathsf M(\bar w)\dot\lambda_{\mathrm r}=\left(- \frac{\mathrm{d}}{\mathrm{d}t}(M(\bar{w}))-\Gamma^\top \right)\lambda_{\mathrm r}-f_{\mathrm r}(t,\bar w).
\]
which coincides with \eqref{eq_def:OC_FRTO_sPODG_3}.

Finally, since $\|\bar u\|_{L^2(0,T;U)}<\varepsilon$, all sufficiently small $L^2(0,T;U)$ variations remain inside the small-control region, and we also have that
\[
 \partial_u\mathcal J_{\scaleto{\mathrm{tPOD-G}}{3pt}}(\bar w,\bar u)
 +\partial_u e(\bar w,\bar u)^*\bar p=0,
\]
which translates into
\begin{equation*}
\mu \bar{u} + B_1(\bar{z})^\top\lambda^{\ell_a} + B_2(\bar{z})^\top\bar{\alpha}z^{\ell_a} = 0.
\end{equation*}
\end{proof}

\begin{rmrk}\label{rem:koopman_transfer}
The analysis presented in this section is not restricted to spatial translations or to one spatial dimension. 
For example, let $f\in C^\infty(\mathbb T^d,\mathbb R^d)$ denote a smooth vector field on the $d$-dimensional torus $\mathbb T^d:=\mathbb R^d/\mathbb Z^d$ and let $\Phi_s$ denote the flow of $x'=f(x)$. 
Then, the Koopman and Perron--Frobenius (transfer) operators, cf.~\cite{mauroy_koopman_2020}, form groups on $L^2(\mathbb T^d)$ given by
\[
 (K(s)\phi)(x)=\phi(\Phi_s(x)),\qquad
 (P(s)\phi)(x)=J_{-s}(x)\phi(\Phi_{-s}(x)),\qquad J_s=\det D\Phi_s.
\]
Their generators act on smooth functions as $f\cdot\nabla\phi$ and $-\operatorname{div}(f\phi)$, respectively. The corresponding tPOD ansatz is
\begin{align*}
 Y(\alpha,z)(x)&=\sum_{i=1}^{\tilde\ell}\alpha_i\phi_i(\Phi_z(x))
 \quad\text{for } \mathcal T(z)=K(z),
 \\
 Y(\alpha,z)(x)&=J_{-z}(x)\sum_{i=1}^{\tilde\ell}\alpha_i\phi_i(\Phi_{-z}(x))
 \quad\text{for } \mathcal T(z)=P(z).
\end{align*}
\end{rmrk}

\section{Numerical results}\label{sec:results}
In the following we consider the periodic shift operator defined in \eqref{eq:periodic_shift_operator} for the transformation group and the full-order equation is the 1D constant-velocity linear advection equation. The numerical models are therefore the sPOD-G specialization of the tPOD-G framework analyzed above.
All numerical tests were run using Python 3.12 on the TU Berlin, Institute of mathematics compute cluster.  

\subsection{Problem setup}\label{ssec:prob_setup}
For problem \eqref{eq_def:PDE_continuous}, we consider a one-dimensional strip of length $l=100$ with $x\in (0, l]$ discretized into $n=3201$ grid points with $\Delta x=0.03124$ and periodic boundary conditions. 
The discretized system is simulated with initial control $u_0= 0$ on $[0, T]$ with  $T=136.2642$, considering $n_t = 2400$ time steps, $\Delta t= 0.0568$, and $v = 0.55$.
The solution obtained is shown in \Cref{fig:advection_and_target} along with the desired target profile for two different setups, where $Q$ and $Q_\mathrm{d}$ are snapshot matrices constructed by stacking discrete state and target state column-wise, respectively. 
The initial condition is given as:
\begin{equation}\label{eq_def:advection_init}
    y_0(x) := \mathrm{exp} \; \left(-\left(x - \frac{l}{12}\right)^2\right) \;.
\end{equation}
In our tests, we consider the control shape functions (first few shapes shown in \Cref{fig:Controls}) given as
\begin{equation}\label{eq_def:mask_sin_cos}
    b_1(x) = 1, \quad b_{2k}(x) = \mathrm{sin}\left(\frac{2\pi k x}{l}\right), \quad b_{2k + 1}(x) = -\mathrm{cos}\left(\frac{2\pi k x}{l}\right)
\end{equation}
for $k=1, \ldots, \xi$, where $\xi = 20$  and $m=2 \xi + 1 = 41$. 
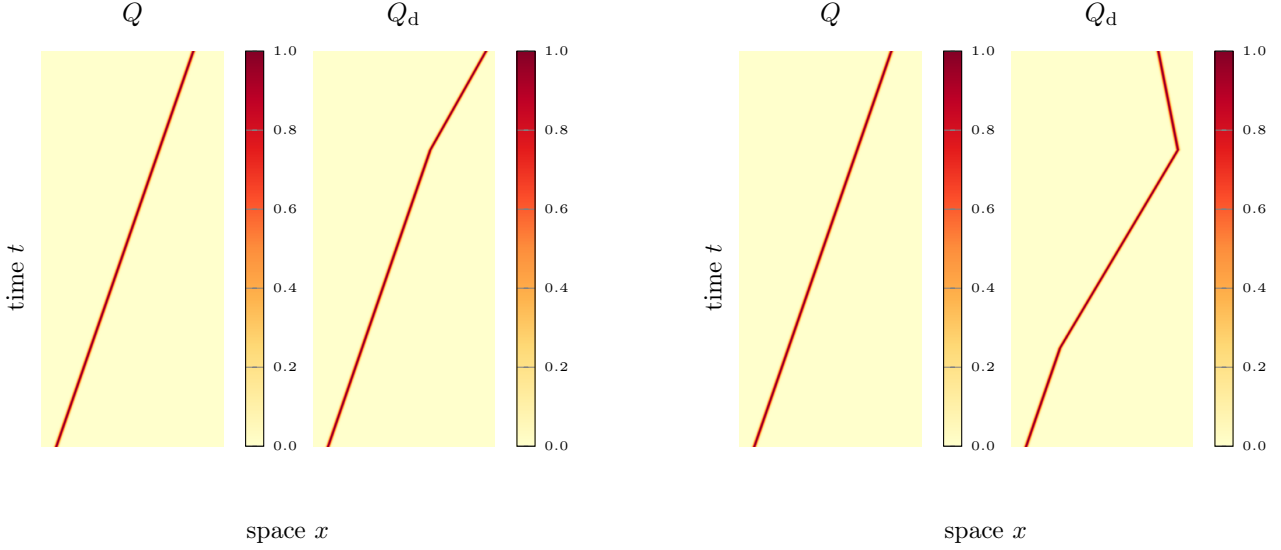
\begin{figure}[htp!]
    \centering
    \begin{subfigure}[t]{0.45\textwidth}
        \centering
        \setlength\figureheight{0.8\linewidth}
        \setlength\figurewidth{0.8\linewidth}
\begin{tikzpicture}

\definecolor{darkgray176}{RGB}{176,176,176}

\begin{groupplot}[group style={group size=2 by 1, horizontal sep=1.2cm}]
\nextgroupplot[
colorbar,
colorbar style={ytick={-0.2,0,0.2,0.4,0.6,0.8,1,1.2},yticklabels={
  \(\displaystyle {\ensuremath{-}0.2}\),
  \(\displaystyle {0.0}\),
  \(\displaystyle {0.2}\),
  \(\displaystyle {0.4}\),
  \(\displaystyle {0.6}\),
  \(\displaystyle {0.8}\),
  \(\displaystyle {1.0}\),
  \(\displaystyle {1.2}\)
},ylabel={}, width=0.1*\pgfkeysvalueof{/pgfplots/parent axis width}, ticklabel style={font=\tiny}},
colormap={mymap}{[1pt]
  rgb(0pt)=(1,1,0.8);
  rgb(1pt)=(1,0.929411764705882,0.627450980392157);
  rgb(2pt)=(0.996078431372549,0.850980392156863,0.462745098039216);
  rgb(3pt)=(0.996078431372549,0.698039215686274,0.298039215686275);
  rgb(4pt)=(0.992156862745098,0.552941176470588,0.235294117647059);
  rgb(5pt)=(0.988235294117647,0.305882352941176,0.164705882352941);
  rgb(6pt)=(0.890196078431372,0.101960784313725,0.109803921568627);
  rgb(7pt)=(0.741176470588235,0,0.149019607843137);
  rgb(8pt)=(0.501960784313725,0,0.149019607843137)
},
height=1.2 * \figureheight,
hide x axis,
hide y axis,
point meta max=1.00000201032067,
point meta min=-1.56173253764891e-05,
tick align=outside,
tick pos=left,
title={\(\displaystyle Q\)},
width=0.7 * \figurewidth,
x grid style={darkgray176},
xmin=0, xmax=800,
xtick style={color=black},
xtick={0,200,400,600,800},
xticklabels={
  \(\displaystyle {0}\),
  \(\displaystyle {200}\),
  \(\displaystyle {400}\),
  \(\displaystyle {600}\),
  \(\displaystyle {800}\)
},
y grid style={darkgray176},
ymin=0, ymax=3360,
ytick style={color=black},
ytick={0,500,1000,1500,2000,2500,3000,3500},
yticklabels={
  \(\displaystyle {0}\),
  \(\displaystyle {500}\),
  \(\displaystyle {1000}\),
  \(\displaystyle {1500}\),
  \(\displaystyle {2000}\),
  \(\displaystyle {2500}\),
  \(\displaystyle {3000}\),
  \(\displaystyle {3500}\)
}
]
\addplot graphics [includegraphics cmd=\pgfimage,xmin=0, xmax=800, ymin=0, ymax=3360] {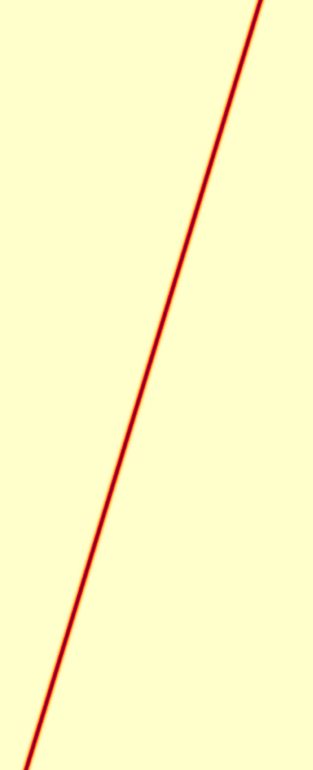};

\nextgroupplot[
colorbar,
colorbar style={ytick={-0.2,0,0.2,0.4,0.6,0.8,1,1.2},yticklabels={
  \(\displaystyle {\ensuremath{-}0.2}\),
  \(\displaystyle {0.0}\),
  \(\displaystyle {0.2}\),
  \(\displaystyle {0.4}\),
  \(\displaystyle {0.6}\),
  \(\displaystyle {0.8}\),
  \(\displaystyle {1.0}\),
  \(\displaystyle {1.2}\)
},ylabel={}, width=0.1*\pgfkeysvalueof{/pgfplots/parent axis width}, ticklabel style={font=\tiny}},
colormap={mymap}{[1pt]
  rgb(0pt)=(1,1,0.8);
  rgb(1pt)=(1,0.929411764705882,0.627450980392157);
  rgb(2pt)=(0.996078431372549,0.850980392156863,0.462745098039216);
  rgb(3pt)=(0.996078431372549,0.698039215686274,0.298039215686275);
  rgb(4pt)=(0.992156862745098,0.552941176470588,0.235294117647059);
  rgb(5pt)=(0.988235294117647,0.305882352941176,0.164705882352941);
  rgb(6pt)=(0.890196078431372,0.101960784313725,0.109803921568627);
  rgb(7pt)=(0.741176470588235,0,0.149019607843137);
  rgb(8pt)=(0.501960784313725,0,0.149019607843137)
},
height=1.2 * \figureheight,
hide x axis,
hide y axis,
point meta max=1.0000020161487,
point meta min=-1.56173253764891e-05,
tick align=outside,
tick pos=left,
title={\(\displaystyle Q_\mathrm{d}\)},
width=0.7 * \figurewidth,
x grid style={darkgray176},
xmin=0, xmax=800,
xtick style={color=black},
xtick={0,200,400,600,800},
xticklabels={
  \(\displaystyle {0}\),
  \(\displaystyle {200}\),
  \(\displaystyle {400}\),
  \(\displaystyle {600}\),
  \(\displaystyle {800}\)
},
y grid style={darkgray176},
ymin=0, ymax=3360,
ytick style={color=black},
ytick={0,500,1000,1500,2000,2500,3000,3500},
yticklabels={
  \(\displaystyle {0}\),
  \(\displaystyle {500}\),
  \(\displaystyle {1000}\),
  \(\displaystyle {1500}\),
  \(\displaystyle {2000}\),
  \(\displaystyle {2500}\),
  \(\displaystyle {3000}\),
  \(\displaystyle {3500}\)
}
]
\addplot graphics [includegraphics cmd=\pgfimage,xmin=0, xmax=800, ymin=0, ymax=3360] {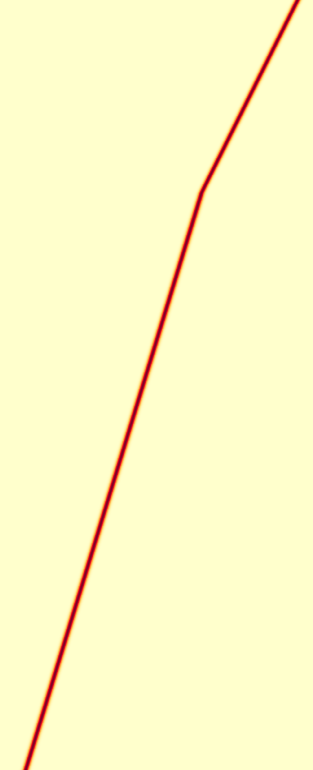};

\end{groupplot}

\draw ({$(current bounding box.south west)!-0.05!(current bounding box.south east)$}|-{$(current bounding box.south west)!0.3!(current bounding box.north west)$}) node[
  scale=0.96,
  anchor=west,
  text=black,
  rotate=90.0
]{time $t$};
\draw ({$(current bounding box.south west)!0.5!(current bounding box.south east)$}|-{$(current bounding box.south west)!-0.2!(current bounding box.north west)$}) node[
  scale=0.96,
  anchor=south,
  text=black,
  rotate=0.0
]{space $x$};
\end{tikzpicture} 
        \caption{Single tilt: Traveling wave with kink at $t=\frac{3}{4}T$.}
        \label{fig:Example_1}
    \end{subfigure}
    \hspace{0.05\textwidth}
    \begin{subfigure}[t]{0.45\textwidth}
        \centering
        \setlength\figureheight{0.8\linewidth}
        \setlength\figurewidth{0.8\linewidth}
\begin{tikzpicture}

\definecolor{darkgray176}{RGB}{176,176,176}

\begin{groupplot}[group style={group size=2 by 1, horizontal sep=1.2cm}]
\nextgroupplot[
colorbar,
colorbar style={ytick={-0.2,0,0.2,0.4,0.6,0.8,1,1.2},yticklabels={
  \(\displaystyle {\ensuremath{-}0.2}\),
  \(\displaystyle {0.0}\),
  \(\displaystyle {0.2}\),
  \(\displaystyle {0.4}\),
  \(\displaystyle {0.6}\),
  \(\displaystyle {0.8}\),
  \(\displaystyle {1.0}\),
  \(\displaystyle {1.2}\)
},ylabel={}, width=0.1*\pgfkeysvalueof{/pgfplots/parent axis width}, ticklabel style={font=\tiny}},
colormap={mymap}{[1pt]
  rgb(0pt)=(1,1,0.8);
  rgb(1pt)=(1,0.929411764705882,0.627450980392157);
  rgb(2pt)=(0.996078431372549,0.850980392156863,0.462745098039216);
  rgb(3pt)=(0.996078431372549,0.698039215686274,0.298039215686275);
  rgb(4pt)=(0.992156862745098,0.552941176470588,0.235294117647059);
  rgb(5pt)=(0.988235294117647,0.305882352941176,0.164705882352941);
  rgb(6pt)=(0.890196078431372,0.101960784313725,0.109803921568627);
  rgb(7pt)=(0.741176470588235,0,0.149019607843137);
  rgb(8pt)=(0.501960784313725,0,0.149019607843137)
},
height=1.2 * \figureheight,
hide x axis,
hide y axis,
point meta max=1.00000201032067,
point meta min=-1.56173253764891e-05,
tick align=outside,
tick pos=left,
title={\(\displaystyle Q\)},
width=0.7 * \figurewidth,
x grid style={darkgray176},
xmin=0, xmax=800,
xtick style={color=black},
xtick={0,200,400,600,800},
xticklabels={
  \(\displaystyle {0}\),
  \(\displaystyle {200}\),
  \(\displaystyle {400}\),
  \(\displaystyle {600}\),
  \(\displaystyle {800}\)
},
y grid style={darkgray176},
ymin=0, ymax=3360,
ytick style={color=black},
ytick={0,500,1000,1500,2000,2500,3000,3500},
yticklabels={
  \(\displaystyle {0}\),
  \(\displaystyle {500}\),
  \(\displaystyle {1000}\),
  \(\displaystyle {1500}\),
  \(\displaystyle {2000}\),
  \(\displaystyle {2500}\),
  \(\displaystyle {3000}\),
  \(\displaystyle {3500}\)
}
]
\addplot graphics [includegraphics cmd=\pgfimage,xmin=0, xmax=800, ymin=0, ymax=3360] {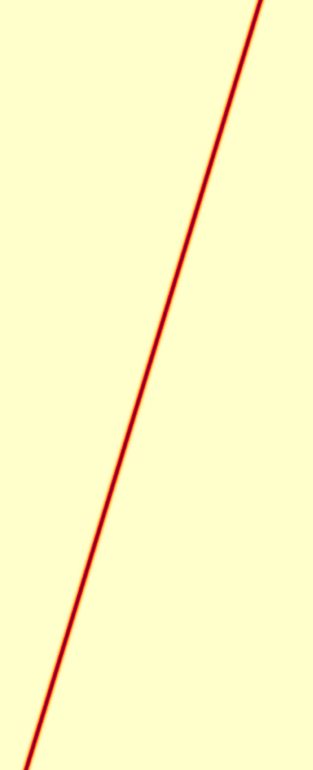};

\nextgroupplot[
colorbar,
colorbar style={ytick={-0.2,0,0.2,0.4,0.6,0.8,1,1.2},yticklabels={
  \(\displaystyle {\ensuremath{-}0.2}\),
  \(\displaystyle {0.0}\),
  \(\displaystyle {0.2}\),
  \(\displaystyle {0.4}\),
  \(\displaystyle {0.6}\),
  \(\displaystyle {0.8}\),
  \(\displaystyle {1.0}\),
  \(\displaystyle {1.2}\)
},ylabel={}, width=0.1*\pgfkeysvalueof{/pgfplots/parent axis width}, ticklabel style={font=\tiny}},
colormap={mymap}{[1pt]
  rgb(0pt)=(1,1,0.8);
  rgb(1pt)=(1,0.929411764705882,0.627450980392157);
  rgb(2pt)=(0.996078431372549,0.850980392156863,0.462745098039216);
  rgb(3pt)=(0.996078431372549,0.698039215686274,0.298039215686275);
  rgb(4pt)=(0.992156862745098,0.552941176470588,0.235294117647059);
  rgb(5pt)=(0.988235294117647,0.305882352941176,0.164705882352941);
  rgb(6pt)=(0.890196078431372,0.101960784313725,0.109803921568627);
  rgb(7pt)=(0.741176470588235,0,0.149019607843137);
  rgb(8pt)=(0.501960784313725,0,0.149019607843137)
},
height=1.2 * \figureheight,
hide x axis,
hide y axis,
point meta max=1.0000020161487,
point meta min=-1.56173253764891e-05,
tick align=outside,
tick pos=left,
title={\(\displaystyle Q_\mathrm{d}\)},
width=0.7 * \figurewidth,
x grid style={darkgray176},
xmin=0, xmax=800,
xtick style={color=black},
xtick={0,200,400,600,800},
xticklabels={
  \(\displaystyle {0}\),
  \(\displaystyle {200}\),
  \(\displaystyle {400}\),
  \(\displaystyle {600}\),
  \(\displaystyle {800}\)
},
y grid style={darkgray176},
ymin=0, ymax=3360,
ytick style={color=black},
ytick={0,500,1000,1500,2000,2500,3000,3500},
yticklabels={
  \(\displaystyle {0}\),
  \(\displaystyle {500}\),
  \(\displaystyle {1000}\),
  \(\displaystyle {1500}\),
  \(\displaystyle {2000}\),
  \(\displaystyle {2500}\),
  \(\displaystyle {3000}\),
  \(\displaystyle {3500}\)
}
]
\addplot graphics [includegraphics cmd=\pgfimage,xmin=0, xmax=800, ymin=0, ymax=3360] {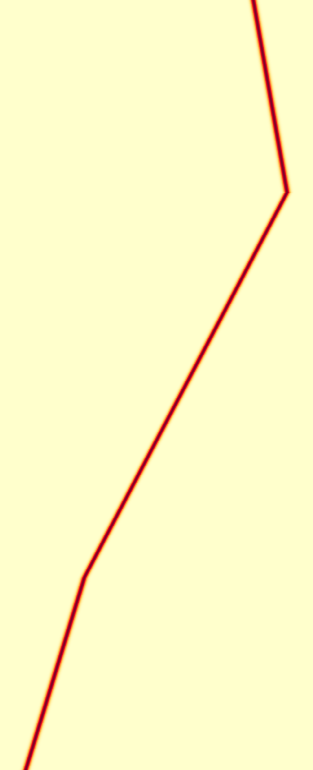};

\end{groupplot}

\draw ({$(current bounding box.south west)!-0.05!(current bounding box.south east)$}|-{$(current bounding box.south west)!0.3!(current bounding box.north west)$}) node[
  scale=0.96,
  anchor=west,
  text=black,
  rotate=90.0
]{time $t$};
\draw ({$(current bounding box.south west)!0.5!(current bounding box.south east)$}|-{$(current bounding box.south west)!-0.2!(current bounding box.north west)$}) node[
  scale=0.96,
  anchor=south,
  text=black,
  rotate=0.0
]{space $x$};
\end{tikzpicture} 
        \caption{Double tilt: Traveling wave with kinks at $t=\frac{1}{4}T$ and $t=\frac{3}{4}T$.}
        \label{fig:Example_4}
    \end{subfigure}
    \caption{Plots for the state and the target for the two example problems}
    \label{fig:advection_and_target}
\end{figure}
\begin{figure}[htp!]
    \centering
    \includegraphics[scale=0.42]{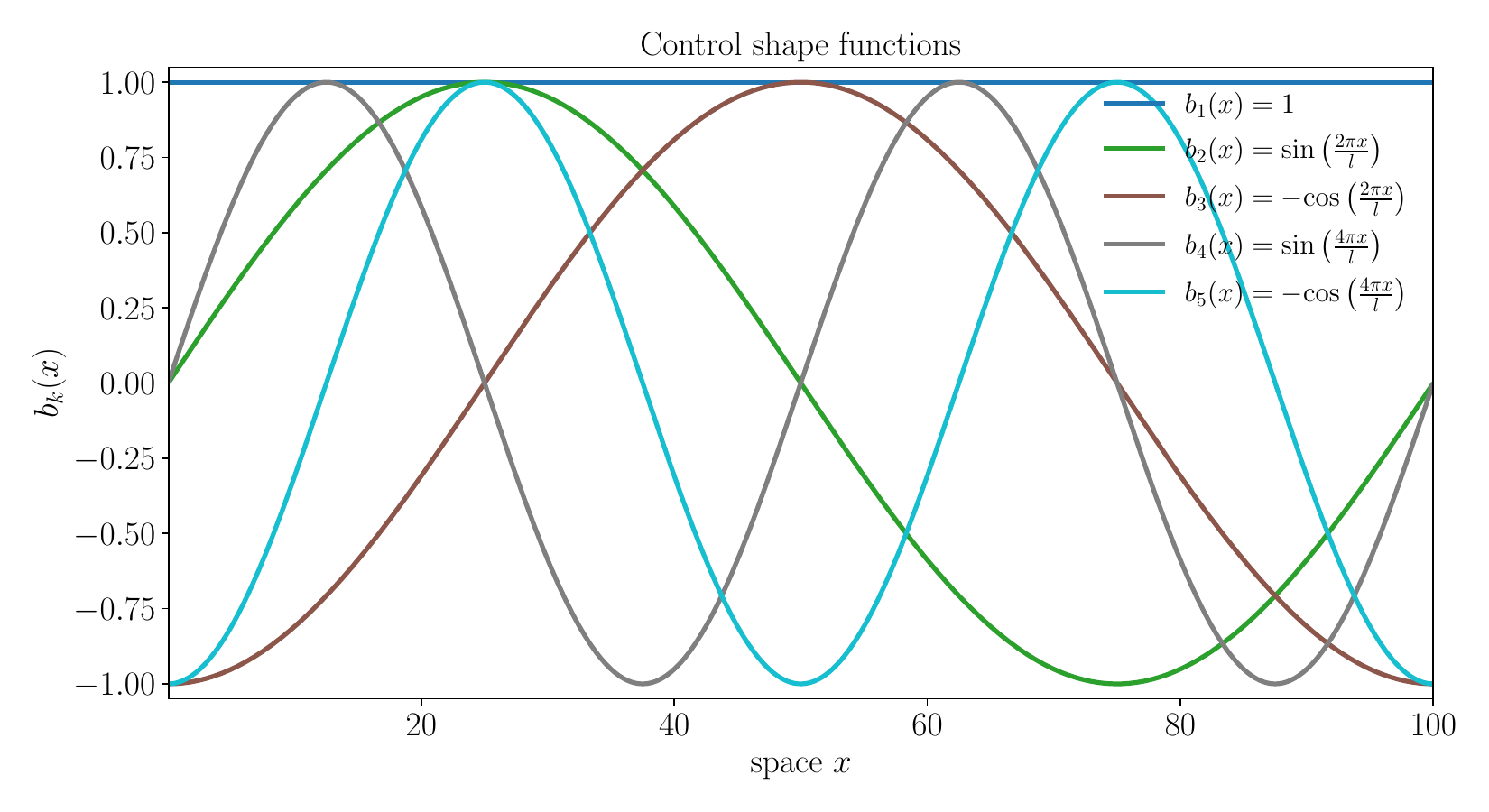}
    \caption{Control shape functions}
    \label{fig:Controls}
\end{figure}

The optimization parameters used for the example problems are given in \Cref{tab:params_1}, cf.~the upcoming \Cref{alg:FRTO_sPODG}.
\begin{table}[h!]
\begin{center}
\begin{minipage}{\textwidth}
\centering
\begin{tabular}{|l||c|c|c|c|c|c|} 
 \hline
 & $\mu$ & $\beta$ & $\omega^0$ & $n_{\mathrm{iter}}$ & $n_\mathrm{s}$ \\
\hline 
Single tilt problem & $10^{-3}$ & $1\times 10^{-5}$ & $1$ & $20000$ & $800$ \\
\hline   
Double tilt problem & $10^{-3}$ & $1\times 10^{-5}$ & $1$ & $20000$ & $800$ \\
\hline  
\end{tabular}
\caption{Constant parameters}
\label{tab:params_1}
\end{minipage}
\end{center}
\end{table}

\subsection{Algorithmic details}\label{ssec:Algorithmic_details}
The algorithmic details for both the sPOD-G method and the POD-G method are shown in \Cref{alg:FRTO_sPODG} and \Cref{alg:FRTO_PODG}, respectively.
\begin{algorithm}  
  \caption{Optimal control with sPOD-G}\label{alg:FRTO_sPODG}
  \begin{algorithmic}[1]
  \Require{$y_0$, $y_\mathrm{d}$, $\tilde{\ell}$, $\mu$, $\omega^0$, $n_\mathrm{iter}$, $\beta$, $n_\mathrm{s}$} 
    \State \textbf{Initialize:} $u^1=u_0=0$
    \State $y= \textsc{State}(u^1,y_0)$ \Comment{Solve uncontrolled \eqref{eq_def:OC_FOM_1} and \eqref{eq_def:OC_FOM_2}}
    \State $z, \mathcal{T}(z) = \textsc{Shifts}([y(t_1) \: \ldots \: y(t_{n_t})])$  \Comment{Compute shifts}
    \State $\{\phi_k\}^{\tilde{\ell}}_{k=1} = \textsc{Basis} ([\mathcal{T}(-z(t_1))y(t_1) \: \ldots \: \mathcal{T}(-z(t_{n_t}))y(t_{n_t})], \tilde{\ell})$  \Comment{Stationary basis}
      \For{$i = 1, \ldots, n_{\mathrm{iter}}$}
        \If{$\mathrm{refine}$} 
            \State $y= \textsc{State}(u^i,y_0)$ \Comment{Solve \eqref{eq_def:OC_FOM_1} and \eqref{eq_def:OC_FOM_2}}
            \State $\{\phi_k\}^{\tilde{\ell}}_{k=1} = \textsc{Basis} ([\mathcal{T}(-z(t_1))y(t_1) \: \ldots \: \mathcal{T}(-z(t_{n_t}))y(t_{n_t})], \tilde{\ell})$
        \EndIf
        \State $\alpha^i, z^i = \textsc{ReducedState}(\{\mathcal{T}(z)\phi_k\}^{\tilde{\ell}}_{k=1}, u^i,y_0)$ \Comment{Solve \eqref{eq_def:OC_FRTO_sPODG_1} and \eqref{eq_def:OC_FRTO_sPODG_2}}
        \State $\lambda^{\ell_a, i}, z^{\ell_a, i} = \textsc{ReducedAdjoint}(\{\mathcal{T}(z)\phi_k\}^{\tilde{\ell}}_{k=1}, u^i,y_\mathrm{d})$ \Comment{Solve \eqref{eq_def:OC_FRTO_sPODG_3} and \eqref{eq_def:OC_FRTO_sPODG_4}}
        \State $\tfrac{\mathrm{d}\mathcal{L}}{\mathrm{d}u^i}=\textsc{Gradient}(\lambda^{\ell_a, i}, z^{\ell_a, i}, \{\mathcal{T}(z)\phi_k\}^{\tilde{\ell}}_{k=1}, \alpha^i, u^i)$ \Comment{Solve \eqref{eq_def:OC_FRTO_sPODG_5}}
        \State $\omega^i = \textsc{StepSize}(\omega^{i- 1}, \tfrac{\mathrm{d}\mathcal{L}}{\mathrm{d} u^i}, \{\mathcal{T}(z)\phi_k\}^{\tilde{\ell}}_{k=1}, \alpha^i, u^{i})$
        \State $u^{i+1} = u^{i} - \omega^i\left(\tfrac{\mathrm{d}\mathcal{L}}{\mathrm{d} u^i}\right)$ 
        \If{$i=n_{\mathrm{iter}}$} 
            \State break
        \ElsIf{$\norm{\tfrac{\mathrm{d}\mathcal{L}}{\mathrm{d} u^i}} / \norm{\tfrac{\mathrm{d}\mathcal{L}}{\mathrm{d} u^1}} < \beta$}
            \State set $u=u^{i+1}$ and return
        \EndIf
      \EndFor
  \Ensure{$u$}
  \end{algorithmic}
\end{algorithm}
To solve both the state and adjoint FOM equations, we employ a first-order upwind scheme.
For time integration in the reduced-order setting, we use the first-order explicit Euler scheme for both the POD-G and sPOD-G reduced-state equations as well as for the corresponding linear reduced-adjoint equations. 

The computational steps of the optimal control framework are summarized in \Cref{alg:FRTO_sPODG}. 
Before the optimization starts, the uncontrolled FOM state equation is solved for $n_t$ time steps with $0$ control.
Subsequently the snapshot matrix comprising of all the $n_t$ snapshots of the FOM trajectory is constructed, following which the shift is determined as $z(t) = vt$, and, since $\Delta t = \Delta x / v$, amounts to an exact periodic shift of the $k$-th snapshot by $k$ grid points.
\begin{rmrk}
    At this point, we note that there are also sophisticated methods of selecting optimal snapshots for assembling the snapshot matrix. 
    These include \cite{alla_posteriori_2018} where the authors propose an a posteriori error control concept to determine the time instances of the snapshots to be chosen, \cite{kunisch_optimal_2010} where optimal time snapshots are chosen by minimizing the error between the POD solution and the trajectory of the dynamical system, \cite{oxberry_limitedmemory_2017}, where snapshots are selected on the fly during a single FOM solve, with the spacing between them chosen by a low-cost error estimator borrowed from adaptive time stepping and the POD basis updated incrementally as each snapshot arrives.
    However, since the problem we deal with is 1D, with the size of the snapshot matrix being $3201 \times 2400$, we consider all the $n_t$ snapshots for basis generation.
\end{rmrk}

From the determined shifts $z(t)$, the transformation operators $\mathcal{T}(z(t))$ are constructed and the stationary basis $\{\phi_i\}^{\tilde{\ell}}_{i=1}$ is computed by solving the optimization problem \eqref{eq_def:sPOD_advection_optimization}.
In the discrete setting this amounts to performing an SVD on the shifted state snapshot matrix to get the stationary basis.
An important point to note here is that the shifts $z(t)$ are the offline shifts and should not be confused with the shifts obtained while solving the sPOD-G reduced state equation.
Once these steps are completed, the optimal control loop starts by solving the reduced state equation for the current control $u^i$ and the reduced adjoint equation, after which the gradient descent steps to update the control are performed as shown in \Cref{alg:FRTO_sPODG}.

Other than before the start of the optimal control loop, the FOM solve for the control iterate $u^i$ and the basis refinement are performed at every fifth optimization step, as well as whenever the step-size selection criterion fails.
Since recomputing the offline shifts $z$ and the associated transformation operators $\mathcal{T}(z)$, which although sparse are costly to assemble, at every refinement would be expensive, we compute them once from the uncontrolled state snapshots and keep them fixed throughout the optimization (see \Cref{alg:FRTO_sPODG}).

\paragraph{Efficient offline/online decomposition}
Within the optimal control loop, the coefficient matrices of the reduced state equation depend on $\alpha^i$ and $z^i$ and hence they need to be updated in every time step.
Since the assembly of the coefficient matrices scales with the FOM dimension, this can render the ROM evaluation very time-consuming.
To address this, we follow a similar idea as in \cite{black_efficient_2021} and employ the following hyperreduction approach to achieve an efficient offline/online decomposition for the shift-dependent coefficient matrices.
In the setting considered in this section, with 1D linear advection equation and periodic shift operators, the only matrices depending on the shift in the reduced state equation are $B_1$ and $B_2$.
In the offline phase, we sample the matrices at $n_\mathrm{s}$ shift samples from a sufficiently large sample space $S$ of the possible shift values, where $S = \left\{ z_k \in [0, l] \ \middle\vert\ k = 1, 2, \dots, n_\mathrm{s} \right\}$ with equidistant $z_k$.
Here, we only consider shift values in $[0,l]$ due to the periodic boundary conditions.
Once the samples of $B_1$ and $B_2$ are computed in the offline phase, we can replace their evaluation in the online phase by corresponding piecewise linear interpolants w.r.t.~the shift value.
This avoids that the online phase scales with the FOM dimension.
Additionally, while solving the reduced adjoint equation we need to construct $B'_2(z^i(t))$ for which an efficient offline/online decomposition can be obtained analogously as for $B_1$ and $B_2$. 

The sPOD-G adjoint equation also involves the terms $\langle \mathcal{T}(\bar{z})\phi_j,  y_\mathrm{d}\rangle_H$ and $\langle\sum^{\tilde{\ell}}_{j=1}\bar{\alpha}_j\mathcal{T}'(\bar{z})\phi_j , y_\mathrm{d}\rangle_H$ whose evaluation scales with the FOM dimension.
For our test case, we can exploit the structure of the target state to achieve an efficient offline/online decomposition also for these terms.
Considering the sPOD ansatz \eqref{eq_def:sPOD_advection} for the target state $y_\mathrm{d}$, we rewrite the above terms as $\langle \mathcal{T}(\bar{z})\phi_j,  \sum^{\tilde{\ell}_\mathrm{d}}_{i=1} b_i\mathcal{T}(z_{\mathrm{d}}) \psi_i\rangle_H$ and $\langle\sum^{\tilde{\ell}}_{j=1}\bar{\alpha}_j\mathcal{T}'(\bar{z})\phi_j , \sum^{\tilde{\ell}_\mathrm{d}}_{i=1} b_i\mathcal{T}(z_{\mathrm{d}}) \psi_i\rangle_H$. 
Here we exploit the property of $\mathcal{T}$ mentioned in \Cref{rem:diff_T_vs_phi}, to rewrite these terms as $\sum^{\tilde{\ell}_\mathrm{d}}_{i=1} b_i\langle \phi_j, \mathcal{T}(z_{\mathrm{d}} - \bar{z}) \psi_i\rangle_H$ and $ \sum_{j=1}^{\tilde\ell}\sum_{i=1}^{\tilde\ell_\mathrm{d}} -\bar\alpha_j b_i \big\langle \phi'_j,\ \mathcal{T}(z_\mathrm{d} - \bar{z})\psi_i \big\rangle_H$, where an efficient offline/online decomposition of the terms $\langle \phi_j, \mathcal{T}(z_{\mathrm{d}} - \bar{z}) \psi_i\rangle_H$ and $-\langle \phi'_j,\ \mathcal{T}(z_\mathrm{d} - \bar{z})\psi_i \big\rangle_H$ can be obtained similarly as for the other shift-dependent matrices, by sampling the relative shift $z_\mathrm{d} - \bar{z}$.
As a result, the evaluation of the ROM adjoint equation does no longer scale with the FOM dimension.
Using the same approach, the cost functional \eqref{eq_def:sPOD_costFunc_continuous} can also be approximated such that its evaluation does not scale with the FOM dimension.
For the numerical experiments, we consider $\tilde{\ell}_\mathrm{d}$ to be $1$, since in our special setting the target state can be expressed as one shifted mode, cf.~\Cref{fig:advection_and_target}.

\begin{rmrk}\label{rmrk:basis_refinement}
    There is an important distinction between the theoretically derived necessary optimality conditions from \Cref{thm:first_order_cond} with its actual numerical implementation shown in \Cref{alg:FRTO_sPODG}.
    The theoretical derivation takes into account a fixed selection of modes $\{\phi_k\}$. 
    The optimality conditions do not warrant basis refinements, but instead the basis is fixed, e.g., determined based on the FOM solution for the initial guess control.
    However, during the numerical simulation, the intermediate controls can steer the system towards an arbitrary dynamics that the fixed initial basis might not have seen.
    In such cases, with a fixed basis, the optimal control loop stagnates.
    To circumvent such scenarios, in our numerical tests, the basis is recomputed at intermediate steps from the full-order trajectory associated with the current control iterate.
    Nevertheless, in certain scenarios, basis refinement might not be needed, cf.~Proposition \ref{prop:max_sv_invariance} and the beginning of \Cref{sec:using_roms_optimal_control}.
\end{rmrk}

The implementation of the POD-G method follows in a similar fashion and is shown in \Cref{alg:FRTO_PODG}.

\begin{algorithm}  
  \caption{Optimal control with POD-G}\label{alg:FRTO_PODG}
  \begin{algorithmic}[1]
  \Require{$y_0$, $y_\mathrm{d}$, $\ell$, $\mu$, $\omega^0$, $n_\mathrm{iter}$, $\beta$} 
    \State \textbf{Initialize:} $u^1=u_0=0$
    \State $y= \textsc{State}(u^1,y_0)$ \Comment{Solve uncontrolled \eqref{eq_def:OC_FOM_1} and \eqref{eq_def:OC_FOM_2}}
    \State $\{\phi_k\}^{\ell}_{k=1} = \textsc{Basis} ([y(t_1) \: \ldots \: y(t_{n_t})], \ell)$   \Comment{Compute the basis}
      \For{$i = 1, \ldots, n_{\mathrm{iter}}$}
        \If{$\mathrm{refine}$}
            \State $y= \textsc{State}(u^i,y_0)$ \Comment{Solve \eqref{eq_def:OC_FOM_1} and \eqref{eq_def:OC_FOM_2}}
            \State $\{\phi_k\}^{\ell}_{k=1} = \textsc{Basis} ([y(t_1) \: \ldots \: y(t_{n_t})], \ell)$  \Comment{Compute the basis}
        \EndIf
        \State $\alpha^i = \textsc{ReducedState}(\{\phi_k\}^{\ell}_{k=1}, u^i,y_0)$ \Comment{Solve \eqref{eq_def:OC_FRTO_PODG_1} and \eqref{eq_def:OC_FRTO_PODG_2}}
        \State $\lambda^{\ell, i} = \textsc{ReducedAdjoint}(\{\phi_k\}^{\ell}_{k=1}, u^i, \alpha^i,y_\mathrm{d})$ \Comment{Solve \eqref{eq_def:OC_FRTO_PODG_3} and \eqref{eq_def:OC_FRTO_PODG_4}}
        \State $\tfrac{\mathrm{d}\mathcal{L}}{\mathrm{d}u^i}= \textsc{Gradient}(\lambda^{\ell, i}, u^i)$ \Comment{Solve \eqref{eq_def:OC_FRTO_PODG_5}}
        \State $\omega^i = \textsc{StepSize}(\omega^{i- 1}, \tfrac{\mathrm{d}\mathcal{L}}{\mathrm{d}u^i}, \{\phi_k\}^{\ell}_{k=1}, \alpha^i, u^{i})$
        \State $u^{i+1} = u^{i} - \omega^i\left(\tfrac{\mathrm{d}\mathcal{L}}{\mathrm{d}u^i}\right)$ 
        \If{$i=n_{\mathrm{iter}}$} 
            \State break
        \ElsIf{$\norm{\tfrac{\mathrm{d}\mathcal{L}}{\mathrm{d} u^i}} / \norm{\tfrac{\mathrm{d}\mathcal{L}}{\mathrm{d} u^1}} < \beta$}
            \State set $ u= u^{i+1}$ and return
        \EndIf
      \EndFor
  \Ensure{$u$}
  \end{algorithmic}
\end{algorithm}

The selection of the specific value for $n_\mathrm{iter}$ depends on the convergence criteria met for the FOM with the given $\beta$. 
Consequently, we terminate the optimal control loop for the techniques discussed either after reaching the prescribed $n_\mathrm{iter}$, or when the relative norm of the gradient $\norm{\tfrac{\mathrm{d}\mathcal{L}}{\mathrm{d} u^i}}$ falls below the specified $\beta$, or if the optimal control loop stagnates and the cost functional does not decrease for the next $100$ iterations.
We use the nonmonotone forward-backward splitting method of \cite{azmi_forwardbackward_2025} with alternating Barzilai-Borwein (BB) initial step sizes, albeit restarted at every basis update, i.e., the nonmonotone reference window and the BB memory are reset whenever the basis changes. 

As described in the algorithms, when using reduced-order models, one needs to input the number of truncated modes.
For this, one could specify a fixed number of modes in advance.
However, it is not necessary to maintain the same number of modes at each step to accurately represent the dynamics. 
To address this, we also implemented a tolerance-based strategy, dynamically selecting the number of truncated modes at each optimization step based on the criterion
\begin{equation}\label{eq_def:epsilon_criteria}
    d = \sum^{\min(n, n_t)}_{i=1} \mathds{1}\left(\frac{\sigma_i}{\sigma_1} > \mathrm{tol}\right)\,.
\end{equation}
Here, $\{\sigma_i\}$ are the singular values for the POD-G method and the sPOD-G method, $\mathrm{tol}$ is a tolerance selected in advance, $n$ is the dimension of the spatially semi-discretized FOM, and $\mathds{1}$ is a function mapping true statements to $1$ and wrong statements to $0$.

\subsection{Using ROMs in optimal control}\label{sec:using_roms_optimal_control}
We begin by numerically verifying Proposition \ref{prop:max_sv_invariance} for the sPOD-G method for which we  refer to \Cref{fig:Theoretical_test_J}. 
\begin{figure}[htp!]
    \centering
    \begin{subfigure}[t]{0.44\textwidth}
        \centering
        \setlength\figureheight{0.85\linewidth}
        \setlength\figurewidth{0.95\linewidth}
        \begin{tikzpicture}

  \definecolor{brown}{RGB}{165,42,42}
  \definecolor{darkgray176}{RGB}{176,176,176}
  \definecolor{green}{RGB}{0,128,0}
  \definecolor{lightgray204}{RGB}{204,204,204}
  \definecolor{darkorange25512714}{RGB}{255,127,14}
  \definecolor{forestgreen4416044}{RGB}{44,160,44}
  \definecolor{sienna}{RGB}{160,82,45}
  \definecolor{steelblue31119180}{RGB}{31,119,180}
    \definecolor{crimson2143940}{RGB}{214,39,40}

  \begin{axis}[
    height=0.9*\figureheight,
    width=1.2*\figurewidth,
    log basis x={10},
    log basis y={10},
    xmin=5.01187233627271e-08,   xmax=66,
    ymin=0.91019163314602,      ymax=40,
    xlabel={$m + 1$},
    ylabel={$\scriptstyle{\mathcal{J}}$},
    xmajorgrids,
    ymajorgrids,
    x grid style={darkgray176},
    y grid style={darkgray176},
    xtick={5,10,15,20,25,30,35,40,45,50,55,60,65},
    xticklabels={
      \(\scriptstyle{5}\),
      \(\scriptstyle{10}\),
      \(\scriptstyle{15}\),
      \(\scriptstyle{20}\),
      \(\scriptstyle{25}\),
      \(\scriptstyle{30}\),
      \(\scriptstyle{35}\),
      \(\scriptstyle{40}\),
      \(\scriptstyle{45}\),
      \(\scriptstyle{50}\),
      \(\scriptstyle{55}\),
      \(\scriptstyle{60}\),
      \(\scriptstyle{65}\),
    },
    ytick={5,10,15,20,25,30,35,40},
    yticklabels={
        \(\scriptstyle{5}\),
      \(\scriptstyle{10}\),
      \(\scriptstyle{15}\),
      \(\scriptstyle{20}\),
      \(\scriptstyle{25}\),
      \(\scriptstyle{30}\),
      \(\scriptstyle{35}\),
      \(\scriptstyle{40}\),
   },
    legend style={
    font=\footnotesize,
      fill opacity=0.5,
      draw opacity=0.5,
      text opacity=1,
      nodes={scale=0.6},
      draw=lightgray204,
      at={(0.95,0.95)}
    }
  ]

    \addplot+[
      semithick,
      brown,
      mark=o,
      mark size=3,
      mark options={solid},
      nodes near coords,
      every node near coord/.append style={
        font=\tiny,
        inner sep=0pt,
        outer sep=0pt,
        yshift=4pt
      },
      point meta=explicit symbolic
    ]
    table[
      row sep=\\,
      x=x,
      y=y,
      meta=label
    ] {
      x           y                    label \\
    4 37.9605 {} \\
    10 33.9152 {} \\
    16 26.8886 {} \\
    22 21.6720 {} \\
    28 17.0689 {} \\
    32 14.0067 {} \\
    36 11.4627 {} \\
    42 8.4992 {} \\
    52 4.6084 {} \\
    62 2.3735 {} \\
    };
    \addlegendentry{\small FOM}

    \addplot+[
      semithick,
      dashed,
      darkorange25512714,
      mark=asterisk,
      mark size=3,
      mark options={solid},
      nodes near coords,
      every node near coord/.append style={
        font=\tiny,
        inner sep=0pt,
        outer sep=0pt,
        yshift=2pt
      },
      point meta=explicit symbolic
    ]
    table[
      row sep=\\,
      x=x,
      y=y,
      meta=label
    ] {
      x           y                    label \\
4 37.960486895360646 {} \\
10 33.91539988127204 {} \\
16 26.889754278733967 {} \\
22 21.67474451229683 {} \\
28 17.074179988618567 {} \\
32 14.014300838205562 {} \\
36 11.473066892485143 {} \\
42 8.51325484385235 {} \\
52 4.629344926740039 {} \\
62 2.399097207599732 {} \\
    };
    \addlegendentry{\small sPOD-G ($\tilde{\ell}=m+1$)}

  \end{axis}
\end{tikzpicture}   
        \caption{Single tilt problem}
        \label{fig:Shifting}
    \end{subfigure}
    \hspace{0.01\textwidth}
    \begin{subfigure}[t]{0.44\textwidth}
        \centering
        \setlength\figureheight{0.85\linewidth}
        \setlength\figurewidth{0.95\linewidth}
        \begin{tikzpicture}

  \definecolor{brown}{RGB}{165,42,42}
  \definecolor{darkgray176}{RGB}{176,176,176}
  \definecolor{green}{RGB}{0,128,0}
  \definecolor{lightgray204}{RGB}{204,204,204}
  \definecolor{darkorange25512714}{RGB}{255,127,14}
  \definecolor{forestgreen4416044}{RGB}{44,160,44}
  \definecolor{sienna}{RGB}{160,82,45}
  \definecolor{steelblue31119180}{RGB}{31,119,180}
    \definecolor{crimson2143940}{RGB}{214,39,40}

  \begin{axis}[
    height=0.9*\figureheight,
    width=1.2*\figurewidth,
    log basis x={10},
    xmin=5.01187233627271e-08,   xmax=66,
    ymin=5,      ymax=125,
    xlabel={$m + 1$},
    ylabel={$\scriptstyle{\mathcal{J}}$},
    xmajorgrids,
    ymajorgrids,
    x grid style={darkgray176},
    y grid style={darkgray176},
    xtick={5,10,15,20,25,30,35,40,45,50,55,60,65},
    xticklabels={
      \(\scriptstyle{5}\),
      \(\scriptstyle{10}\),
      \(\scriptstyle{15}\),
      \(\scriptstyle{20}\),
      \(\scriptstyle{25}\),
      \(\scriptstyle{30}\),
      \(\scriptstyle{35}\),
      \(\scriptstyle{40}\),
      \(\scriptstyle{45}\),
      \(\scriptstyle{50}\),
      \(\scriptstyle{55}\),
      \(\scriptstyle{60}\),
      \(\scriptstyle{65}\),
    },
    ytick={10,30,50,70,90,110},
    yticklabels={
     \(\scriptstyle{10}\),
      \(\scriptstyle{30}\),
      \(\scriptstyle{50}\),
      \(\scriptstyle{70}\),
      \(\scriptstyle{90}\),
      \(\scriptstyle{110}\),
    },
    legend style={
    font=\footnotesize,
      fill opacity=0.5,
      draw opacity=0.5,
      text opacity=1,
      nodes={scale=0.6},
      draw=lightgray204,
      at={(0.95,0.95)}
    }
  ]

    \addplot+[
      semithick,
      brown,
      mark=o,
      mark size=3,
      mark options={solid},
      nodes near coords,
      every node near coord/.append style={
        font=\tiny,
        inner sep=0pt,
        outer sep=0pt,
        yshift=4pt
      },
      point meta=explicit symbolic
    ]
    table[
      row sep=\\,
      x=x,
      y=y,
      meta=label
    ] {
      x           y                    label \\
    4 117.9264 {} \\
    10 98.3670 {} \\
    16 80.6875 {} \\
    22 64.7664 {} \\
    28 50.6654 {} \\
    32 42.2032 {} \\
    36 34.7754 {} \\
    42 25.4186 {} \\
    52 13.9992 {} \\
    62 7.1123 {} \\
    };
    \addlegendentry{\small FOM}

    \addplot+[
      semithick,
      dashed,
      darkorange25512714,
      mark=asterisk,
      mark size=3,
      mark options={solid},
      nodes near coords,
      every node near coord/.append style={
        font=\tiny,
        inner sep=0pt,
        outer sep=0pt,
        yshift=2pt
      },
      point meta=explicit symbolic
    ]
    table[
      row sep=\\,
      x=x,
      y=y,
      meta=label
    ] {
      x           y                    label \\
4 117.92641631793231 {} \\
10 98.36771086897349 {} \\
16 80.69087301200628 {} \\
22 64.77561382703757 {} \\
28 50.685482068323836 {} \\
32 42.23299605005346 {} \\
36 34.81730918831691 {} \\
42 25.482475704328056 {} \\
52 14.098645158293944 {} \\
62 7.239504151252186 {} \\
    };
    \addlegendentry{\small sPOD-G ($\tilde{\ell}=m+1$)}

  \end{axis}
\end{tikzpicture}   
        \caption{Double tilt problem}
        \label{fig:Shifting_3}
    \end{subfigure}
    \caption{Number of controls vs. $\mathcal{J}$}
    \label{fig:Theoretical_test_J}
\end{figure}
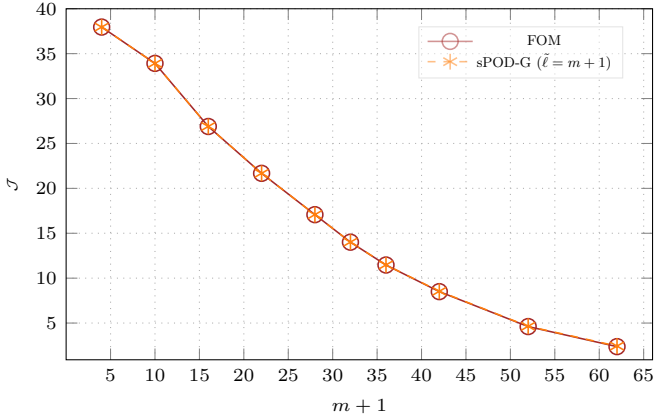
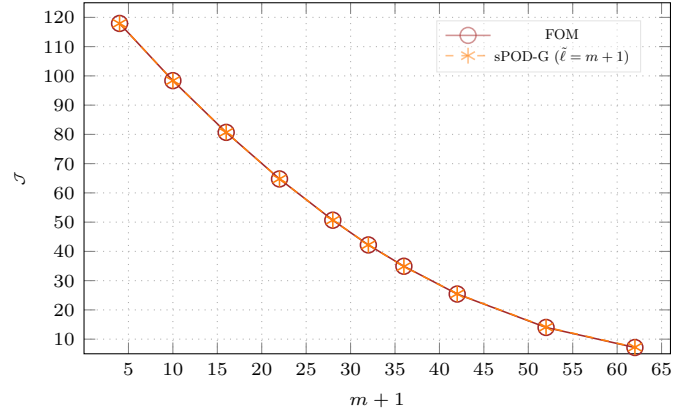
The FOM optimal control problem is solved for different numbers of control basis functions $b_k$ and consequently, for varying numbers of controls. 
The converged cost functional values are reported in the plots. 
For the sPOD-G method we use the same number of controls as for the FOM.
We follow the optimal control framework as shown in \Cref{alg:FRTO_sPODG}, albeit with minor modifications. 
The stationary basis functions are not computed by applying an SVD on the shifted state snapshots, rather 
$\{\phi_i\}^{\tilde{\ell}}_{i=1}$ is built from pairs of complex conjugate eigenfunctions of the operator $\mathcal{A}=-v\tfrac{\mathrm{d}}{\mathrm{d}x}$.
We collect $m$ eigenfunctions along with the initial condition $y_0$ and extract an orthonormal basis of size $\tilde{\ell} = m + 1$ using an SVD.
This basis is then used as the stationary basis for the sPOD-G method.
Also, since $\mathrm{span}\{\phi_i\}^{\tilde{\ell}}_{i=1} = \mathrm{span}\{y_0, b_1, \ldots, b_m\}$, no basis updates and thereby no FOM state solve are required during the intermediate optimization steps. 
From \Cref{fig:Theoretical_test_J}, we see that the costs obtained by the sPOD-G method match the cost functional values of the FOM. 
Moreover, \Cref{fig:theoretical_svd} demonstrates that $m + 1$ stationary modes suffice to approximate the advection snapshots in the stationary frame where $\frac{\sigma_{m + 2}}{\sigma_1} \ll \frac{\sigma_{m + 1}}{\sigma_1}$ for all optimization steps, which numerically confirms Proposition \ref{prop:max_sv_invariance}.
\begin{figure}[htp!]
    \centering
        \begin{subfigure}[t]{0.44\textwidth}
        \centering
        \setlength\figureheight{0.9\linewidth}
        \setlength\figurewidth{1.0\linewidth}

\begin{tikzpicture}

  \definecolor{brown}{RGB}{165,42,42}
  \definecolor{darkgray176}{RGB}{176,176,176}
  \definecolor{green}{RGB}{0,128,0}
  \definecolor{lightgray204}{RGB}{204,204,204}

\definecolor{crimson2143940}{RGB}{214,39,40}
  \definecolor{darkorange25512714}{RGB}{255,127,14}
  \definecolor{forestgreen4416044}{RGB}{44,160,44}
  \definecolor{sienna}{RGB}{160,82,45}
  \definecolor{steelblue31119180}{RGB}{31,119,180}

\begin{axis}[
height=\figureheight,
legend cell align={left},
legend style={
  fill opacity=0.8,
  draw opacity=1,
  text opacity=1,
  at={(0.97,0.5)},
  anchor=east,
  draw=lightgray204
},
log basis y={10},
tick align=outside,
tick pos=left,
width=\figurewidth,
x grid style={darkgray176},
xlabel={\(\displaystyle n_{\mathrm{iter}}\)},
xmajorgrids,
xmin=-27.5, xmax=577.5,
xtick style={color=black},
xtick={0,100,200,300,400,500,600},
xticklabels={
  \(\scriptstyle {0}\),
  \(\scriptstyle {100}\),
  \(\scriptstyle {200}\),
  \(\scriptstyle {300}\),
  \(\scriptstyle {400}\),
  \(\scriptstyle {500}\),
  \(\scriptstyle {600}\),
},
y grid style={darkgray176},
ylabel={\(\displaystyle \frac{\sigma}{\sigma_1}\)},
ymajorgrids,
ymin=5.0118723362727144e-17, ymax=1.9952623149688786e-08,
ymode=log,
ytick style={color=black},
ytick={1e-19,1e-17,1e-15,1e-13,1e-11,1e-09,1e-07,1e-05,0.001},
yticklabels={
  \(\scriptstyle {10^{-19}}\),
  \(\scriptstyle {10^{-17}}\),
  \(\scriptstyle {10^{-15}}\),
  \(\scriptstyle {10^{-13}}\),
  \(\scriptstyle {10^{-11}}\),
  \(\scriptstyle {10^{-9}}\),
  \(\scriptstyle {10^{-7}}\),
  \(\scriptstyle {10^{-5}}\),
  \(\scriptstyle {10^{-3}}\)
}
]
\addplot [thick, steelblue31119180, mark=o, mark size=2, mark options={solid}]
table {%
1 2.816550247259015e-16
6 1.3949300926158632e-14
11 8.447992345240342e-11
16 1.2618730526199749e-10
21 1.2021496767797157e-10
26 1.3744064134136952e-10
31 1.4544932165391465e-10
36 2.1750997707850677e-10
41 2.2508216271526882e-10
46 2.2861807751013583e-10
51 3.627947981831608e-11
56 3.635655369271986e-10
61 4.135675535941924e-10
66 4.1990918205258605e-10
71 4.5950705632356225e-10
76 4.701399710253249e-10
81 4.799940427312254e-10
86 5.223599956050149e-10
91 5.319874697223988e-10
96 5.36776747407327e-10
101 7.610901652426583e-10
106 7.609441402858414e-10
111 7.682678279902364e-10
116 9.231593503396068e-10
121 9.309183999749671e-10
126 1.0127162289293633e-09
131 1.0151180221897427e-09
136 1.0179059253631503e-09
141 1.0384295889501665e-09
146 1.3036221359059119e-09
151 1.310695864597783e-09
156 1.3147178888091501e-09
161 1.3292010058807862e-09
166 1.4419987687527124e-09
171 1.4425335684090339e-09
176 1.430639770252027e-09
181 1.4703076074238433e-09
186 1.5016165991900671e-09
191 1.484971287602889e-09
196 1.621979511480811e-09
201 1.6245014687019926e-09
206 1.633074820789418e-09
211 1.6332426590264974e-09
216 1.6405397754194157e-09
221 1.6421984832034571e-09
226 1.7166850901707078e-09
231 1.7101408512847928e-09
236 1.7238836389030443e-09
241 1.7224633337183526e-09
246 1.7041890319234463e-09
251 1.7070153184777989e-09
256 1.710266956780818e-09
261 1.7203463910035936e-09
266 1.729497898883359e-09
271 1.733907056679592e-09
276 1.734898081063128e-09
281 1.7354126808636604e-09
286 1.7345771825001762e-09
291 1.750342781949891e-09
296 1.7628519988611369e-09
301 1.7814837032441681e-09
306 1.775767960811909e-09
311 1.8093528098977426e-09
316 1.9332612291397664e-09
321 1.8053872504628796e-09
326 1.811809932154769e-09
331 1.8731057104593546e-09
336 1.818673244601355e-09
341 6.188750978493291e-09
346 2.0508173454611257e-09
351 1.8331902013311838e-09
356 1.8588841292053355e-09
361 1.8276282925722042e-09
366 1.986720110151699e-09
371 1.840827864848385e-09
376 1.8497003827704392e-09
381 1.8415653301463944e-09
386 3.6190615753374055e-09
391 2.280128302243235e-09
396 1.8435389727943944e-09
401 1.959546401811344e-09
406 1.8563528182670757e-09
411 7.640354623408557e-10
416 1.8505676329461979e-09
421 1.8511030186359673e-09
426 1.8627055616557692e-09
431 4.2313811738532955e-09
436 9.48908511280032e-09
441 1.9652478447402436e-09
446 1.9567892259223626e-09
451 7.162244228809722e-09
456 1.9567527837083903e-09
461 1.9582764852446977e-09
466 1.957472846158587e-09
471 2.0047620260604388e-09
476 2.244515665567608e-09
481 1.9730432611993206e-09
486 2.0093455305364614e-09
491 1.965067606930643e-09
496 1.9540614021690605e-09
501 2.03106105133232e-09
506 1.9776620499654935e-09
511 1.962019151121049e-09
516 1.9580716202195245e-09
521 1.7413078164370647e-09
526 2.0259204133557203e-09
531 1.973914607633644e-09
536 2.0089328251462725e-09
541 1.980644683390233e-09
546 1.985996845617432e-09
550 2.0173559675448855e-09
};
\addlegendentry{$\scriptstyle{\sigma_{m + 1}}$}
\addplot [thick, darkorange25512714, mark=star, mark size=2, mark options={solid}]
table {%
1 1.5221912623114055e-16
6 1.5245586648608745e-16
11 1.8068266090347137e-16
16 1.3501695386442933e-16
21 1.9228947502792307e-16
26 1.984915146586772e-16
31 2.2770359710730226e-16
36 1.9441755597986418e-16
41 1.3269658780812065e-16
46 1.6824091199965928e-16
51 2.0090059043247962e-16
56 2.536596659177348e-16
61 1.9593253343211242e-16
66 2.267049828377366e-16
71 2.2527155488119787e-16
76 2.9057273374669774e-16
81 2.1783263939017551e-16
86 1.7926287047673138e-16
91 2.39066035406712e-16
96 2.3882023348239187e-16
101 1.5816330720936398e-16
106 1.4196354070346158e-16
111 2.0123164346549366e-16
116 2.686816893198347e-16
121 3.1742435811510156e-16
126 2.2238723226699796e-16
131 3.2585087204600976e-16
136 1.9100829897418442e-16
141 2.1484855086865078e-16
146 1.3754024916795694e-16
151 2.0193766086596233e-16
156 1.9324483242195038e-16
161 2.3543909459495316e-16
166 2.5325701575530842e-16
171 2.7311896028614123e-16
176 7.116979201767155e-16
181 2.326898962643398e-16
186 2.724867081925845e-16
191 2.373657458329583e-16
196 1.2069608853947699e-15
201 1.1049387700053001e-15
206 2.863163539996067e-16
211 2.2549248917056836e-16
216 1.8700107716569614e-16
221 2.626022463595601e-16
226 8.8602883402086085e-16
231 9.88038534779368e-16
236 4.625700712715409e-16
241 3.2797635547154355e-16
246 1.0958276134368567e-15
251 1.1294784988463724e-15
256 1.1703659233795e-15
261 2.984583704409076e-16
266 2.3093183265657065e-16
271 1.0334877667476184e-15
276 2.052076579652185e-16
281 1.509391063926496e-16
286 8.906356080804411e-16
291 9.770384372170428e-16
296 1.0589529508144664e-15
301 5.656116811136665e-16
306 1.9194716798403688e-16
311 1.1306721718803735e-15
316 9.631208851780932e-16
321 8.31244917197123e-16
326 1.0091471204318136e-15
331 1.1068543077713625e-15
336 2.34359665232484e-16
341 1.530686916360476e-16
346 7.366572573404228e-16
351 1.3244163699605128e-15
356 2.5798807406375397e-16
361 6.599753348021013e-16
366 7.934025358570729e-16
371 2.2162021020573756e-16
376 3.0004168674282294e-16
381 2.6601573568219953e-16
386 3.106718329860744e-16
391 9.347955174087311e-16
396 1.1126560909122263e-15
401 1.3768459842465383e-16
406 1.9921577313948342e-16
411 3.122286402624437e-16
416 1.1313676276909816e-15
421 9.568405021179995e-16
426 1.6137577652748736e-16
431 1.2185709590130477e-16
436 1.2965993288988882e-16
441 3.751131244023709e-16
446 1.7177803938249126e-16
451 1.537472936537793e-16
456 2.6274465120017013e-16
461 9.557502672114987e-16
466 3.770925624242585e-16
471 1.840653121558311e-16
476 9.113844556243832e-16
481 7.363511542248374e-16
486 9.156322024828153e-16
491 1.0186640263658682e-15
496 1.0740422237798375e-15
501 1.0678704299939185e-15
506 9.895304299292398e-16
511 2.810497561796687e-16
516 2.052599409880545e-16
521 6.477435437659597e-16
526 1.2098597004408535e-15
531 7.125911771528733e-16
536 1.0894935981898307e-15
541 1.9898251691921807e-16
546 1.832245008604258e-16
550 1.1789789599809161e-15
};
\addlegendentry{$\scriptstyle{\sigma_{m + 2}}$}
\end{axis}

\end{tikzpicture}   
        \caption{Single tilt problem}
        \label{fig:Shifting_svd}
    \end{subfigure}
    \hspace{0.02\textwidth}
    \begin{subfigure}[t]{0.44\textwidth}
        \centering
        \setlength\figureheight{0.9\linewidth}
        \setlength\figurewidth{1.0\linewidth}

\begin{tikzpicture}

  \definecolor{brown}{RGB}{165,42,42}
  \definecolor{darkgray176}{RGB}{176,176,176}
  \definecolor{green}{RGB}{0,128,0}
  \definecolor{lightgray204}{RGB}{204,204,204}

\definecolor{crimson2143940}{RGB}{214,39,40}
  \definecolor{darkorange25512714}{RGB}{255,127,14}
  \definecolor{forestgreen4416044}{RGB}{44,160,44}
  \definecolor{sienna}{RGB}{160,82,45}
  \definecolor{steelblue31119180}{RGB}{31,119,180}

\begin{axis}[
height=\figureheight,
legend cell align={left},
legend style={
  fill opacity=0.8,
  draw opacity=1,
  text opacity=1,
  at={(0.97,0.5)},
  anchor=east,
  draw=lightgray204
},
log basis y={10},
tick align=outside,
tick pos=left,
width=\figurewidth,
x grid style={darkgray176},
xlabel={\(\displaystyle n_{\mathrm{iter}}\)},
xmajorgrids,
xmin=-28.75, xmax=603.75,
xtick style={color=black},
xtick={0,100,200,300,400,500,600},
xticklabels={
  \(\scriptstyle {0}\),
  \(\scriptstyle {100}\),
  \(\scriptstyle {200}\),
  \(\scriptstyle {300}\),
  \(\scriptstyle {400}\),
  \(\scriptstyle {500}\),
  \(\scriptstyle {600}\),
},
y grid style={darkgray176},
ylabel={\(\displaystyle \frac{\sigma}{\sigma_1}\)},
ymajorgrids,
ymin=5.011872336272715e-18, ymax=1.9952623149688786e-05,
ymode=log,
ytick style={color=black},
ytick={1e-19,1e-17,1e-15,1e-13,1e-11,1e-09,1e-07,1e-05,0.001,0.1},
yticklabels={
  \(\scriptstyle {10^{-19}}\),
  \(\scriptstyle {10^{-17}}\),
  \(\scriptstyle {10^{-15}}\),
  \(\scriptstyle {10^{-13}}\),
  \(\scriptstyle {10^{-11}}\),
  \(\scriptstyle {10^{-9}}\),
  \(\scriptstyle {10^{-7}}\),
  \(\scriptstyle {10^{-5}}\),
  \(\scriptstyle {10^{-3}}\),
  \(\scriptstyle {10^{-1}}\)
}
]
\addplot [thick, steelblue31119180, mark=o, mark size=2, mark options={solid}]
table {%
1 8.760565572202517e-17
6 1.414015121452402e-12
11 4.7943146563248455e-09
16 7.655915428830504e-07
21 2.043377122850249e-07
26 1.597310351539669e-06
31 1.7953859729232747e-06
36 6.882412521622186e-07
41 1.871814673260632e-06
46 2.007740193699336e-06
51 6.941492896729807e-08
56 2.2995712985111427e-06
61 2.618978912633635e-06
66 2.711940035613392e-06
71 1.853004441741311e-07
76 1.259097328399507e-07
81 2.946564217366215e-06
86 2.9546069467196536e-06
91 3.021895912411921e-06
96 3.042802779303405e-06
101 2.662311870868979e-06
106 2.218245083909906e-07
111 3.161517933595683e-06
116 3.4042405074370385e-06
121 9.162537481934844e-08
126 1.3513189324570555e-07
131 3.840916108015701e-06
136 2.0479864616063943e-06
141 1.1283778184948268e-06
146 3.7022370794118577e-06
151 3.851691846565535e-06
156 3.8785435435858125e-06
161 4.1522968437712e-06
166 3.236321430824034e-06
171 3.95391016870816e-06
176 4.1600181805212635e-06
181 1.038953573523175e-07
186 1.3012905332224817e-07
191 9.06684120934889e-07
196 4.215507348327552e-06
201 4.266886513627747e-06
206 4.221208683353379e-06
211 2.1266760732995107e-06
216 4.25182406144579e-06
221 3.7525247837236637e-06
226 4.030392483969899e-06
231 4.247393840207087e-06
236 4.044228310953561e-06
241 3.914215730290835e-06
246 2.5414878040691987e-06
251 1.4059451486900467e-06
256 3.991105502391235e-06
261 4.148827194505153e-06
266 5.030049623500298e-07
271 1.6198737833438352e-07
276 1.8117476964714682e-06
281 4.152650037194772e-06
286 2.9658066366780986e-06
291 6.212008654716574e-07
296 4.223172659355028e-06
301 5.356689592711028e-07
306 4.5843693076232446e-06
311 4.556936302161312e-06
316 4.465891470855243e-06
321 1.4424494394779545e-06
326 4.593023677603701e-06
331 9.508562633871635e-07
336 4.476380505744672e-06
341 3.5945464919041564e-07
346 2.810026576569004e-06
351 4.4222958465504846e-06
356 4.411640418683393e-06
361 1.522062147454162e-06
366 4.616436395927228e-06
371 2.911818921990224e-06
376 4.437271353102917e-06
381 4.759648157579961e-06
386 4.517230207648289e-06
391 2.3968014751935263e-06
396 4.721367396566465e-06
401 3.7900421811030915e-07
406 4.264266304909754e-06
411 4.528492146179888e-06
416 1.6792421001193524e-06
421 3.136009050034478e-07
426 7.99364057855617e-08
431 2.0526180233181413e-08
436 2.187185394890167e-06
441 4.040120067871176e-06
446 4.688180807852972e-06
451 2.260740089688752e-06
456 1.1093583372960898e-06
461 2.4824489401026373e-06
466 4.6438488378568055e-06
471 2.1796707039057986e-07
476 4.525080001750574e-06
481 3.3362054634917475e-06
486 4.758402785853655e-06
491 2.058963874582866e-06
496 4.672595853811947e-06
501 1.5266722998022371e-06
506 2.5329264942146306e-06
511 1.6366373707948465e-07
516 4.781896872240444e-06
521 4.775150676714711e-06
526 3.838568401818521e-06
531 1.1076542235901602e-07
536 8.421798243201018e-07
541 5.88159452056856e-07
546 4.661054507869406e-06
551 4.7380254885865495e-06
556 4.753027075448025e-06
561 3.343547692745029e-07
566 3.7472739217106656e-06
571 4.461983487837666e-07
575 1.7035934150211413e-07
};
\addlegendentry{$\scriptstyle{\sigma_{m + 1}}$}
\addplot [thick, darkorange25512714, mark=star, mark size=2, mark options={solid}]
table {%
1 7.411398733944768e-17
6 9.566743725400238e-17
11 9.405168536642398e-17
16 1.155606851145806e-16
21 1.12596294050567e-16
26 2.886903244712338e-16
31 1.8005846227527388e-16
36 2.508406254601108e-16
41 1.1933186103074148e-16
46 1.2984108787006575e-16
51 1.5510471701776248e-16
56 1.651065114986088e-16
61 1.9158405029928644e-16
66 2.642285426100865e-16
71 1.420092868126218e-16
76 1.9511572577697645e-16
81 1.6068796037588115e-16
86 1.2862762295785318e-16
91 2.353864609251438e-16
96 1.5891628684846783e-16
101 2.949664099811648e-16
106 7.882152748618443e-16
111 2.439991685579141e-16
116 1.6886859732929689e-16
121 3.1610630855785473e-16
126 3.313897801527265e-16
131 1.4279803415170252e-16
136 1.1026880670774951e-15
141 1.7029496716591696e-16
146 2.828179173375001e-16
151 1.2912101284781152e-16
156 8.200404131911296e-16
161 2.095339319513201e-16
166 1.1609196285305707e-15
171 4.1109302166367983e-16
176 3.488644891996415e-16
181 1.0582960213127616e-15
186 9.216095182172143e-16
191 5.292115337991618e-16
196 6.756900539964706e-16
201 9.002714017434276e-16
206 2.719256833784446e-16
211 3.464056750690166e-16
216 4.933611186726295e-16
221 1.7362614547864971e-16
226 2.188042332862731e-16
231 2.878926294723087e-16
236 1.5322941511758835e-16
241 9.275163035867906e-16
246 1.637862323275758e-16
251 4.198437226676262e-16
256 7.866707032115673e-16
261 1.3855832552155695e-16
266 1.5496022028923533e-16
271 3.4171356130311093e-16
276 3.860754213324773e-16
281 2.567606470815721e-16
286 6.728172860537168e-16
291 4.99178900760484e-16
296 1.4777633198223245e-16
301 2.729694782544262e-16
306 4.181148073199653e-16
311 3.184048289598813e-16
316 1.8064254106343686e-16
321 3.1620729490230136e-16
326 2.1991628186969413e-16
331 8.53638115270093e-16
336 1.315573619572449e-16
341 2.25265926182533e-16
346 1.6799990067749762e-16
351 1.5460932204723569e-16
356 2.312565081491756e-16
361 4.952860694149966e-16
366 4.490793438796737e-16
371 2.3834733638370507e-16
376 1.2648340278039025e-16
381 2.6446666640597717e-16
386 4.174698152793896e-16
391 4.444573217479641e-16
396 3.72133270134649e-16
401 3.539908259848795e-16
406 1.795461878131119e-16
411 1.648985686688148e-16
416 1.5959496897976293e-16
421 3.2045974509954185e-16
426 5.522595113052906e-16
431 1.910085143147053e-16
436 1.4104901726810521e-16
441 1.2805456681861333e-16
446 1.2201711631789956e-16
451 3.5009871315013044e-16
456 2.4092969698965315e-16
461 1.5864770649903133e-16
466 1.9883215978055674e-16
471 3.4207586936121484e-16
476 2.2193562270814668e-16
481 3.38721911993854e-16
486 3.200047875549586e-16
491 1.6224669565112108e-16
496 1.3872890596103812e-16
501 2.049082067179232e-16
506 1.6521031666386085e-16
511 4.1333303152759976e-16
516 3.8515520749855886e-16
521 3.6127622316277506e-16
526 1.8541063984815741e-16
531 8.850263669494837e-16
536 2.820640751092553e-16
541 1.4891935183404888e-16
546 1.6110963764750425e-16
551 2.784946425435892e-16
556 2.0270455580581218e-16
561 2.2770749625600655e-16
566 6.57523619916472e-16
571 3.5656710685632763e-16
575 5.367327153902076e-16
};
\addlegendentry{$\scriptstyle{\sigma_{m + 2}}$}
\end{axis}

\end{tikzpicture}   
        \caption{Double tilt problem}
        \label{fig:Shifting_3_svd}
    \end{subfigure}
    \caption{Plots showing singular value trend over optimization steps for $m = 9$}
    \label{fig:theoretical_svd}
\end{figure}
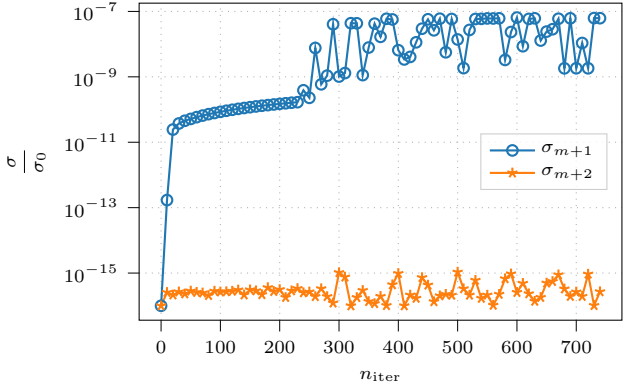
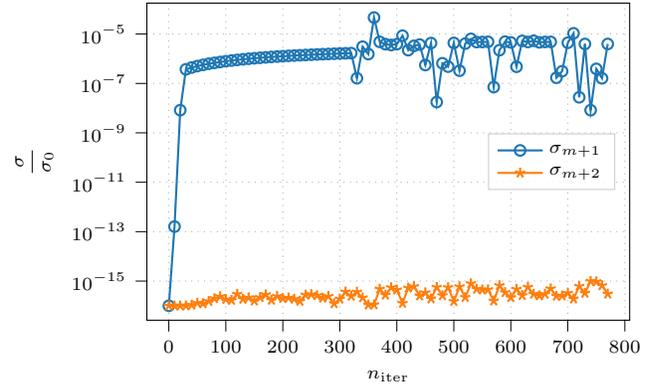

Next, we consider a standard configuration as described in \Cref{alg:FRTO_sPODG}, in which the basis is formed directly from the state snapshots as discussed in \Cref{ssec:Algorithmic_details}.
For comparison, we fix $m=41$ controls. 
Subsequently, the results for the single tilt problem are shown in \Cref{fig:J_shifting}.
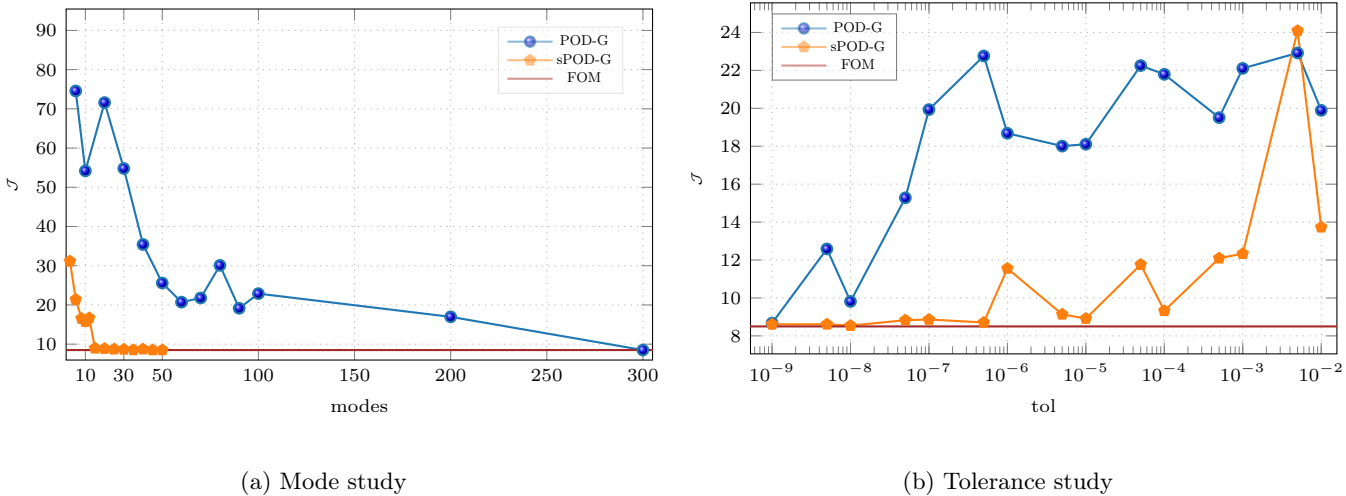
\begin{figure}[htp!]
    \centering
        \begin{subfigure}[t]{0.45\textwidth}
        \centering
        \setlength\figureheight{0.9\linewidth}
        \setlength\figurewidth{0.95\linewidth}
        \begin{tikzpicture}

  \definecolor{brown}{RGB}{165,42,42}
  \definecolor{darkgray176}{RGB}{176,176,176}
  \definecolor{green}{RGB}{0,128,0}
  \definecolor{lightgray204}{RGB}{204,204,204}
  \definecolor{darkorange25512714}{RGB}{255,127,14}
  \definecolor{forestgreen4416044}{RGB}{44,160,44}
  \definecolor{sienna}{RGB}{160,82,45}
  \definecolor{steelblue31119180}{RGB}{31,119,180}
    \definecolor{crimson2143940}{RGB}{214,39,40}

  \begin{axis}[
    height=0.9*\figureheight,
    width=1.2*\figurewidth,
    xmin=5.01187233627271e-08,   xmax=305,
    ymin=5.95019163314602,      ymax=45.3,
    xlabel={modes},
    ylabel={$\scriptstyle{\mathcal{J}}$},
    xmajorgrids,
    ymajorgrids,
    x grid style={darkgray176},
    y grid style={darkgray176},
    xtick={10,30,50,100,150,200,250,300,350,400,450,500},
    xticklabels={
      \(\scriptstyle {10}\),
      \(\scriptstyle {30}\),
      \(\scriptstyle {50}\),
      \(\scriptstyle {100}\),
      \(\scriptstyle {150}\),
      \(\scriptstyle {200}\),
      \(\scriptstyle {250}\),
      \(\scriptstyle {300}\),
      \(\displaystyle 350\),
      \(\displaystyle 400\),
      \(\displaystyle 450\),
      \(\displaystyle 500\),
    },
    ytick={0,10,20,30,40,50,60,70,80},
    yticklabels={
      \(\scriptstyle {0}\),
      \(\scriptstyle {10}\),
      \(\scriptstyle {20}\),
      \(\scriptstyle {30}\),
      \(\scriptstyle {40}\),
      \(\scriptstyle {50}\),
      \(\scriptstyle {60}\),
      \(\scriptstyle {70}\),
      \(\scriptstyle {80}\),
    },
    legend style={
    font=\footnotesize,
      fill opacity=0.5,
      draw opacity=0.5,
      text opacity=1,
      nodes={scale=0.6},
      draw=lightgray204,
      at={(0.95,0.95)}
    }
  ]

    \addplot+[
      thick,
      steelblue31119180,
      mark=ball,
      mark size=2,
      mark options={solid},
      nodes near coords,
      every node near coord/.append style={
        font=\tiny,
        inner sep=0pt,
        outer sep=0pt,
        yshift=2pt
      },
      point meta=explicit symbolic
    ]
    table[
      row sep=\\,
      x=x,
      y=y,
      meta=label
    ] {
      x           y                    label \\
    5   38.15    {} \\
    10   38.15    {} \\
    20   35.44    {} \\
    30   26.24    {} \\
    40   24.24    {} \\
    50   20.24   {} \\
    60   18.34    {} \\
    70   17.97    {} \\
    80   17.47    {} \\
    90   17.48    {} \\
    100   18.37    {} \\
    200   14.95    {} \\
    300   8.499    {} \\
    400   8.499    {} \\
    500   8.499    {} \\
    };
    \addlegendentry{\small POD-G}
    
    \addplot+[
      thick,
      darkorange25512714,
      mark=pentagon*,
      mark size=2,
      mark options={solid},
      nodes near coords,
      every node near coord/.append style={
        font=\tiny,
        inner sep=0pt,
        outer sep=0pt,
        yshift=2pt
      },
      point meta=explicit symbolic
    ]
    table[
      row sep=\\,
      x=x,
      y=y,
      meta=label
    ] {
      x           y                    label \\
    2   23.1705    {} \\
    5   17.3402    {} \\
    8   15.8240   {} \\
    10   10.9898    {} \\
    12   9.7632    {} \\
    15   9.1941    {} \\
    20   8.8381    {} \\
    25   8.7122    {} \\
    30   8.5777    {} \\
    35   8.5502    {} \\
    40   8.5562    {} \\
    45   8.5152    {} \\
    50   8.5149    {} \\
    };
    \addlegendentry{\small sPOD-G}

    \addplot[
    thick,
      mark=none,
      brown,
      samples=150
    ] coordinates {
      (5.01187233627271e-08,8.499)
      (510.0199526231496888,8.499)
    };
    \addlegendentry{\small FOM}

  \end{axis}
\end{tikzpicture}   
        \caption{Mode study}
        \label{fig:J_vs_modes_FRTO_AB_shifting}
    \end{subfigure}
    \hspace{0.01\textwidth}
    \begin{subfigure}[t]{0.45\textwidth}
        \centering
        \setlength\figureheight{0.9\linewidth}
        \setlength\figurewidth{0.95\linewidth}
        \begin{tikzpicture}

  \definecolor{brown}{RGB}{165,42,42}
  \definecolor{darkgray176}{RGB}{176,176,176}
  \definecolor{green}{RGB}{0,128,0}
  \definecolor{lightgray204}{RGB}{204,204,204}

  \definecolor{crimson2143940}{RGB}{214,39,40}
  \definecolor{darkorange25512714}{RGB}{255,127,14}
  \definecolor{forestgreen4416044}{RGB}{44,160,44}
  \definecolor{sienna}{RGB}{160,82,45}
  \definecolor{steelblue31119180}{RGB}{31,119,180}

  \begin{axis}[
    height=0.9*\figureheight,
    width=1.2*\figurewidth,
    xmode=log,
    xmin=5.30957344480193e-10, xmax=0.0158489319246111,
    ymin=7.052039401636526,    ymax=20.6,
    xlabel={tol},
    ylabel={$\scriptstyle{\mathcal{J}}$},
    xmajorgrids,
    ymajorgrids,
    x grid style={darkgray176},
    y grid style={darkgray176},
    xtick={1e-10,1e-09,1e-08,1e-07,1e-06,1e-05,1e-04,1e-03,1e-02,1e-01,1},
    xticklabels={
      \(\scriptstyle 10^{-10}\),
      \(\scriptstyle 10^{-9}\),
      \(\scriptstyle 10^{-8}\),
      \(\scriptstyle 10^{-7}\),
      \(\scriptstyle 10^{-6}\),
      \(\scriptstyle 10^{-5}\),
      \(\scriptstyle 10^{-4}\),
      \(\scriptstyle 10^{-3}\),
      \(\scriptstyle 10^{-2}\),
      \(\scriptstyle 10^{-1}\),
      \(\scriptstyle 10^{0}\)
    },
    ytick={0,4,8,12,16,20,24,28,32,36},
    yticklabels={
      \(\scriptstyle {0}\),
      \(\scriptstyle {4}\),
      \(\scriptstyle {8}\),
      \(\scriptstyle {12}\),
      \(\scriptstyle {16}\),
      \(\scriptstyle {20}\),
      \(\scriptstyle {24}\),
      \(\scriptstyle {28}\),
      \(\scriptstyle {32}\),
      \(\scriptstyle {36}\),
    },
    legend style={
    font=\footnotesize,
      fill opacity=0.5,
      draw opacity=0.2,
      text opacity=1,
      nodes={scale=0.6},
      at={(0.35,0.97)},
      anchor=north east
    }
  ]

    \addplot+[
      thick,
      steelblue31119180,
      mark=ball,
      mark size=2,
      mark options={solid},
      nodes near coords,
      every node near coord/.append style={
        font=\tiny,
        inner sep=0pt,
        outer sep=0pt,
        yshift=-7pt
      },
      point meta=explicit symbolic
    ]
    table[
      row sep=\\,
      x=x,
      y=y,
      meta=label
    ] {
      x           y                    label \\
    1e-09   8.57    {} \\
    5e-09   9.70    {} \\
    1e-08   10.24      {} \\
    5e-08   14.01    {} \\
    1e-07   14.41      {} \\
    5e-07   14.70      {} \\
    1e-06   15.37    {} \\
    5e-06   15.84    {} \\
    1e-05   16.19    {} \\
    5e-05   15.82      {} \\
    0.0001  16.23    {} \\
    0.0005  16.93      {} \\
    0.001   17.65      {} \\
    0.005   16.90      {} \\
    0.01    17.22     {} \\
    };
    \addlegendentry{\small POD-G}

    \addplot+[
      thick,
      darkorange25512714,
      mark=pentagon*,
      mark size=2,
      mark options={solid},
      nodes near coords,
      every node near coord/.append style={
        font=\tiny,
        inner sep=0pt,
        outer sep=0pt,
        yshift=2pt
      },
      point meta=explicit symbolic
    ]
    table[
      row sep=\\,
      x=x,
      y=y,
      meta=label
    ] {
      x         y                   label \\
    1e-09   8.5262    {} \\
    5e-09   8.5515    {} \\
    1e-08   8.5446      {} \\
    5e-08  8.7100    {} \\
    1e-07  8.6863      {} \\
    5e-07   8.8001      {} \\
    1e-06   8.8385    {} \\
    5e-06   9.4063    {} \\
    1e-05   8.9923    {} \\
    5e-05   9.2229      {} \\
    0.0001  8.7687    {} \\
    0.0005  9.4855      {} \\
    0.001   9.5083      {} \\
    0.005   15.0298     {} \\
    0.01    13.4044     {} \\
    };
    \addlegendentry{\small sPOD-G}

    \addplot[
    thick,
      mark=none,
      brown,
      samples=150
    ] coordinates {
      (5.01187233627271e-11, 8.499)
      (0.0199526231496888,   8.499)
    };
    \addlegendentry{\small FOM}

  \end{axis}
\end{tikzpicture}   
        \caption{Tolerance study}
        \label{fig:J_vs_tol_FRTO_AB}
    \end{subfigure}
    \caption{Plots for $\mathcal{J}$ behavior for the single tilt problem}
    \label{fig:J_shifting}
\end{figure}
\Cref{fig:J_vs_modes_FRTO_AB_shifting} shows the mode study, and the \Cref{fig:J_vs_tol_FRTO_AB} shows the tolerance study.
\Cref{fig:J_vs_modes_FRTO_AB_shifting} shows that the sPOD-G method requires close to $30$ modes to reach near the FOM cost functional, whereas the POD-G method requires close to $300$ modes. 
From \Cref{fig:J_vs_tol_FRTO_AB}, we observe that as the tolerance decreases, both curves show  an approaching trend towards the FOM cost functional.
The average number of modes $n_\mathrm{avg}$ required by the methods per optimization step for the tolerance study is shown in \Cref{tab:tol_FRTO}.
\begin{table}[htp]
\centering
\caption{$n_\mathrm{avg}$ values for tolerance study}
\label{tab:tol_FRTO}
\begin{tabular}{l|cc||cc}
\toprule
 & \multicolumn{2}{c||}{Single tilt problem} & \multicolumn{2}{c}{Double tilt problem} \\
\midrule
$\mathrm{tol}$ & POD-G & sPOD-G & POD-G & sPOD-G \\
\midrule
$10^{-2}$         & $103$ & $6$  & $100$ & $6$  \\
$10^{-3}$         & $127$ & $10$ & $124$ & $16$ \\
$10^{-4}$         & $148$ & $13$ & $145$ & $23$ \\
$10^{-5}$         & $166$ & $15$ & $165$ & $29$ \\
$10^{-6}$         & $184$ & $18$ & $184$ & $35$ \\
$10^{-7}$         & $200$ & $21$ & $200$ & $39$ \\
$10^{-8}$         & $215$ & $28$ & $216$ & $40$ \\
$10^{-9}$         & $229$ & $32$ & $230$ & $40$ \\
\bottomrule
\end{tabular}
\end{table}

For the double tilt problem, the results are shown in \Cref{fig:J_shifting_3}.
\begin{figure}[htp!]
    \centering
    \begin{subfigure}[t]{0.45\textwidth}
        \centering
        \setlength\figureheight{0.9\linewidth}
        \setlength\figurewidth{0.95\linewidth}
        \begin{tikzpicture}

  \definecolor{brown}{RGB}{165,42,42}
  \definecolor{darkgray176}{RGB}{176,176,176}
  \definecolor{green}{RGB}{0,128,0}
  \definecolor{lightgray204}{RGB}{204,204,204}
  \definecolor{darkorange25512714}{RGB}{255,127,14}
  \definecolor{forestgreen4416044}{RGB}{44,160,44}
  \definecolor{sienna}{RGB}{160,82,45}
  \definecolor{steelblue31119180}{RGB}{31,119,180}
    \definecolor{crimson2143940}{RGB}{214,39,40}

  \begin{axis}[
    height=0.9*\figureheight,
    width=1.2*\figurewidth,
    xmin=5.01187233627271e-08,   xmax=305,
    ymin=15.95019163314602,      ymax=120.452607734195,
    xlabel={modes},
    ylabel={$\scriptstyle{\mathcal{J}}$},
    xmajorgrids,
    ymajorgrids,
    x grid style={darkgray176},
    y grid style={darkgray176},
    xtick={10,30,50,100,150,200,250,300,350,400,450,500},
    xticklabels={
      \(\scriptstyle {10}\),
      \(\scriptstyle {30}\),
      \(\scriptstyle {50}\),
      \(\scriptstyle {100}\),
      \(\scriptstyle {150}\),
      \(\scriptstyle {200}\),
      \(\scriptstyle {250}\),
      \(\scriptstyle {300}\),
      \(\displaystyle 350\),
      \(\displaystyle 400\),
      \(\displaystyle 450\),
      \(\displaystyle 500\),
    },
    ytick={0,20,40,60,80,100,120,140},
    yticklabels={
      \(\scriptstyle {0}\),
      \(\scriptstyle {20}\),
      \(\scriptstyle{40}\),
      \(\scriptstyle {60}\),
      \(\scriptstyle {80}\),
      \(\scriptstyle{100}\),
      \(\scriptstyle {120}\),
      \(\scriptstyle {140}\),
    },
    legend style={
    font=\footnotesize,
      fill opacity=0.5,
      draw opacity=0.5,
      text opacity=1,
      nodes={scale=0.6},
      draw=lightgray204,
      at={(0.95,0.95)}
    }
  ]

    \addplot+[
      thick,
      steelblue31119180,
      mark=ball,
      mark size=2,
      mark options={solid},
      nodes near coords,
      every node near coord/.append style={
        font=\tiny,
        inner sep=0pt,
        outer sep=0pt,
        yshift=2pt
      },
      point meta=explicit symbolic
    ]
    table[
      row sep=\\,
      x=x,
      y=y,
      meta=label
    ] {
      x           y                    label \\
    5   112.68    {} \\
    10   85.71    {} \\
    20   71.22    {} \\
    30   68.05    {} \\
    40   66.54    {} \\
    50   61.28   {} \\
    60   53.95    {} \\
    70   52.05    {} \\
    80   51.82    {} \\
    90   53.20    {} \\
    100   52.09    {} \\
    200   36.48    {} \\
    300   25.42    {} \\
    400   25.42    {} \\
    500   25.42    {} \\
    };
    \addlegendentry{\small POD-G}

    \addplot+[
      thick,
      darkorange25512714,
      mark=pentagon*,
      mark size=2,
      mark options={solid},
      nodes near coords,
      every node near coord/.append style={
        font=\tiny,
        inner sep=0pt,
        outer sep=0pt,
        yshift=2pt
      },
      point meta=explicit symbolic
    ]
    table[
      row sep=\\,
      x=x,
      y=y,
      meta=label
    ] {
      x           y                    label \\
    2   72.2594    {} \\
    5   68.6325    {} \\
    8   65.1985   {} \\
    10   63.0929    {} \\
    12   60.4079    {} \\
    15   56.3457    {} \\
    20   45.8322    {} \\
    25   36.4309    {} \\
    30   33.7874    {} \\
    35   27.8157    {} \\
    40   25.8661    {} \\
    45   25.5189    {} \\
    50   25.4916    {} \\
    };
    \addlegendentry{\small sPOD-G}

    \addplot[
    thick,
      mark=none,
      brown,
      samples=150
    ] coordinates {
      (5.01187233627271e-08,25.42)
      (510.0199526231496888,25.42)
    };
    \addlegendentry{\small FOM}

  \end{axis}
\end{tikzpicture}   
        \caption{Mode study}
        \label{fig:J_vs_modes_FRTO_AB_CS}
    \end{subfigure}
    \hspace{0.015\textwidth}
    \begin{subfigure}[t]{0.45\textwidth}
        \centering
        \setlength\figureheight{0.9\linewidth}
        \setlength\figurewidth{0.95\linewidth}
        \begin{tikzpicture}

  \definecolor{brown}{RGB}{165,42,42}
  \definecolor{darkgray176}{RGB}{176,176,176}
  \definecolor{green}{RGB}{0,128,0}
  \definecolor{lightgray204}{RGB}{204,204,204}

  \definecolor{crimson2143940}{RGB}{214,39,40}
  \definecolor{darkorange25512714}{RGB}{255,127,14}
  \definecolor{forestgreen4416044}{RGB}{44,160,44}
  \definecolor{sienna}{RGB}{160,82,45}
  \definecolor{steelblue31119180}{RGB}{31,119,180}

  \begin{axis}[
    height=0.9*\figureheight,
    width=1.2*\figurewidth,
    xmode=log,
    xmin=5.30957344480193e-10, xmax=0.0158489319246111,
    ymin=23.052039401636526,    ymax=70.5,
    xlabel={tol},
    ylabel={$\scriptstyle{\mathcal{J}}$},
    xmajorgrids,
    ymajorgrids,
    x grid style={darkgray176},
    y grid style={darkgray176},
    xtick={1e-10,1e-09,1e-08,1e-07,1e-06,1e-05,1e-04,1e-03,1e-02,1e-01,1},
    xticklabels={
      \(\scriptstyle 10^{-10}\),
      \(\scriptstyle 10^{-9}\),
      \(\scriptstyle 10^{-8}\),
      \(\scriptstyle 10^{-7}\),
      \(\scriptstyle 10^{-6}\),
      \(\scriptstyle 10^{-5}\),
      \(\scriptstyle 10^{-4}\),
      \(\scriptstyle 10^{-3}\),
      \(\scriptstyle 10^{-2}\),
      \(\scriptstyle 10^{-1}\),
      \(\scriptstyle 10^{0}\),
    },
    ytick={0,20,40,60,80,100,120,140,160},
    yticklabels={
      \(\scriptstyle {0}\),
      \(\scriptstyle {20}\),
      \(\scriptstyle {40}\),
      \(\scriptstyle {60}\),
      \(\scriptstyle {80}\),
      \(\scriptstyle {100}\),
      \(\scriptstyle {120}\),
      \(\scriptstyle {140}\),
      \(\scriptstyle {160}\),
    },
    legend style={
    font=\footnotesize,
      fill opacity=0.5,
      draw opacity=0.2,
      text opacity=1,
      nodes={scale=0.6},
      at={(0.35,0.97)},
      anchor=north east
    }
  ]

    \addplot+[
      thick,
      steelblue31119180,
      mark=ball,
      mark size=2,
      mark options={solid},
      nodes near coords,
      every node near coord/.append style={
        font=\tiny,
        inner sep=0pt,
        outer sep=0pt,
        yshift=-7pt
      },
      point meta=explicit symbolic
    ]
    table[
      row sep=\\,
      x=x,
      y=y,
      meta=label
    ] {
      x           y                    label \\
    1e-09   27.90    {} \\
    5e-09   29.98    {} \\
    1e-08   30.07      {} \\
    5e-08   33.70    {} \\
    1e-07   33.73      {} \\
    5e-07   38.53      {} \\
    1e-06   37.73    {} \\
    5e-06   42.35    {} \\
    1e-05   42.90    {} \\
    5e-05   48.01      {} \\
    0.0001  45.71    {} \\
    0.0005  48.19      {} \\
    0.001   51.07      {} \\
    0.005   50.73      {} \\
    0.01    50.26     {} \\
    };
    \addlegendentry{\small POD-G}

    \addplot+[
      thick,
      darkorange25512714,
      mark=pentagon*,
      mark size=2,
      mark options={solid},
      nodes near coords,
      every node near coord/.append style={
        font=\tiny,
        inner sep=0pt,
        outer sep=0pt,
        yshift=2pt
      },
      point meta=explicit symbolic
    ]
    table[
      row sep=\\,
      x=x,
      y=y,
      meta=label
    ] {
      x         y                   label \\
    1e-09   25.4829    {} \\
    5e-09   25.4827    {} \\
    1e-08   25.4835      {} \\
    5e-08  25.4864    {} \\
    1e-07  25.4887      {} \\
    5e-07   25.6432      {} \\
    1e-06   25.8314    {} \\
    5e-06   26.6821    {} \\
    1e-05   28.5058    {} \\
    5e-05   28.0130      {} \\
    0.0001  31.8464    {} \\
    0.0005  43.1669      {} \\
    0.001   40.8415      {} \\
    0.005   62.8533     {} \\
    0.01    63.7592     {} \\
    };
    \addlegendentry{\small sPOD-G}

    \addplot[
    thick,
      mark=none,
      brown,
      samples=150
    ] coordinates {
      (5.01187233627271e-11, 25.42)
      (0.0199526231496888,   25.42)
    };
    \addlegendentry{\small FOM}

  \end{axis}
\end{tikzpicture}   
        \caption{Tolerance study}
        \label{fig:J_vs_tol_FRTO_AB_CS}
    \end{subfigure}
    \caption{Plots for $\mathcal{J}$ behavior for the double tilt problem}
    \label{fig:J_shifting_3}
\end{figure}
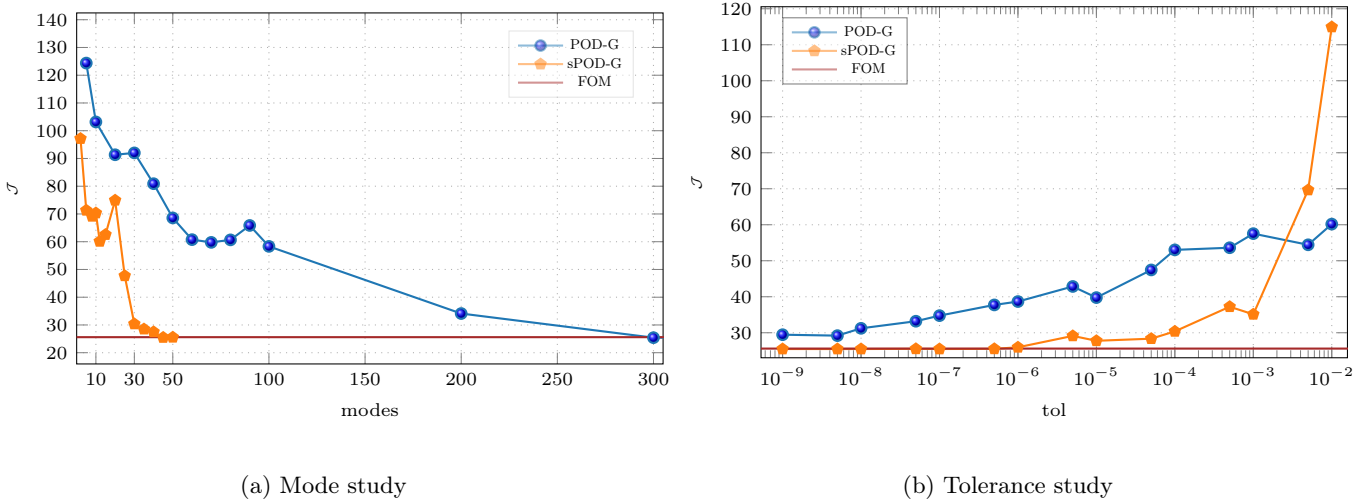
We observe that for the mode study, the sPOD-G method requires close to $45$ modes to reach the FOM cost functional, whereas the POD-G method requires close to $300$ modes to reach the FOM cost functional.
The average mode numbers throughout the iterations are listed in \Cref{tab:tol_FRTO}.
In the following we summarize the crucial observations from both test cases.
\begin{itemize}
    \item We observe that the sPOD-G uses roughly $6-17$ times fewer modes than the POD-G. 
    While further reduction could be achieved by updating the shifts regularly, intermediate snapshots develop smeared (diffused) features that make the resulting shift estimates ambiguous.
    This problem thus becomes non-trivial and for this reason, we do not pursue shift-updating in the present study.
    \item The minimum number of modes required for sPOD-G to reach the FOM cost in the mode study roughly matches the $n_\mathrm{avg}$ obtained in the tolerance study at the strictest tolerance.
    For the POD-G, the $n_{\mathrm{avg}}$ values at the strictest tolerance fall short of the actual number of modes required in the mode study to converge, therefore, the POD-G tolerance curves, do not reach the FOM cost in the considered range of tolerance values.
    \item The difference between $n_\mathrm{avg}$ (listed in \Cref{tab:tol_FRTO}) and the modes in the mode study is expected.
    The mode numbers given in the mode study are exact and prescribed by the user upfront, while $n_\mathrm{avg}$ are average mode numbers per iteration that are automatically selected by the tolerance criterion and are rough estimates.
    Thus, they should not be confused with each other.
\end{itemize}

\begin{rmrk}\label{rmrk:commonbasis_statebasis}
    In our numerical tests, for POD-G, computing a common basis with both state and adjoint snapshots accelerated convergence.
    Only $\sim$100 modes were needed versus 300 for a state-only basis shown here. 
    This is consistent with \cite{negri_reduced_2015}, where a purely state-based basis was shown to yield an inconsistent reduced adjoint equation, since full-order adjoint information never enters the basis construction. 
    This benefit, however, is established only for linear ROMs and is not guaranteed for a nonlinear ROM such as the sPOD-G ROM.
    For the sPOD-G, we observed that the mass matrix becomes ill-conditioned when a common basis was used, which leads to a large approximation error.
    Therefore, we only use state snapshots for the basis construction in the numerical experiments reported in this section.
\end{rmrk}

\subsection{Timing analysis}
So far, we have examined the performance differences between the POD-G and sPOD-G methods in terms of reduced-order dimensions and tolerance values. 
Let us also focus on a comparison between the two techniques based on their computational times.
\Cref{fig:J_vs_runtime_full} illustrates a study of the time it takes for computations when using reduced-order models for both example problems. 
For each of the ROMs, the plotted curve is a lower envelope over every configuration in the mode and tolerance studies shown in \Cref{sec:using_roms_optimal_control}. 
Each individual run contributes its own cost functional trajectory against a cumulative wall-clock time, which is linearly interpolated onto a common time grid and held constant beyond its own runtime.
At each time $t$, the envelope reports the minimum cost functional value attained by any configuration of that method up to $t$. 
It therefore represents the best cost achievable by that method with a given computational time. 
Note that, being a lower envelope, the curves shown for the ROMs need not coincide with any of the trajectories of the individual configurations.
\begin{figure}[htp!]
    \centering
    \begin{subfigure}{0.48\textwidth}
        \centering
        \setlength\figureheight{0.8\linewidth}
        \setlength\figurewidth{0.95\linewidth}
        \input{Jerr_vs_runtime_Shifting}
        \caption{Single tilt problem}
        \label{fig:J_vs_runtime_S}
    \end{subfigure}
    \hspace{0.001\textwidth}
    \begin{subfigure}{0.48\textwidth}
        \centering
        \setlength\figureheight{0.8\linewidth}
        \setlength\figurewidth{0.95\linewidth}
        \input{Jerr_vs_runtime_Shifting_3}
        \caption{Double tilt problem}
        \label{fig:J_vs_runtime_S_3}
    \end{subfigure}
    \caption{Plots of the relative drop in the obtained cost functional vs. run time for both examples}
    \label{fig:J_vs_runtime_full}
\end{figure} 
Amongst the two ROMs, POD-G is in general faster than the sPOD-G.
For a more detailed comparison of the computational time for the main steps, we refer to \Cref{tab:timing_split_st,tab:timing_split_dt}.
\begin{table}[htp]
  \centering
  \caption{Computational time ($\mathrm{s}$) for crucial steps (Single tilt problem).
  Times are accumulated up to the first iteration within $1\%$ of $\mathcal{J}_\mathrm{FOM}$.
  The per-call cost is shown in parentheses.\tablefootnote{Per-call costs are the average time per iteration (per refine step for the basis row) up to that same point.}
Dominant costs for each method are bold-faced.}
  \label{tab:timing_split_st}
  \begin{tabular}{l|ccc}
    \toprule
    Computational Steps & POD-G ($\ell = 300$) & sPOD-G ($\tilde{\ell} = 40$) & FOM \\
    \midrule
    FOM state solve + Basis constr. & \shortstack{$\mathbf{55.17}$\\ \footnotesize$(\mathbf{2.3985\text{/refine})}$} & \shortstack{$\mathbf{204.55}$\\ \footnotesize$(\mathbf{3.4092\text{/refine})}$} & \shortstack{$0.00$\\ \footnotesize$(0.00)$} \\[1.2em]
    ROM/FOM state solve             & \shortstack{$6.08$\\ \footnotesize$(0.0547)$}  & \shortstack{$11.11$\\ \footnotesize$(0.0373)$}  & \shortstack{$22.14$\\ \footnotesize$(0.1960)$} \\[1.2em]
    Compute $\mathcal{J}$           & \shortstack{$1.21$\\ \footnotesize$(0.0109)$}  & \shortstack{$0.99$\\ \footnotesize$(0.0033)$} & \shortstack{$4.55$\\ \footnotesize$(0.0403)$} \\[1.2em]
    ROM/FOM adjoint solve           & \shortstack{$5.32$\\ \footnotesize$(0.0480)$}  & \shortstack{$13.63$\\ \footnotesize$(0.0457)$}  & \shortstack{$18.77$\\ \footnotesize$(0.1661)$} \\[1.2em]
    Compute gradient                & \shortstack{$1.22$\\ \footnotesize$(0.0110)$}  & \shortstack{$3.63$\\ \footnotesize$(0.0122)$}  & \shortstack{$4.27$\\ \footnotesize$(0.0378)$} \\[1.2em]
    Update control                  & \shortstack{$10.30$\\ \footnotesize$(0.0928)$}  & \shortstack{$13.51$\\ \footnotesize$(0.0453)$} & \shortstack{$\mathbf{38.65}$\\ \footnotesize$\mathbf{(0.3420)}$} \\[1.2em]
    \midrule
    Total (time to within $1\%$ of $\mathcal{J}_\mathrm{FOM}$) & $79.31$  & $247.42$ & $88.37$ \\
    $n_\mathrm{iter}$ (to within $1\%$) & $111$  & $298$  & $113$ \\
    Num. of $\mathrm{refine}$ steps (to within $1\%$) & $23$  & $60$  & -- \\
    \bottomrule
  \end{tabular}
\end{table}
\begin{table}[htp]
  \centering
  \caption{Computational time ($\mathrm{s}$) for crucial steps (Double tilt problem).
  Times are accumulated up to the first iteration within $1\%$ of $\mathcal{J}_\mathrm{FOM}$.
  The per-call cost is shown in parentheses.
   Dominant costs for each method are bold-faced.}
  \label{tab:timing_split_dt}
  \begin{tabular}{l|ccc}
    \toprule
    Computational Steps & POD-G ($\ell = 300$) & sPOD-G ($\tilde{\ell} = 45$) & FOM \\
    \midrule
    FOM state solve + Basis constr. & \shortstack{$\mathbf{146.47}$\\ \footnotesize$\mathbf{(2.4012\text{/refine})}$} & \shortstack{$\mathbf{246.19}$\\ \footnotesize$\mathbf{(3.7302\text{/refine})}$} & \shortstack{$0.00$\\ \footnotesize$(0.00)$} \\[1.2em]
    ROM/FOM state solve             & \shortstack{$14.90$\\ \footnotesize$(0.0489)$}  & \shortstack{$12.77$\\ \footnotesize$(0.0388)$}  & \shortstack{$47.00$\\ \footnotesize$(0.1722)$} \\[1.2em]
    Compute $\mathcal{J}$           & \shortstack{$2.23$\\ \footnotesize$(0.0073)$}  & \shortstack{$1.05$\\ \footnotesize$(0.0032)$}  & \shortstack{$10.48$\\ \footnotesize$(0.0384)$} \\[1.2em]
    ROM/FOM adjoint solve           & \shortstack{$13.45$\\ \footnotesize$(0.0441)$}  & \shortstack{$15.82$\\ \footnotesize$(0.0481)$}  & \shortstack{$39.48$\\ \footnotesize$(0.1446)$} \\[1.2em]
    Compute gradient                & \shortstack{$1.86$\\ \footnotesize$(0.0061)$}   & \shortstack{$4.00$\\ \footnotesize$(0.0121)$}   & \shortstack{$7.90$\\ \footnotesize$(0.0289)$} \\[1.2em]
    Update control                  & \shortstack{$26.70$\\ \footnotesize$(0.0875)$}  & \shortstack{$16.83$\\ \footnotesize$(0.0512)$}  & \shortstack{$\mathbf{74.19}$\\ \footnotesize$\mathbf{(0.2718)}$} \\[1.2em]
    \midrule
    Total (time to within $1\%$ of $\mathcal{J}_\mathrm{FOM}$) & $205.62$  & $296.66$ & $179.05$ \\
    $n_\mathrm{iter}$ (to within $1\%$) & $305$  & $329$  & $273$ \\
    Num. of $\mathrm{refine}$ steps (to within $1\%$) & $61$  & $66$  & -- \\
    \bottomrule
  \end{tabular}
\end{table}

\begin{figure}[htp!]
    \centering
    \begin{subfigure}{0.48\textwidth}
        \centering
        \setlength\figureheight{0.8\linewidth}
        \setlength\figurewidth{0.95\linewidth}
        \input{Jerr_vs_runtime_single_Shifting}
        \caption{Single tilt problem}
        \label{fig:Jerr_vs_runtime_single_S}
    \end{subfigure}
    \hspace{0.001\textwidth}
    \begin{subfigure}{0.48\textwidth}
        \centering
        \setlength\figureheight{0.8\linewidth}
        \setlength\figurewidth{0.95\linewidth}
        \input{Jerr_vs_runtime_single_Shifting_3}
        \caption{Double tilt problem}
        \label{fig:Jerr_vs_runtime_single_S_3}
    \end{subfigure}
    \caption{Plots of the relative drop in the obtained cost functional vs. run time for mode numbers mentioned in \Cref{tab:timing_split_st,tab:timing_split_dt}. The markers for each curve indicate the termination of the run.}
    \label{fig:Jerr_vs_runtime_single_full}
\end{figure}

The tables report timings for selected mode numbers and suffice to illustrate the discrepancies shown in \Cref{fig:J_vs_runtime_full}. 
The FOM, especially in the double tilt problem is the fast for two main reasons, it does not incur the overhead of constructing a reduced basis, and its coefficient matrices are sparse which allows to use efficient sparse-matrix algebra to solve the state and adjoint equations.
Looking at the computational steps, for both the ROMs, the FOM state solve and the subsequent basis construction takes a large part of the computational time. 
For the POD-G, this amounts to applying SVD on the FOM state snapshots and  precomputing the reduced matrices. 
However, for the sPOD-G, this step particularly takes more time, since it requires to build shift-dependent coefficient matrices for each sampled shift value. 
Nevertheless, both the reduced state and adjoint equations for the ROMs are faster than the FOM state and adjoint equations, despite the FOM exploiting the sparse algebra. 
Similar to \Cref{fig:J_vs_runtime_full}, we also plot the runtime plots for the mode values mentioned in \Cref{tab:timing_split_st,tab:timing_split_dt}, see \Cref{fig:Jerr_vs_runtime_single_full}.

\begin{table}[htp]
  \centering
  \caption{Per-iteration computational complexity for each step.
  We assume $m, \ell, \tilde{\ell} \ll n, n_t$, $n_\mathrm{bt}$ is the number of non-monotone line search trials and an explicit Euler time integration method is used. 
  The basis construction uses a randomized SVD \cite{halko_random_2011} with a fixed number of components, whose cost $\mathcal{O}(n_t n \ell)$ or $\mathcal{O}(n_t n \tilde{\ell})$ is included in the first row.
  }
  \label{tab:complexity}
  \begin{tabular}{l|ccc}
    \toprule
    Computational Steps & POD-G & sPOD-G & FOM \\
    \midrule
    FOM state solve + Basis constr. & $\mathcal{O}(n_t n(m+\ell))$ & $\mathcal{O}(n_t n(m+\tilde\ell)+n_\mathrm{s}n\tilde\ell m)$ & -- \\
    ROM/FOM state solve             & $\mathcal{O}(n_t\ell(\ell+m))$ & $\mathcal{O}(n_t\tilde\ell(\tilde\ell+m))$ & $\mathcal{O}(n_t n m)$ \\
    Compute $\mathcal{J}$           & $\mathcal{O}(n_t\ell)$ & $\mathcal{O}(n_t\tilde\ell)$ & $\mathcal{O}(n_t n)$ \\
    ROM/FOM adjoint solve           & $\mathcal{O}(n_t\ell^2)$ & $\mathcal{O}(n_t\tilde\ell(\tilde\ell+m))$ & $\mathcal{O}(n_t n)$ \\
    Compute gradient                & $\mathcal{O}(n_t\ell m)$ & $\mathcal{O}(n_t\tilde\ell m)$ & $\mathcal{O}(n_t n m)$ \\
    Update control                  & $\mathcal{O}(n_\mathrm{bt}n_t\ell(\ell+m))$ & $\mathcal{O}(n_\mathrm{bt}n_t\tilde\ell(\tilde\ell+m))$ & $\mathcal{O}(n_\mathrm{bt}n_t n m)$ \\
    \midrule
    \textbf{Total (per iteration)}  & $\mathcal{O}(n_\mathrm{bt}n_t\ell(\ell+m))$ & $\mathcal{O}(n_\mathrm{bt}n_t\tilde\ell(\tilde\ell+m))$ & $\mathcal{O}(n_\mathrm{bt}n_t n m)$ \\
    \bottomrule
  \end{tabular}
\end{table}
We report each step's $\mathcal{O}$-notation complexity directly in \Cref{tab:complexity}.
Wherever a computational step includes several distinct contributions, we report the entire sum unless there is a clear dominant cost which allows to neglect the other contributions.
The FOM's per-iteration cost is $\mathcal{O}(n_\mathrm{bt}n_tnm)$, dominated by the $n_\mathrm{bt}$ repeated state solves inside the line search, matching the largest measured entry. 
The basis construction is the only step where the POD-G and sPOD-G methods scale with the FOM dimension, and this step is not performed in every iteration, but only when the basis is refined.
As for the reduced state and adjoint solves, POD-G requires no linear solve since its mass matrix is the identity, so each time step costs only the matrix-vector products, i.e., $\mathcal{O}(\ell(\ell+m))$.
In sPOD-G, the mass matrix is an arrowhead matrix, whose system can be solved by block elimination with a Schur complement in $\mathcal{O}(\tilde\ell)$ operations.
This results in $\mathcal{O}(\tilde{\ell})$ at each time step which is not the dominant cost of the state or the adjoint solve, instead the matrix vector products dominate the cost.

\section{Conclusion and outlook}\label{sec:conclusion}
In this paper, we investigated an open-loop optimal control problem using ROMs, with a particular focus on the tPOD-Galerkin approach.  
The well-posedness and optimality results apply to strongly continuous transformation groups with an explicit a priori control threshold.
The numerical experiments for the linear advection equation yield first promising results, which motivates applying this framework to more challenging problems.
Nonetheless, several difficulties persist, as highlighted by the computational timing results.  
In principle, one could also aim to rigorously formalize the adaptive basis refinement procedure in the spirit of OSPOD and TRPOD.  
Furthermore, applying the sPOD-G method to optimal control problems for more complex PDE systems, for instance a wildland fire model \cite{burela_parametric_2023}, could yield additional insights into the practical relevance and applicability of this approach.

\section*{Data Availability Statement}

To enhance the reproducibility and transparency of the research in this paper, the source code for the experiments and analyzes has been made publicly available via the following Zenodo repository: \\
\begin{center}
\urlstyle{tt}
\url{https://doi.org/10.5281/zenodo.19185335}
\end{center}

\vspace{1em}
We encourage researchers to utilize and build upon the code for their own research purposes.

\begin{samepage}
\section*{Declaration of generative AI use}
The original manuscript was written by the authors. For the revision, the authors proposed extending the analysis to a more general group setting. OpenAI's ChatGPT assisted with developing and checking proofs, revising the mathematical exposition and optimizing the code.
Google's Gemini was used in preliminary discussions of possible generalizations. The authors are responsible for the correctness, originality, and final content of the manuscript.\par
\end{samepage}

\section*{Acknowledgement}
T.B.~and S.B. gratefully acknowledge the support of the Deutsche Forschungsgemeinschaft (DFG) as part of GRK2433 DAEDALUS (DFG Project number: 384950143) as well as the financial support from the SFB TRR154 (DFG project number: 239904186) under the sub-project B03.
P.S. thanks the Deutsche Forschungsgemeinschaft for their support within the SFB 1294 ``Data Assimilation – The Seamless Integration of Data and Models'' (Project 318763901).
We also thank Philipp Krah for his insightful comments and feedback and for providing access to the sPOD codebase.

\addcontentsline{toc}{section}{Conflict Of Interest (COI)}
\section*{Conflict Of Interest (COI)}
All authors declare that they have no conflicts of interest.

\bibliographystyle{plain}
\bibliography{abbr,refs}

\end{document}